\documentclass[reqno, 11pt]{amsart}
\usepackage{amssymb}
\usepackage[colorlinks = true,
            linkcolor = blue,
            urlcolor  = blue,
            citecolor = blue,
            anchorcolor = blue]{hyperref}
\usepackage{caption}
\newtheorem{theorem}{Theorem}[section]

\newtheorem{lemma}[theorem]{Lemma}
\newtheorem{corollary}[theorem]{Corollary}
\newtheorem{claim}{Claim}
\theoremstyle{definition}

\newtheorem{example}[theorem]{Example}
\newtheorem{step}{Step}

\usepackage{colortbl}
\usepackage{enumitem}
\usepackage{anysize}
\marginsize{3 cm}{3 cm}{3 cm}{3 cm}
\theoremstyle{remark}
\newtheorem{remark}[theorem]{Remark}
\DeclareMathOperator{\diam}{diam}

\DeclareMathOperator{\Var}{\mathbf {Var}}

\DeclareMathOperator{\Li}{Li}
\newcommand{\R}{\mathbb R} 
\newcommand{\E}{\mathbb{E}} 
\newcommand{\Prob}{\mathbb{P}}

\newcommand\etal{et~al.~}

\newcommand\acks{\section*{Acknowledgements}}
\numberwithin{equation}{section}
\usepackage{graphicx}
\graphicspath{ {./images/} }
\usepackage{amsrefs}
\def\MR#1{\href{https://mathscinet.ams.org/mathscinet-getitem?mr=#1}{MR#1}}

\begin{document}

\title{The number of real zeros of elliptic polynomials}      
  
\author{Nhan D. V. Nguyen}

\address{Department of Mathematics, University of Colorado Boulder,  Boulder, CO 80309, USA \and Department of Mathematics and Statistics, Quy Nhon University, Binh Dinh, Vietnam}
\email{Nhan.Nguyen-1@colorado.edu, nguyenduvinhan@qnu.edu.vn}


\date{\today}
\begin{abstract} 
Let $N_n(a, b)$ denote the number of real zeros of Gaussian elliptic polynomials of degree $n$ on the interval $(a, b)$, where $a$ and $b$ may vary with $n$. We obtain a precise formula for the variance of $N_n(a, b)$ and utilize this expression to derive an asymptotic expansion for large values of $n$. Furthermore, we provide sharp estimates for the cumulants and central moments of $N_n(a, b)$. These estimates are instrumental in establishing sufficient conditions on the interval $(a, b)$ for $N_n(a, b)$ to satisfy both a central limit theorem and a strong law of large numbers. In the second part of the paper, we extend our analysis to nondegenerate Gaussian analytic functions, including well-known examples such as the Gaussian Weyl series and Weyl polynomials.
\end{abstract} 
\maketitle
\tableofcontents
\section{Introduction and main results} 
\subsection{Background} \label{S1.1}
Consider a positive integer $n$ and a nonempty interval $(a,b)\subset \mathbb R$, where $a$ and $b$ may depend on $n$. Let $p_0, p_1, \ldots, p_n$ be polynomials defined on $(a, b)$, and let $\omega_0, \omega_1, \ldots, \omega_n$ be jointly independent copies of a real random variale $\omega$ with zero mean and unit variance. The linear combination
\[P_n(x):=\sum_{j=0}^{n}\omega_j p_j(x)\]
is an example of a random polynomial. Various choices for the polynomials $p_j(x)$ give rise to distinct classes of random polynomials. Notable classes, of significant interest in probability theory and subjects of research attention in mathematical physics, include
\begin{enumerate}
    \item Kac polynomials (i.e., $p_j(x)=x^j$); and more generally, hyperbolic polynomials (i.e., $p_j(x)= \sqrt{\frac{L(L+1)\cdots(L+j-1)}{j!}}x^j$ for $L>0$);
    \item elliptic polynomials or binomial polynomials (i.e., $p_j(x)=\sqrt{\binom{n}{j}} x^j$);
    \item Weyl polynomials or flat polynomials (i.e., $p_j(x)=\frac{1}{\sqrt{j!}}x^j$);
    \item orthogonal polynomials (i.e., $p_j(x)$ form a system of orthonormal polynomials with respect to a fixed compactly supported measure); and 
    \item trigonometric polynomials (i.e., $p_j(x)$ are trigonometric polynomials). 
\end{enumerate}

Let $N_n(a,b)$ denote the number of real zeros of $P_n(x)$ inside $(a,b)$. Then,  $N_n(a,b)$ is a random variable taking values in $\{0,1,...,n\}$. A key problem in the theory of random polynomials is understanding the behavior of this random variable, with $n$ tending to infinity. During the past 90 years, most studies have been concerned with the estimation of $N_n(a,b)$, the expectation $\mathbb E[N_n(a,b)]$, the variance $\Var[N_n(a,b)]$, and the distribution of $N_n(a,b)$ in the large $n$ limit. These problems also naturally arise in different branches of physics because random polynomials serve as a basic model for eigenfunctions of chaotic quantum systems (see, for example, Bogomolny, Bohias, and Leb{\oe}uf  \cites{BBL, BBL1}).

Earlier investigations focused on Kac polynomials, with seminal contributions from Bloch and P\'{o}lya \cite{BP}, Kac \cites{K, K1},  Littlewood and Offord \cites{LO1, LO2, LO3}, and Erd{\H o}s and Offord \cite{EO}. Classical results, accompanied by numerous references on the subject, are available in the books by Bharucha-Reid and Sambandham \cite{BS} and Farahmand \cite{F}. We emphasize that when $\omega$ follows a normal distribution, the expected number of real zeros can be explicitly computed using the Kac-Rice formula (see \cite{K}, \cite{R}, or \cite{AW}*{Chapter 3}). Edelman and Kostlan \cite{EK} provided an elementary geometric derivation of the Kac-Rice formula, showing that $\mathbb E[N_n(a,b)]$ is simply the length of the moment curve $x\mapsto (p_0(x),...,p_n(x))$ for $x\in (a,b)$, projected onto the surface of the unit sphere, divided by $\pi$. In non-Gaussian scenarios, the universality method is crucial. It employs a replacement principle, enabling the comparison of correlation functions between two random functions when their log magnitudes closely match in distribution and meet specific non-concentration bounds (see, for example, Nguyen and Vu \cite{NV}, Tao and Vu \cite{TV}). Implementing this principle transforms the computation of zero distributions and interactions into a Gaussian framework. Currently, one can determine $\mathbb E[N_n(a,b)]$ for various classes of random polynomials, considering different choices for $p_j(x)$ and under very general assumptions for $\omega$ (see Do \cite{Do}, Do, H. Nguyen and Vu \cite{DHV}, Do, O. Nguyen and Vu \cites{DNV, DNV1}, Nguyen and Vu \cite{NV}, and the references provided therein).

Estimating the variance, however, has proved to be a much more difficult task and it is evident that this problem still awaits rigorous treatment. Despite a large number of prior studies, only a few are about $\Var[N_n(a,b)]$. For Kac polynomials, Maslova \cite{M} proved that if $\Prob(\{\omega=0\})=0$ and $\mathbb E[|\omega|^{2+\varepsilon}]<\infty$ for some  $\varepsilon>0$, then 
\[
\Var[N_n(\mathbb R)] =\frac{4}{\pi}\left(1-\frac{2}{\pi}\right)\log n+o(\log n)\quad \text{as}\quad n\to \infty.
\]
Beyond Kac polynomials, investigating the asymptotics of $\Var[N_n(a,b)]$ for other models of random polynomials has been extensively considered since the 1990s and has emerged as an active area of research in recent years. Utilizing the Kac-Rice formula and the universality method, the leading asymptotic terms for the variances of the real zeros were established for elliptic polynomials (see Bleher and Di \cite{BD}, Dalmao  \cite{D}), Weyl polynomials  (see Do and Vu \cite{DV}, Schehr and  Majumdar \cite{SM}),  orthogonal polynomials (see Lubinsky and  Pritsker \cite{LP}), and for trigonometric polynomials (see Bally, Caramellino, and Poly \cite{BCP}, Do, H. Nguyen and O. Nguyen \cite{DNN}, Granville and Wigman \cite{GW}). It is essential to note that most works focus on the case where $\omega$ is Gaussian and the second terms in the variance asymptotics for these random models remain unknown.

Establishing the limiting law of $N_n(a,b)$ presents a more intricate challenge. We say that $N_n(a,b)$ satisfies the central limit theorem (CLT) if the following convergence in distribution holds:
\[
    \frac{N_n(a,b)-\mathbb E[N_n(a,b)]}{\sqrt{\Var[N_n(a,b)]}}\xrightarrow{d} \mathcal N(0,1) \quad \text{as}\quad n\to \infty,
\]
where $\mathcal N(0,1)$ denotes the standard normal distribution. In 1974, Maslova \cite{M1} proved the CLT for Kac polynomials. Nearly four decades later, CLTs were extended to other classes of random polynomials. Granville and Wigman \cite{GW} and Aza\"{i}s and Le\'{o}n \cite{AzL} studied the CLT for Gaussian Qualls’ trigonometric polynomials using different methods. Aza\"{i}s, Dalmao, and Le\'{o}n \cite{ADL} extended this result to classical trigonometric polynomials. Dalmao \cite{D} achieved the same for elliptic polynomials, with Ancona and Letendre \cite{AL} recently generalizing this result using the method of moments. The primary tool employed in \cite{ADL}, \cite{AzL}, and \cite{D} is an $L^2$ expansion of the number of real zeros. CLTs for Weyl polynomials and Weyl series were obtained by Do and Vu \cite{DV} using the cumulant convergence theorem. In 2022, Nguyen and Vu \cite{NV1} established the CLT for random polynomials with coefficients of polynomial growth. Their proof adapted the universality method and the argument in Maslova \cite{M1}, approximating the number of zeros through a sum of independent random variables.

This paper is dedicated to establishing a full asymptotic expansion for the variance, along with sharp estimates for the cumulants and moments, of the number of real zeros of certain Gaussian processes. These findings are essential in comprehending the behavior and the limiting law of the real zeros. Our exploration begins with Gaussian elliptic polynomials,  which arise when considering the quantum mechanics of a spin $\rm S$ system whose modulus $S$ is conserved  (see Bogomolny, Bohias, and Leb{\oe}uf \cite{BBL}). These polynomials are highly relevant for applications in quantum chaos and have been extensively studied in the mathematical literature (see Ancona and Letendre \cite{AL}, Bleher and Di \cites{BD, BD1}, Dalmao \cite{D}, Edelman and Kostlan \cite{EK}, Flasche and Kabluchko \cite{FK}, Nguyen and Vu \cite{NV}, Schehr and  Majumdar \cite{SM}, Tao and Vu \cite{TV}). Additionally, our investigation extends to nondegenerate Gaussian analytic functions, leading to the establishment of new findings concerning the Gaussian Weyl series and Weyl polynomials. These random analytic functions are also objects of considerable interest in probability theory and mathematical physics (see Do and Vu \cite{DV}, Edelman and Kostlan \cite{EK}, Flasche and Kabluchko \cite{FK}, Nazarov and Sodin \cite{NS}, Nguyen and Vu  \cite{NV}, Schehr and  Majumdar \cite{SM}, Tao and Vu \cite{TV}).
\subsection{Real zeros of Gaussian elliptic polynomials} \label{S1.2}
Let $N_n(a,b)$ denote the number of real zeros on $(a,b)$ of the elliptic polynomial
\[
P_n(x)=\sum_{j=0}^n \omega_j \sqrt{\binom{n}{j}} x^j,
\]
where $\omega_j$ are i.i.d. normalized Gaussian random variables. 
In 1995, Edelman and Kostlan \cite{EK} showed that 
\begin{equation}
    \label{e1.1}
    \mathbb E[N_n(a,b)]=\frac{1}{\pi}\int_a^b \frac{\sqrt n}{1+x^2}dx=\frac{1}{\pi}\sqrt n(\arctan b-\arctan a).
\end{equation}
In 1997, Bleher and Di \cite{BD} determined the leading term in the large $n$ expansion of the variance $\Var[N_n(a,b)]$ for fixed $a$ and $b$. Specifically, defining
\[
    \delta_0(s)=\frac{e^{-s^2/2}(1-s^2-e^{-s^2})}{1-e^{-s^2}-s^2e^{-s^2}},\quad  \gamma_0(s)=\frac{1-e^{-s^2}-s^2e^{-s^2}}{(1-e^{-s^2})^{3/2}},
\]
and 
\begin{equation*}
f_0(s)=\left(\sqrt{1-\delta_0^2(s)}+\delta_0(s) \arcsin \delta_0(s)\right)\gamma_0(s)-1,
\end{equation*}
it was shown in \cite{BD}*{\S6} that
\begin{equation}
    \label{e1.2}
\Var[N_n(a,b)]=(1+\kappa_0+o(1))\mathbb E[N_n(a,b)]\quad \text{as}\quad n\to \infty,
\end{equation}
where 
\begin{equation}
    \label{e1.3}
    \kappa_0:=\frac{2}{\pi}\int_0^\infty f_0(s)ds
\end{equation}
and $(1+\kappa_0)\approx 0.5717310486$. In \cite{D}, Dalmao obtained the same result for $\Var[N_n(\mathbb R)]$.

Notably, for fixed $a$ and $b$, \eqref{e1.1} provides an exact formula for $\mathbb E[N_n(a,b)]$, while \eqref{e1.2} offers an asymptotic bound with a less precise error term $o(\sqrt n)$. The precise characterization of the error term in $\Var[N_n(a,b)]$ remains an open challenge, requiring a nontrivial and highly technical endeavor. Additionally, the determination of the second asymptotic term in the variance expansion for the number of real zeros across all classes of random polynomials in Section \ref{S1.1} is yet to be accomplished, presenting ample opportunities for further improvement.

Our interest extends to establishing a complete asymptotic expansion for the variance of real zeros in the large degree limit. As a crucial first step, we aim to derive an exact and accessible formula for $\Var[N_n(a,b)]$. To formulate our results, we introduce the functions
\begin{align*}
    \Delta_n(s)&:=(1+s^2)^{-n/2}\frac{(1+s^2)[1-(1+s^2)^{-n}]-n s^2}{1-(1+s^2)^{-n}-ns^2(1+s^2)^{-n}},\\ 
    \Gamma_n(s)&:=\frac{1-(1+s^2)^{-n}-ns^2(1+s^2)^{-n}}{[1-(1+s^2)^{-n}]^{3/2}},\\
F_n(s)&:=\left(\sqrt{1-\Delta_n^2(s)}+\Delta_n(s) \arcsin \Delta_n(s)\right)\Gamma_n(s)-1,
\end{align*}
along with the integrals
\begin{align*}
    K_n(a,b)&:=\frac{2}{\pi}\int_0^{\sqrt n |\alpha(a,b)|}\frac{F_n(s/\sqrt n)}{1+s^2/n}ds,\\
 L_n(a,b)&:=\frac{2}{\pi^2}\int_0^{\sqrt n |\alpha(a,b)|}\frac{F_n(s/\sqrt n)}{1+s^2/n}\sqrt n\arctan(s/\sqrt n)ds,
\end{align*}
where $\alpha(a,b):=(b-a)/(1+ab)$, for $-\infty\le a,b\le \infty$. For brevity, we will use $K_n$ and $L_n$ instead of $K_n(a,b)$ and $L_n(a,b)$ if $\sqrt n|\alpha(a,b)|$ is replaced by $\infty$.
\begin{theorem}[Exact variance formulas]\;
\label{T1.1}
\begin{enumerate}
\item If $\alpha(a,b)>0$, then 
\begin{equation}
    \label{e1.4}
    \Var[N_n(a,b)]= (1+K_n(a,b))\mathbb E[N_n(a,b)]-L_n(a,b).
\end{equation}
\item For $\alpha(a,b)=0$, one has $(a,b)=\mathbb R$ and 
\begin{equation}
    \label{e1.5}
    \Var[N_n(\mathbb R)]=(1+K_n)\mathbb E[N_n(\mathbb R)].
\end{equation}
\item When $\alpha(a,b)<0$, it holds that
\begin{equation}
    \label{e1.6}
    \begin{split}
    \Var[N_n(a,b)]&=(1+K_n)\mathbb E[N_n(a,b)]\\
    &+\left(K_n-K_n(a,b)\right)\left(\mathbb E[ N_n(a,b)]-\sqrt n\right)-L_n(a,b).
    \end{split}
\end{equation}
\end{enumerate}
\end{theorem}
By leveraging Theorem \ref{T1.1} and rigorously applying Taylor expansions, we derive precise asymptotic expressions for the variance $\Var[N_n(a,b)]$ in the large $n$ limit. The key advantage of Theorem \ref{T1.1} is its applicability to cases where the interval $(a,b)$ depends on $n$. Specifically, we establish a complete asymptotic expansion for $\Var[N_n(a,b)]$, provided that the interval $(a,b)$ does not contract too rapidly as $n \to \infty$.
\begin{theorem}[Variance asymptotic expansions]  \label{T1.2} Write
$\alpha_n=\sqrt n\alpha(a,b).$
\begin{enumerate}
    \item Assume first that $|\alpha_n|\to \infty$ as $n\to \infty$. Let
        \begin{equation}
        \label{e1.7}
        d_n=\left\lfloor\frac{\alpha_n^2+3\log|\alpha_n|}{\log n} \right\rfloor,
        \end{equation}
    where $\lfloor\cdot\rfloor$ denotes the integer part. Then 
    \begin{equation}
        \label{e1.8}
        \Var[N_n(a,b)]=\left(1+\sum_{k=0}^{d_n} \frac{\kappa_k}{n^k}\right)\mathbb E[N_n(a,b)]-\sum_{k=0}^{d_n} \frac{\ell_k}{n^k}+O(\alpha_n^4e^{-\alpha_n^2}),
    \end{equation}
    in which $\kappa_k$ and $\ell_k$ are real constants independent of $n$, $a$, and $b$. In particular, $\kappa_0$ is defined as in  \eqref{e1.3} and 
    \[
        \ell_0=\frac{2}{\pi^2}\int_0^\infty sf_0(s)ds.
    \]
    Therefore, if $\alpha_n^2/\log n \to \infty$ as $n\to \infty$, then 
    $\Var[N_n(a,b)]$ admits a full asymptotic expansion of the form
     \begin{equation}
        \label{e1.9}
        \Var[N_n(a,b)]\sim \left(1+\sum_{k=0}^\infty \frac{\kappa_k}{n^k}\right)\mathbb E[N_n(a,b)]-\sum_{k=0}^\infty \frac{\ell_k}{n^k}.
    \end{equation}
    \item  Assume now that $|\alpha_n|=O(1)$ as $n\to \infty$. If $\alpha_n=c>0$, then 
    \begin{equation}
        \label{e1.10}
        \Var[N_n(a,b)]\sim \left(1+\sum_{k=0}^\infty \frac{\kappa_{c,k}}{n^k}\right)\mathbb E[N_n(a,b)]-\sum_{k=0}^\infty \frac{\ell_{c,k}}{n^k},
    \end{equation}
    in which $\kappa_{c,k}$ and $\ell_{c,k}$ are real constants. In particular, 
    \[
    \kappa_{c,0}=\frac{2}{\pi}\int_0^cf_0(s)ds \quad \text{and}\quad \ell_{c,0}=\frac{2}{\pi^2}\int_0^csf_0(s)ds.
    \]
    For $\alpha_n=-c<0$, we have 
    \begin{equation}
    \begin{split}
        \label{e1.11}
        \Var[N_n(a,b)]&\sim \left(1+\sum_{k=0}^\infty \frac{\kappa_{k}}{n^k}\right)\mathbb E[N_n(a,b)]\\
        &-\left(\sum_{k=0}^\infty\frac{\kappa_k-\kappa_{c,k}}{n^k}\right)\frac{\sqrt n}{\pi} \arctan\frac{c}{\sqrt n}-\sum_{k=0}^\infty \frac{\ell_{c,k}}{n^k}.
\end{split}
    \end{equation}
    \item Finally, assume that $\alpha_n=o(1)$ as $n\to \infty$.
 If $\alpha_n>0$,  then 
    \begin{equation}
    \label{e1.12}
    \Var[N_n(a,b)]=\frac{1}{\pi}\alpha_n-\frac{1}{\pi^2}\alpha_n^2+\frac{1}{12\pi}\alpha_n^3-\frac{5}{12\pi}\frac{\alpha_n^3}{n}+\frac{2}{3\pi^2}\frac{\alpha_n^4}{n}+O(\alpha_n^5).
\end{equation}
If $\alpha_n<0$, then 
    \begin{equation}
    \label{e1.13}
    \begin{split}
    \Var[N_n(a,b)]&=\left(1+\sum_{k=0}^{u_n} \frac{\kappa_k}{n^k}\right)\sqrt n +\frac{2}{\pi}\left(\alpha_n-\frac{\alpha_n^3}{3n}\right)\sum_{k=0}^{v_n}\frac{\kappa_k}{n^k}\\
&+\frac{1}{\pi}\alpha_n-\frac{1}{\pi^2}\alpha_n^2-\frac{1}{12\pi}\alpha_n^3-\frac{1}{4\pi}\frac{\alpha_n^3}{n}+\frac{2}{3\pi^2}\frac{\alpha_n^4}{n}+O(|\alpha_n|^5),
\end{split}
\end{equation}
where 
\[
u_n:=\left\lfloor \frac 12 -\frac{5\log|\alpha_n|}{\log n} \right\rfloor \quad \text{and}\quad v_n:=\left\lfloor -\frac{4\log|\alpha_n|}{\log n} \right\rfloor.
\]
If, in addition, $v_n\to \infty$ as $n\to \infty$, then
\begin{equation}
        \label{e1.14}
        \Var[N_n(a,b)]\sim \left(1+\sum_{k=0}^\infty \frac{\kappa_k}{n^k}\right)\sqrt n.
    \end{equation}
For $\alpha_n=0$, we have $(a,b)=\mathbb R$ and  $\Var[N_n(\mathbb R)]$ has a full asymptotic expansion of the form
    \begin{equation}
        \label{e1.15}
        \Var[N_n(\mathbb R)]\sim \left(1+\sum_{k=0}^\infty \frac{\kappa_k}{n^k}\right)\sqrt n.
    \end{equation}
\end{enumerate}
\end{theorem}
In Section \ref{S2.3}, we provide exact definitions for $\kappa_k$, $\kappa_{c,k}$, $\ell_{k}$, and $\ell_{c,k}$, along with some detailed numerical computations (see Table \ref{tab1}).

\begin{remark}
    While determining the second-order term in the variance asymptotic expansion for the number of real zeros remains challenging and largely unexplored for other classes of random polynomials listed in Section \ref{S1.1}, it becomes more tractable for elliptic polynomials, thanks to the exact variance formulas provided in Theorem \ref{T1.1}. In the next subsection, we extend our analysis to encompass other classes of Gaussian processes, establishing precise expressions for cumulants of the number of real zeros (as detailed in Theorems \ref{T1.14} and \ref{T1.24}). These expressions form the foundation for deriving second-order terms or even full asymptotic expansions for all cumulants, central moments, and moments of the number of real zeros of these processes (as outlined in Remark \ref{r.asym.cumulants}, Theorem \ref{T1.20}, Remark \ref{r.asym.central.moments}, and Remark \ref{r.asym.moments}).
\end{remark}

Here and throughout, for a positive integer $k$ and a random variable $X$, let $s_k[X]$ and $\mu_k[X]$ denote the $k$th cumulant and $k$th central moment of $X$, respectively.

Our next objective is to explore the asymptotic behaviors of the cumulants $s_k[N_n(a,b)]$.
\begin{theorem}[Asymptotics of  cumulants] \label{T1.3} For each positive integer $k$, there exists a finite constant $\beta_k$ independent of $n$, $a$, and $b$, such that
\begin{equation}
\label{e1.16}
s_k[N_n(a,b)]=\beta_k \mathbb E[N_n(a,b)]+O(1) \quad \text{as}\quad n\to \infty.
\end{equation}
\end{theorem}
\begin{remark}
Since $s_1[N_n(a,b)]=\mathbb E[N_n(a,b)]$, \eqref{e1.16} is trivial for $k=1$, where $\beta_1=1$. Utilizing Theorem \ref{T1.2} and recognizing that $s_2[N_n(a,b)] = \Var[N_n(a,b)]$, we establish the validity of \eqref{e1.16} for $k=2$, with $\beta_2=1+\kappa_0$. It is worth noting that when studying the gap probabilities for elliptic polynomials, Schehr and Majumdar \cite{SM}*{Appendix E} proved that
\[
s_3[N_n(a,b)]\sim \beta_3 \mathbb E[N_n(a,b)] \quad \text{as}\quad n\to \infty,
\]
under the assumptions that $\beta_3$ is well-defined and  $\mathbb E[N_n(a,b)]\sim \sqrt n$ in the large $n$ limit. They expected a similar mechanism to hold for higher values of $k$ (see \cite{SM}*{Equation 93}). Accordingly, our theorem provides a fuller treatment.
\end{remark}
Using the asymptotics in \eqref{e1.16}, we establish the asymptotic normality of $N_n(a,b)$.
\begin{theorem}[Central limit theorem] \label{T1.5} Let $\alpha_n$ be defined as in Theorem \ref{T1.2}.  If either $\alpha_n\le 0$ or $\alpha_n\to \infty$ as $n\to \infty$, then $N_n(a,b)$ satisfies the CLT. 
\end{theorem}
\begin{remark}
    In 2015, Dalmao \cite{D} established the CLT for $N_n(\mathbb R)$ using the Wiener-It\^o expansion and the fourth-moment theorem. Ancona and Letendre \cite{AL} independently validated Dalmao's finding in 2021 using the method of moments. In our current study, we utilize the asymptotic behaviors of cumulants, as described in Theorem \ref{T1.3}, to establish sufficient conditions on $(a, b)$ for $N_n(a, b)$ to satisfy the CLT. These conditions essentially imply that the interval $(a, b)$ should not shrink too rapidly as $n \to \infty$.
\end{remark}

Next, we can apply Theorem \ref{T1.3} to derive the asymptotics of the central moments $\mu_k[N_n(a,b)]$.

\begin{corollary}[Asymptotics of  central moments] \label{C1.7}
Fix $k\ge 1$. As $n\to \infty$, it holds that 
\begin{equation}
\label{e1.17}
    \mu_{2k}[N_n(a,b)]=\frac{(2k)!\beta_2^k}{k! 2^k} (\mathbb E[N_n(a,b)])^k+O((\mathbb E[N_n(a,b)])^{k-1})
\end{equation}
and 
\begin{equation}
\label{e1.18}
    \mu_{2k+1}[N_n(a,b)]=\frac{(2k+1)!\beta_2^{k-1}\beta_3}{(k-1)!2^{k-1}3!} (\mathbb E[N_n(a,b)])^k+O((\mathbb E[N_n(a,b)])^{k-1}).
\end{equation}
\end{corollary}
\begin{remark}
Ancona and Letendre \cite{AL} previously examined the asymptotics of $\mu_k[N_n(\mathbb R)]$ and showed that, as $n\to \infty$, 
\begin{equation}
\label{e1.19}
\mu_k[N_n(\mathbb R)]=\mu_k[\mathcal N(0,1)]\beta_2^{k/2}n^{k/4}+O(n^{(k-1)/4}\log^k(n)),
\end{equation}
where $\mu_k[\mathcal N(0,1)]$ denotes the $k$th moment of the standard normal distribution. However, since $\mu_{2k+1}[\mathcal N(0,1)]=0$, formula \eqref{e1.19} does not provide the leading asymptotics for $\mu_{2k+1}[N_n(\mathbb R)]$. Therefore, our results in Corollary \ref{C1.7} not only fill this gap but also offer an improvement of \eqref{e1.19}, as our error terms are only $O(n^{(\lfloor k/2 \rfloor-1)/2})$.
\end{remark}
By employing the asymptotics of central moments provided in \eqref{e1.17}, reinforced by a Borel-Cantelli type argument, we establish a strong law of large numbers for $N_n(a,b)$.
\begin{theorem}[Strong law of large numbers] \label{T1.9} If either $\alpha_n\le 0$ or $\sum_{n=1}^\infty \alpha^{-k}_n<\infty$ for some positive constant $k$, then 
\[
\frac{N_n(a,b)}{\mathbb E[N_n(a,b)]}\xrightarrow{a.s.} 1 \quad \text{as } n\to \infty.
\]
\end{theorem}
\begin{remark}  Similar considerations may apply to the linear statistics $N_n(\phi)$, defined as
\[
N_n(\phi)=\sum_{x\in Z_n}\phi(x),
\]
where $Z_n$ is the real zero set of the elliptic polynomial $P_n(x)$ and $\phi$ satisfies suitable assumptions. Further details on this topic can be found in \cite{AL}, where Ancona and Letendre considered the leading asymptotics of the central moments, the CLT, and the strong law of large numbers for these linear statistics. Note that $N_n(\phi)$ reduces to $N_n(a,b)$ if we set $\phi(x)=\pmb 1_{(a,b)}(x)$, denoting the indicator function of the interval $(a,b)$.
\end{remark}
In concluding this subsection, we identify potential directions for future research in the realm of elliptic polynomials.

To begin, for Gaussian elliptic polynomials, it follows from \eqref{e1.1} that $\mathbb E[N_n(\mathbb R)]$ is precisely $\sqrt{n}$ for all $n$. In \cite{BD1}, Bleher and Di, among other significant findings, extended this result to non-Gaussian counterparts.
\begin{theorem}[\cite{BD1}]
Assume there exist positive constants $c$ and $C$ such that the characteristic function $\varphi(s)$ of $\omega$ satisfies the conditions
\[
|\varphi(s)|\le \frac{1}{(1+c|s|)^6},\quad \left|\frac{d^j\varphi(s)}{ds^j} \right|\le \frac{C}{(1+c|s|)^6},\quad j=1,2,3,\quad  s\in \mathbb R. 
\]
Then, as $n\to \infty$,
\begin{equation}\label{e1.20}
\mathbb E[N_n(\R)]=\sqrt n+o(n^{1/2}).
\end{equation}
\end{theorem}
The same result, without the assumption on $\varphi(s)$, was established in a recent work by Flasche and Kabluchko \cite{FK}. In \cite{TV}*{Theorem 5.6}, Tao and Vu demonstrated the universality of this result, extending it to scenarios where the random variable $\omega$ has zero mean, unit variance, and finite $(2+\varepsilon)$-moments. A more refined quantitative version of \eqref{e1.20} was recently provided by Nguyen and Vu \cite{NV}*{Corollary 6.4}:
\begin{equation*}
\mathbb E[N_n(\R)]=\sqrt n+O(n^{1/2-c}),\quad c>0.
\end{equation*}

Considering the Gaussian elliptic polynomials, we infer from \eqref{e1.2} that
\begin{equation}
\label{e1.21}
\Var[N_n(\R)]=(1+\kappa_0)\sqrt{n}+o(n^{1/2})\quad \text{as}\quad n\to \infty.
\end{equation}
A natural question arises: Is \eqref{e1.21} still valid if $\omega$ has zero mean, unit variance, and finite $(2+\varepsilon)$-moments? More broadly, there is an interest in generalizing all the aforementioned results to a non-Gaussian setting.

Additionally, it could be intriguing to extend the findings of this paper to the number of real zeros of a square system $\pmb P=(P_1,\dots, P_m)$ of $m$ polynomial equations in $m$ variables 
of common degree $n>1$, 
\begin{equation*}
  P_{\ell}(\pmb x)=\sum_{|\pmb j|\le n}\pmb \omega^{(\ell)}_{\pmb j}\pmb x^{\pmb j},
\end{equation*}
where 
\begin{itemize}
  \item $\pmb j=(j_{1},\dots,j_{m})\in\mathbb N^{m}$ 
and $|\pmb j|=\sum^{m}_{k=1}j_{k}$; 
  \item $\pmb \omega^{(\ell)}_{\pmb j}=\omega^{(\ell)}_{j_1\dots j_m}\in\R$, $\ell=1,\dots,m$, $|\pmb j|\leq n$, and the coefficients $\pmb \omega^{(\ell)}_{\pmb j}$ 
are independent centered normally distributed random variables with variances
\begin{equation*}
  \Var[\pmb \omega^{(\ell)}_{\pmb j}]=\binom{n}{\pmb j}=\frac{n!}{ 
j_1!\dots j_m!(n-|\pmb j|)!};
\end{equation*}
  \item $\pmb x=(x_{1},\dots,x_{m})$ %is a point in $\R^{m}$ 
    and $\pmb x^{\pmb j}=\prod^m_{k=1} x^{j_k}_{k}$.
\end{itemize}

Such a system, also referred to as a Kostlan-Shub-Smale system,  was initially introduced and investigated by Kostlan \cite{Ko} and further developed by Armentano \etal  \cites{AADL, AADL1}, Aza\"{i}s and Wschebor \cite{AW0}, Bleher and Di \cite{BD1}, Edelman and Kostlan \cite{EK}, Shub and Smale \cite{SS}, Wschebor \cite{W}. Accordingly, we hope that the concepts and techniques of this paper may stimulate further research in this fascinating area. 
\subsection{Real zeros of nondegenerate real Gaussian analytic functions} \label{S1.3}
Let 
\[Q(z)=\sum_{j=0}^\infty \omega_j q_j(z)\] 
be a real Gaussian analytic function (real GAF) on $\mathbb C$; that is, $\omega_j$ are i.i.d. normalized Gaussian random variables and $q_j$ are analytic functions on $\mathbb C$ such that $\sum_{j=0}^\infty |q_j(z)|^2<\infty$ uniformly on any compact subset of $\mathbb C$ (see Do and Vu \cite{DV}). 

Let $k\ge 1$ be an integer. For $\pmb z =(x_1,y_1,...,x_k,y_k)\in \mathbb R^{2k}$, let $L^{\pmb z}$ be the linear functional defined as
\[
L^{\pmb z}Q(\pmb \xi)=\sum_{1\le j\le k}[x_jQ(\xi_j)+y_j Q'(\xi_j)],\quad \pmb \xi=(\xi_1,...,\xi_k)\in \mathbb R^k.
\]
Moreover, for any subset $I\subset \{1,...,k\}$, the linear functional $L_I^{\pmb z}$ is defined by summing over $j\in I$ instead of the full range. For further insights into the concept of linear functionals, we refer the reader to Do and Vu \cite{DV} and Nazarov and Sodin \cite{NS}.

For $\pmb \xi=(\xi_1,...,\xi_k)\in \mathbb R^k$ and a nonempty subset $I\subset \{1,...,k\}$, we define $\pmb \xi_I=(\xi_i)_{i\in I}$, and express the distance between the configurations $\pmb \xi_I$ and $\pmb \xi_J$ as
\[d(\pmb \xi_I, \pmb \xi_J)=\inf_{i\in I, j\in J}|\xi_i-\xi_j|.\]

Define $\Psi_k$ as the set of all non-increasing functions $\psi: [0,\infty)\to [0,\infty)$ such that
\[
\int_0^\infty\psi(x) x^{k-1}dx<\infty,
\]
and let $\Psi_\infty =\bigcap_{k=1}^\infty \Psi_k$. The inclusion chain $\Psi_1\supset \Psi_2 \supset \cdots \supset \Psi_\infty$ follows immediately.

Define $\mathcal A_k$ as the set of all real GAFs $Q$ satisfying the following three hypotheses:
\begin{enumerate}[label={$(H_\arabic*)$}]
\item  \label{H1} $Q$ is $2k$-nondegenerate. 
\item  \label{H2} For each $x\in \mathbb R$, there exists a deterministic function $Q_x: \mathbb R\to (0,\infty)$ such that 
\[\mathbb  E[Q_x(t)Q(t+x)Q_x(s)Q(s+x)]=\mathbb  E[Q(t)Q(s)],\quad t,s\in \mathbb R.\]
\item  \label{H3} There  exist finite positive constants $c_k$ and $\tau_k$, along with a function $\psi\in \Psi_k$, such that the following holds: For any $\pmb z=(x_1,y_1,...,x_k,y_k)\in \mathbb R^{2k}$, $\pmb \xi=(\xi_1,...,\xi_k)\in \mathbb R^k$, and any partition $\{1,...,k\}=I\cup J$, with $d:=d(\pmb \xi_I,\pmb \xi_J)\ge 2\tau_k$, we have 
\[
\left|\mathbb  E\left[L_I^{\pmb z}Q(\pmb \xi_I)L_J^{\pmb z}Q(\pmb \xi_J)\right]\right| \le c_k \psi(d-\tau_k)\big(\mathbb  E[\left|L_I^{\pmb z}Q(\pmb \xi_I)\right|^2 ]+\mathbb  E [\left|L_J^{\pmb z}Q(\pmb \xi_J)\right|^2]\big).
\]
\end{enumerate} 
Furthermore, let $\mathcal A_\infty=\bigcap_{k=1}^\infty \mathcal A_k$.
\begin{remark} \label{R1.12} We refer the reader to \cite{DV}*{\S9} for the precise definition of nondegenerate real GAFs, which serves as the real counterpart to the complex nondegeneracy notion introduced in \cite{NS}. For instance, if $Q(z) = \sum_{j=0}^\infty \omega_j c_j z^j$, where real constants $c_j$ satisfy $\sum_{j=0}^\infty c_j^2<\infty$ and $c_0, c_1, \ldots, c_{2k-1}\ne 0$, then $Q$ is a $2k$-nondegenerate GAF. Hypothesis \ref{H1} ensures the local boundedness of the $k$-point correlation function $\rho_{k}$ for the real zeros of $Q$ (see Lemma \ref{L3.2}). 

Assumption \ref{H2} further asserts that the distribution of the real zeros of $Q$ is invariant under translations on $\mathbb R$. Specifically, if $Q$ is stationary, then \ref{H2} holds with $Q_x\equiv 1$. 

Hypothesis \ref{H3} introduces a clustering property for $\rho_k$, indicating that if the variables in $\mathbb{R}^k$ can be divided into two well-separated clusters, the correlation function $\rho_k$ closely approximates the product of the corresponding factors (see Lemma \ref{L3.3}). 

If $Q$ is $2k$-nondegenerate and $m\in \{1,...,k\}$, it is also $2m$-nondegenerate. Additionally, when \ref{H3} holds for $k\ge 1$, it extends to all positive integers $m\le k$. This establishes the inclusion chain $\mathcal A_1 \supset \mathcal A_2 \supset \cdots \supset A_\infty$. Moreover, if $Q\in \mathcal A_k$ and $m$ is a positive integer such that $m\le k$, the $m$-point correlation function $\rho_{m}$ for the real zeros of $Q$ is uniformly bounded on $\mathbb R^m$ (see Lemma \ref{L3.4}).
 \end{remark}
 \begin{example} \label{E1.13} For the Gaussian Weyl series, defined as $W(z) = \sum_{j=0}^\infty \omega_j\frac{z^j}{\sqrt{j!}}$ with $\omega_j$ being i.i.d. normalized Gaussian random variables, we have $W\in \mathcal A_\infty$.

Indeed, according to Remark \ref{R1.12}, $W$ is $2k$-nondegenerate for any $k \ge 1$. For each $x \in \mathbb{R}$, let $W_x(t) = e^{-xt-x^2/2}$. Then, for any $t, s \in \mathbb{R}$, we have
 \begin{align*}
  \mathbb E[W_x(t)W(t+x)W_x(s)W(s+x)]&=e^{-xt-x^2/2}e^{-xs-x^2/2}e^{(t+x)(s+x)}\\
  &=e^{ts}\\
  &=\mathbb E[W(t)W(s)],   
 \end{align*}
 showing that $W$ satisfies hypothesis \ref{H2}. Finally, it follows from \cite{DV}*{Lemma 18} that $W$ satisfies hypothesis \ref{H3} for any $k\ge 1$ with $\psi(t):=e^{-t^2/2}\in \Psi_\infty$.   
 \end{example}
\begin{theorem}[Precise expressions for cumulants] \label{T1.14} 
    Fix $k\in \mathbb N\cup\{\infty\}$ and let $Q\in \mathcal A_k$. Given $R>0$, let $N_Q(R)$ denote the number of real zeros of $Q$ on $[0,R]$. For any positive integer $m$ with $m\le k$, there are bounded functions $\theta_m^Q$ and $\lambda_m^Q: [0,\infty)\to \mathbb R$ such that
    \begin{equation} \label{e1.22}
        s_m[N_Q(R)] =R\theta_m^Q(R)+\lambda_m^Q(R).
    \end{equation}
    Furthermore, there exist finite constants $\theta_{m,\infty}^Q$, $\lambda_{m,\infty}^Q$, and a function $\psi\in \Psi_k$, independent of $R$, satisfying, as $R\to \infty$,
    \begin{equation} 
    \label{e1.23} 
    \begin{split} \theta_m^Q(R)&=\theta_{m,\infty}^Q+O\left(\int_{R}^\infty \psi(x)x^{m-2}dx\right),\\ 
    \lambda_m^Q(R)&= \lambda_{m,\infty}^Q+O\left(\int_{R}^\infty \psi(x)x^{m-1}dx\right).  
    \end{split}
    \end{equation}
\end{theorem}
\begin{remark}\label{r.asym.cumulants}
    If $\psi\in \Psi_k$, then as $R\to \infty$,
    \[
    \int_R^\infty \psi(x)x^{k-1}dx=o(1).
    \]
    Consequently, from \eqref{e1.22} and \eqref{e1.23}, as $R\to \infty$,
    \[ s_m[N_Q(R)]=R\theta_{m,\infty}^Q+\lambda_{m,\infty}^Q+o(R^{m-k}),
    \]
   where $m$ and $k$ are positive integers with $m\leq k$, $Q\in \mathcal A_k$, and $\theta_{m,\infty}^Q$ and $\lambda_{m,\infty}^Q$ are finite real numbers independent of $R$. Furthermore, if $Q\in \mathcal A_\infty$ and $m$ is any positive integer, then $s_m[N_Q(R)]$ has a full asymptotic expansion of the form
    \[
    s_m[N_Q(R)]\sim R\theta_{m,\infty}^Q+\lambda_{m,\infty}^Q.
    \]
\end{remark}
\begin{remark}
    In \cite{G}, Gass investigated cumulant asymptotics for random models with slowly decreasing covariance functions. Gass's approach refines recent works by Ancona and Letendre \cites{AL, AL1}, which established the convergence of the $k$th central moment of the number of real zeros, suitably scaled, to the $k$th moment of a Gaussian random variable. These investigations assume that the covariance functions and their derivatives belong to certain Lebesgue spaces. For instance, as shown in \cite{G}*{Theorem 1.5}, if $Q$ is a stationary Gaussian process with $C^\infty$ paths and its covariance function and derivatives belong to $L^k(\mathbb R)$ for all $k>1$, then for any positive integer $m$, 
    \[
    \lim_{R\to \infty}\frac{s_m[N_Q(R)]}{R}=\theta_{m,\infty}^Q.
    \]
   In comparison to this limit, under the assumptions of fast decorrelation and analyticity, the outcomes in Remark \ref{r.asym.cumulants} provide a more thorough and detailed characterization.
\end{remark}

 By combining Example \ref{E1.13} and Theorem \ref{T1.14}, we derive precise expressions and asymptotic behaviors for the cumulants of the number of real zeros of the Gaussian Weyl series. 
\begin{corollary}[Real zeros of the Gaussian Weyl series] \label{C1.17} 
For $R>0$, let $N_W(R)$ denote the number of real zeros of the Gaussian Weyl series $W$ on $[0, R]$. For any positive integer $k$, there exist bounded functions $\theta_k^W, \lambda_k^W: [0,\infty)\to \mathbb R$ such that
\[
s_k[N_W(R)]=R\theta^W_k(R)+\lambda^W_k(R).
\]
Furthermore, there exist constants $\theta_{k,\infty}^W$,  $\lambda_{k,\infty}^W$, and $c_k>0$, which are independent of $R$ and satisfy, as $R\to \infty$,
    \[
    \theta_k^W(R)=\theta_{k,\infty}^W+O(R^{k-3}e^{-c_kR^2}) \quad \text{and}\quad \lambda_k^W(R)=\lambda_{k,\infty}^W+O(R^{k-2}e^{-c_kR^2}).
    \]
    Specifically, $\theta_{k,\infty}^W=\frac{1}{\pi}\beta_k$, where $\beta_k$ is as in Theorem \ref{T1.3}. Consequently, $\theta^W_{2,\infty}>0$ and $N_W(R)$ follows the CLT as $R\to \infty$.
\end{corollary}
\begin{remark} 
   For a nonzero, compactly supported, and bounded function $\phi: \mathbb R\to \mathbb R$, define $N^\phi_W(R)=\sum_{x\in Z(W)}\phi(x/R)$, where $Z(W)$ represents the multiset of real zeros of $W$. In \cite{DV}, Do and Vu established the asymptotic normality result for $N^\phi_W(R)$ by employing the cumulant convergence theorem (see \cite{J}) and demonstrating that
    \[
    \lim_{R\to \infty}\frac{\Var[N^\phi_W(R)]}{R\|\phi\|_2^2}=\theta^W_{2,\infty}
    \]
    and 
     \[|s_k[N^\phi_W(R)]| \le C_{\phi,k}R,\]
    where $C_{\phi, k}$ is a positive constant depending only on $\phi$ and $k$. These findings were motivated by related results for the complex zeros of $W$ by Nazarov and Sodin \cite{NS}. Choosing $\phi(t)=\pmb 1_{[0,R]}(t)$ implies $N^\phi_W(R)=N_W(R)$. Consequently, the CLT also applies to $N_W(R)$. However, it is worth noting that Corollary \ref{C1.17} provides much more precise estimates for $s_k[N_W(R)]$.
\end{remark}
Now, we broaden the scope of Theorem \ref{T1.14} to include a wider range of Gaussian processes, such as random polynomials. To this end, fix a positive integer $k$ and let $Q\in \mathcal A_k$. For $n\ge 1$, let $\mathcal I_n\subset \mathbb R$ be an interval whose endpoints may depend on $n$ and let $\varepsilon_n>0$ be such that $\varepsilon_n\to 0$ as $n\to \infty$. Define $\mathcal A_k^Q(\mathcal I_n, \varepsilon_n)$ as the set of all smooth centered Gaussian processes $\{Q_n(x): x\in \mathcal I_n\}$ that satisfy the following hypothesis:
\begin{enumerate}[label={$(H_4)$}]
\item  \label{H4} For any $\pmb z=(x_1,y_1,...,x_k,y_k)\in \mathbb R^{2k}$ and $\pmb \xi=(\xi_1,...,\xi_k)\in \mathcal I_n^k$, we have 
\[
(1-\varepsilon_n)\mathbb E[|L^{\pmb z}Q(\pmb \xi)|^2] \le \mathbb E[|L^{\pmb z}Q_{n}(\pmb \xi)|^2]\le (1+\varepsilon_n)\mathbb E[|L^{\pmb z}Q(\pmb \xi)|^2].
\]
\end{enumerate}
If $Q\in \mathcal A_\infty$ and hypothesis \ref{H4} holds for all $k\ge 1$, we write $Q_n\in \mathcal A_\infty^Q(\mathcal I_n,\varepsilon_n)$.
\begin{remark} For $\pmb \xi=(\xi_1,...,\xi_k)\in \mathbb R^k$, let $\Lambda(\pmb \xi)$ and $\Lambda_n(\pmb \xi)$ be the covariance matrices of $(Q(\xi_1),Q'(\xi_1),...,Q(\xi_k), Q'(\xi_k))$ and $(Q_n(\xi_1),Q_n'(\xi_1),...,Q_n(\xi_k), Q_n'(\xi_k))$, respectively. Through elementary computation, for any $\pmb z=(x_1,y_1,...,x_k,y_k)\in \mathbb R^{2k}$, we have
\begin{equation} \label{e1.24}
    \langle \pmb z\Lambda(\pmb \xi),\pmb z\rangle=\mathbb E[|L^{\pmb z}Q(\pmb \xi)|^2].
\end{equation}
Therefore, hypothesis \ref{H4} implies that 
\begin{equation}\label{e1.25}
    (1-\varepsilon_n)\Lambda(\pmb \xi) \le \Lambda_n(\pmb \xi) \le (1+\varepsilon_n)\Lambda(\pmb \xi),\quad \pmb\xi \in \mathcal I_n^k.
\end{equation}
Using the Kac-Rice formula (see, for example, \cite{AW}*{Chapter 3}) and \eqref{e1.25}, we deduce that the correlation functions for the real zeros of $Q_n$ on $\mathcal I_n$ closely approximate those of $Q$. It is worth noting that $Q_n$ may not meet the criteria for a real GAF.

Alternatively, we can establish \eqref{e1.25} by examining the covariance functions of $Q$ and $Q_n$, defined as $r(x,y)=\mathbb E[Q(x) Q(y)]$ and $r_n(x,y)=\mathbb E[Q_n(x) Q_n(y)]$.
In this context,
\[
\Lambda(\pmb\xi) =\begin{pmatrix}
  \Lambda_{ij}  
\end{pmatrix}_{i,j=1}^k,\quad \text{where}\quad \Lambda_{ij}=\begin{pmatrix}
    r(\xi_i,\xi_j)&\frac{\partial r}{\partial y}(\xi_i,\xi_j)\\
    \frac{\partial r}{\partial x}(\xi_i,\xi_j)&\frac{\partial^2r}{\partial x\partial y}(\xi_i,\xi_j)
\end{pmatrix}
\]
and 
\[
\Lambda_n(\pmb\xi) =\begin{pmatrix}
  \Lambda_{ij}^{(n)}  
\end{pmatrix}_{i,j=1}^k,\quad \text{where}\quad \Lambda_{ij}^{(n)}=\begin{pmatrix}
    r_n(\xi_i,\xi_j)&\frac{\partial r_n}{\partial y}(\xi_i,\xi_j)\\
    \frac{\partial r_n}{\partial x}(\xi_i,\xi_j)&\frac{\partial^2r_n}{\partial x\partial y}(\xi_i,\xi_j)
\end{pmatrix}.
\]
Therefore, assuming that there exists a positive constant $\varepsilon_n$, with $\lim_{n\to \infty}\varepsilon_n=0$, such that for all $i,j\in \{0,1\}$ and $(x,y)\in \mathcal I_n^2$, we have
\begin{equation}
    \label{e1.26}
    \frac{\partial^{i+j}r_n}{\partial x^i\partial y^j}(x,y)=(1+O(\varepsilon_n))\frac{\partial^{i+j}r}{\partial x^i\partial y^j}(x,y),
\end{equation}
we can then obtain \eqref{e1.25}. 

Note that when $Q_n$ represents the $n$th partial sum of $Q$, the partial derivatives $\frac{\partial^{i+j}r_n}{\partial x^i\partial y^j}$ converge uniformly to $\frac{\partial^{i+j}r}{\partial x^i\partial y^j}$ on any compact subset of $\mathbb R^2$. In this scenario, it suffices to determine the rate of convergence $\varepsilon_n$ and verify \eqref{e1.26} for the case $\sup_{n}|I_n|=\infty$.
\end{remark}

\begin{theorem}[Asymptotic cumulants] \label{T1.20} 
   Let $k\in \mathbb N\cup\{\infty\}$, $Q\in \mathcal A_k$, and  $Q_n\in \mathcal A_k^Q(\mathcal I_n, \varepsilon_n)$. For a finite interval $I_n\subset \mathcal I_n$, we define $N_{Q_n}(I_n)$ as the number of real zeros of $Q_{n}$ on $I_n$. For any positive integer $m$ such that $m\le k$, one has 
    \begin{equation}\label{e1.27}
        s_m[N_{Q_n}(I_n)] =|I_n|\theta_{m}^Q(|I_n|)+\lambda_{m}^Q(|I_n|)+O(\varepsilon_n|I_n|^m), 
    \end{equation}
    where $\theta_{m}^Q$ and $\lambda_m^Q$ are bounded functions provided in Theorem \ref{T1.14}. Furthermore, as $n\to \infty$, if $|I_n|\to \infty$ and $\varepsilon_n|I_n|^m\to 0$, then 
    \begin{equation}\label{e1.28}
        s_m[N_{Q_n}(I_n)] =|I_n|\theta_{m,\infty}^Q +\lambda_{m,\infty}^Q+o(1).
    \end{equation}
\end{theorem}
\begin{corollary}[Strong law of large numbers] \label{C1.22}
    If $Q_n\in \mathcal A_{k+1}^Q(\mathcal I_n, \varepsilon_n)$ and $I_n\subset \mathcal I_n$ with $|I_n|<\infty$ and $\sum_{n=1}^\infty \max(|I_n|^{-m}, \varepsilon_n^m)<\infty$ for some finite positive constant $m\le k$, then
    \[
        \frac{N_{Q_n}(I_n)}{\mathbb E[N_{Q_n}(I_n)]} \xrightarrow{a.s.} 1 \quad \text{as } n\to \infty.
    \]
\end{corollary}

Consider the Gaussian Weyl polynomial, $W_n(z) = \sum_{j=0}^n \omega_j \frac{z^j}{\sqrt{j!}}$, where $\omega_j$ are i.i.d. normalized Gaussian random variables, serving as the $n$th partial sum of the Weyl series $W(z)$. It is well-known that the majority of real zeros of the Weyl polynomial $W_n(z)$ lie within the interval $[-\sqrt n+n^{1/6}, \sqrt n-n^{1/6}]$ (see, for example, \cite{DV} and \cite{TV}). 

Our focus is on the subinterval $\mathcal I_n=[-\varepsilon_1\sqrt n, \varepsilon_2\sqrt n]$, where  $0<\varepsilon_1\le \varepsilon_2<1$ and $\varepsilon =\max(\varepsilon_1\varepsilon_2e^{1+\varepsilon_1\varepsilon_2}, \varepsilon_2^2 e^{1-\varepsilon_2^2})<1$. Under these conditions, $W_n\in \mathcal A_\infty^W(\mathcal I_n, \varepsilon_n)$ with $\varepsilon_{n}=\varepsilon^n/\sqrt n$. To establish this, we only need to verify the estimates in \eqref{e1.26}, where $r(x,y)=\mathbb E[W(x)W(y)]$ and $r_n(x,y)=\mathbb E[W_n(x)W_n(y)]$. 

Through elementary computation, we find that $r(x,y) = e^{xy}$ and $r_n(x,y) = e_n(xy)$, where $e_n(x)=\sum_{j=0}^n\frac{x^j}{j!}$. Notably, the relationship $\frac{d}{dx}e_n(x) = e_{n-1}(x)$ allows us to estimate $\frac{\partial^{i+j}r_n}{\partial x\partial y}(x,y)$ by focusing solely on the estimation of $e_n(xy)$. Proving \eqref{e1.26} is then simplified to demonstrating
\[
e_n(nt)=e^{nt}(1+O(\varepsilon_n)),\quad t\in [-\varepsilon_1\varepsilon_2,\varepsilon_2^2],
\]
which immediately follows from Buckholtz \cite{Bu} and Stirling's formula.

Therefore, by synthesizing insights from Theorem \ref{T1.20}, Corollary \ref{C1.17}, and Corollary \ref{C1.22}, we establish the large $n$ asymptotic behaviors of cumulants, the CLT, and a strong law of large numbers for the real zeros of $W_n$ on $I_n\subset \mathcal I_n$.
\begin{corollary}[Real zeros of the Gaussian Weyl polynomials] Let $0<\varepsilon_1\le \varepsilon_2<1$ be such that $\varepsilon =\max(\varepsilon_1\varepsilon_2e^{1+\varepsilon_1\varepsilon_2}, \varepsilon_2^2 e^{1-\varepsilon_2^2})<1$. For a nonempty interval $I_n\subset [-\varepsilon_1\sqrt n, \varepsilon_2\sqrt n]$, let $N_{W_n}(I_n)$ denote the number of real zeros of the Gaussian Weyl polynomial $W_n$ of degree $n$. For any positive integer $k$, we have
    \[
  s_k[N_{W_n}(I_n)]=|I_n|\theta_{k}^W(|I_n|)+\lambda_{k}^W(|I_n|)+O\left(\frac{\varepsilon^n}{\sqrt n}|I_n|^k\right),
  \]
  where $\theta_{k}^W$ and $\lambda_{k}^W$ are defined as in Corollary \ref{C1.17}. Consequently, as $n\to \infty$, if $|I_n|\to \infty$, then 
  \[
  s_k[N_{W_n}(I_n)]=|I_n|\theta_{k,\infty}^W+\lambda_{k,\infty}^W+o(1)
  \]
 and hence, $N_{W_n}(I_n)$ satisfies the CLT. Furthermore, if there exists some positive constant $m$ such that $\sum_{n=1}^\infty |I_n|^{-m}<\infty$, then 
  \[
        \frac{N_{W_n}(I_n)}{\frac{1}{\pi}|I_n|} \xrightarrow{a.s.} 1 \quad \text{as } n\to \infty.
  \]
\end{corollary}

Finally, we propose an alternative approach to extending Theorem \ref{T1.14} to encompass more general Gaussian processes. %with non-translation-invariant distributions of real zeros.
\begin{theorem}[Real zeros of composition of functions] \label{T1.24}
    Fix $k\in \mathbb N\cup\{\infty\}$ and let $Q\in \mathcal A_k$. For $U=(-u, u)\subset \mathbb R$, let $\rho: U \to (0,\infty)$ be an integrable function. We set $\varrho(x) = \int_0^x \rho(t) dt$ and introduce a new process $P(x)=Q(\varrho(x))$ defined on $U$. For $(a,b)\subset U$, let $N_P(a,b)$ denote the number of real zeros of $P$ on $(a,b)$. For any positive integer $m$ with $m\le k$, it holds that 
    \begin{equation} \label{e1.29} s_m[N_P(a,b)]=\left(\int_a^b\rho(t)dt\right)\theta_m^Q\left(\int_a^b\rho(t)dt\right)+\lambda_m^Q\left(\int_a^b\rho(t)dt\right),
    \end{equation}
    where $\theta_m^Q$ and $\lambda_m^Q$ are bounded functions defined as in Theorem \ref{T1.14}.
\end{theorem}
\begin{remark} 
It is worth noting that the Gaussian processes $\{P(x): x\in U\}$ discussed in Theorem \ref{T1.24} might not possess analytic properties, and the distribution of the real zeros of $P$ may lack translation invariance. Additionally, if $\{x_j\}$ are the zeros of $P(x)$, then $\{z_j = \varrho(x_j)\}$ are the zeros of $Q(z)$. Therefore, if $\rho$ represents the density of real zeros $x_j$ of $P(x)$, so that
\[
\mathbb E[N_P(a,b)]=\int_a^b\rho(t)dt, 
\]
then $z_j = \varrho(x_j)$, referred to as the straightening of $x_j$, are uniformly distributed on $\varrho(U)$. This observation, coupled with Theorems \ref{T1.20} and \ref{T1.24}, provides a framework for studying the cumulants of real zeros of random functions through the examination of their straightened zero distribution. Exploiting these invariance properties offers substantial advantages through this approach.

Consider, for example, the Gaussian Kac polynomial $P_n(x) = \sum_{j=0}^n \omega_jx^j$ on $(-1,1)$. By the Kac formula \cite{K}, the limit of the density function of the real zeros of $P_n$ is given by  
\[\rho(t)= \frac{1}{\pi(1-t^2)}, \quad t\in (-1,1).\]
For the real zeros $x_j\in (-1,1)$, define the straightening of $x_j$ as
\[
z_j=\varrho(x_j),\quad \text{where}\quad \varrho(x)=\int_0^x\rho(t)dt=\frac{1}{2\pi}\log\left|\frac{1+x}{1-x}\right|.
\]
In the limit as $n\to \infty$, the straightened zeros $z_j$ are uniformly distributed on the real line. Furthermore, it was shown in \cite{BD}*{Theorem 1.2} that the limit $k$-point correlation function $\varrho_k$ of the straightened zeros $z_j=\varrho(x_j)$ of the Kac polynomial is given by
\[
\varrho_k(\pmb \xi)=\frac{1}{2^k}\prod_{1\le i<j \le k}\tanh^4\pi(\xi_i-\xi_j)\int_{\mathbb R^k}|\eta_1\cdots \eta_k|e^{-\frac 12 \langle \pmb \eta \Lambda(\pmb \xi), \pmb \eta\rangle}d\eta_1\cdots d\eta_k,
\]
where $\pmb \xi=(\xi_1,...,\xi_k)$, $\pmb \eta=(\eta_1,...,\eta_k)$, and the matrix $\Lambda(\pmb \xi)$ is defined as 
\[
\Lambda(\pmb \xi)=\begin{pmatrix}
    \frac{1}{\cosh \pi(\xi_i-\xi_j)}
\end{pmatrix}_{i,j=1}^k.
\]
Remarkably, the function $\varrho_k(\pmb \xi)$ is translation-invariant. While it is possible to estimate the cumulants of the number of straightened zeros, we will not delve into this topic further here.
\end{remark}

\begin{remark}
    Extending the insights from Theorems \ref{T1.14}, \ref{T1.20}, and \ref{T1.24} by relaxing the constraints on $Q$ opens up intriguing avenues for exploration. One promising direction involves building upon the techniques delineated in \cites{AL} and \cite{G}.
\end{remark}
\subsection{Organization of the paper and notational conventions}
In Section \ref{S2.1}, we revisit the Kac-Rice formula,  crucial for proving Theorem \ref{T1.1} in Section \ref{S2.2}. Section \ref{S2.3} then utilizes asymptotic expansions of $K_n(a,b)$ and $L_n(a,b)$ to establish Theorem \ref{T1.2}.
Moving to Section \ref{S3.1}, we provide a concise review of correlation and truncated correlation functions, emphasizing their role in computing moments and cumulants. Section \ref{S3.2} offers estimates for the correlation functions of real zeros of real GAFs, playing a pivotal role in proving Theorems \ref{T1.14}, \ref{T1.20}, and \ref{T1.24} in Sections \ref{S3.3}, \ref{S3.4}, and \ref{S3.5}. Section \ref{S3.6} analyzes correlation functions of real zeros of Gaussian elliptic polynomials, contributing to the establishment of Theorem \ref{T1.3} in Section \ref{S3.7}.
Section \ref{S4} focuses on asymptotic normality results, presenting the proof of Theorem \ref{T1.5}.
In Section \ref{S5.1}, we establish the asymptotics of central moments, as stated in Corollary \ref{C1.7}, while the proof of a strong law of large numbers is provided in Section \ref{S5.2}.

In this paper, we employ the standard asymptotic notation $A=O(B)$ or $A\ll B$ to indicate the bound $|A| \le cB$, where $c$ is independent of $B$. For any $I\subset \mathbb R$, we denote $|I|$ as the length of $I$ if $I$ is an interval, or the cardinality of $I$ if $I$ is a finite set. Given $I\subset \mathbb R$ and a positive integer $k$, we define $I^k=I\times \cdots \times I\subset \mathbb R^k$. The constants $c_k$ and $C_k$ may depend on $k$ and can vary across different contexts.
\section{Variance of the number of real zeros} 
Our primary tool for computing the variance is the Kac-Rice formula (see \cite{AW}*{Chapter 3}). It is very general and allows one to obtain an integral formula for the variance of the number of real zeros of a smooth Gaussian process. 
\subsection{The Kac-Rice formula}\label{S2.1} 
Let $\mathcal G=\{G(x), x\in I\}$, $I$ an interval on the real line, be a non-degenerate, centered Gaussian process having $C^1$ paths. We normalize $G$ so that the covariance kernel defined by $r(x,y)=\mathbb E[G(x)G(y)]$ satisfies $r(x,x)=1$. A trivial verification shows that
\begin{align*}
    \mathbb E[G(x)G'(x)]&=\mathbb E[G(y)G'(y)]=0,\\
    \mathbb E[G'(x)G(y)]&=\frac{\partial}{\partial x}r(x,y)=:r_{10}(x,y),\\ \mathbb E[G(x)G'(y)]&=\frac{\partial}{\partial y}r(x,y)=:r_{01}(x,y),\\ \mathbb E[G'(x)G'(y)]&=\frac{\partial^2}{\partial x\partial y}r(x,y)=:r_{11}(x,y).
\end{align*}
Let $\rho_1(x)=\frac{1}{\pi}\sqrt{r_{11}(x,x)}$ and 
\begin{equation}
    \label{e2.1}
    \rho_2(x,y)=\frac{1}{\pi^2}\left(\sqrt{1-\rho^2(x,y)}+\rho(x,y) \arcsin \rho(x,y)\right)\frac{\sigma(x,y)}{\sqrt{1-r^2(x,y)}},
\end{equation}
where 
\begin{align*}
\rho(x,y)&:=\frac{r_{11}(x,y)+\frac{r(x,y)r_{10}(x,y)r_{01}(x,y)}{1-r^2(x,y)}}{\sqrt{\left(r_{11}(x,x)-\frac{r_{10}^2(x,y)}{1-r^2(x,y)}\right)\left(r_{11}(y,y)-\frac{r_{01}^2(x,y)}{1-r^2(x,y)}\right)}},\\
\sigma(x,y)&:=\sqrt{\left(r_{11}(x,x)-\frac{r_{10}^2(x,y)}{1-r^2(x,y)}\right)\left(r_{11}(y,y)-\frac{r_{01}^2(x,y)}{1-r^2(x,y)}\right)}.
\end{align*}
The following lemma is standard (see, for example, \cite{BD} and  \cite{LP}), and we include a proof here for the convenience of the reader.
\begin{lemma}
\label{L2.1}
Let $N(I)$ denote the number of real zeros of $G$ on $I$. One has
\begin{equation}
    \label{e2.2}
    \Var[N(I)]=\int_{I^2}\left[\rho_2(x,y)-\rho_1(x)\rho_1(y)\right]dydx +\int_{I} \rho_1(x)dx.
\end{equation}
\end{lemma}
\begin{proof}
We first write
\begin{equation}
    \label{e2.3}
    \Var[N(I)]=\mathbb E[N(I)(N(I)-1)]-( \mathbb E[N(I)])^2+ \mathbb E[N(I)].
\end{equation}
By the Kac-Rice formula, 
\begin{equation}
    \label{e2.4}
    \mathbb E[N(I)]=\int_{I}\rho_1(x)dx.
\end{equation}
As a consequence, 
\begin{equation}
    \label{e2.5}
(\mathbb E[N(I)])^2=\int_{I^2}\rho_1(x)\rho_1(y)dydx.
\end{equation}
The Rice formula for the second factorial moment now asserts that
\[
    \mathbb E[N(I)(N(I)-1)]=\int_{I^2}\mathbb E[|G'(x)G'(y)|\mid G(x)=0, G(y)=0]p_{x,y}(0,0)dydx,
\]
where $p_{x,y}$ is the joint density of $(G(x),G(y))$, so 
\[
p_{x,y}(0,0)=\frac{1}{2\pi \sqrt{1-r^2(x,y)}}.
\]
Observe that conditionally on $\mathcal C:=\{G(x)=0, G(y)=0\}$, $G'(x)$ and $G'(y)$ have a joint Gaussian distribution. Utilizing regression formulas, we derive the following expressions for expectations, variances, and covariances:
\begin{align*}
    \mathbb E[G'(x)\mid \mathcal C]&=\mathbb E[G'(y)\mid \mathcal C]=0,\\
    \Var[G'(x)\mid \mathcal C]&=r_{11}(x,x)-\frac{r_{10}^2(x,y)}{1-r^2(x,y)},\\
    \Var[G'(y)\mid \mathcal C]&=r_{11}(y,y)-\frac{r_{01}^2(x,y)}{1-r^2(x,y)},\\
    \mathbb E[G'(x)G'(y)\mid \mathcal C]&=r_{11}(x,y)+\frac{r(x,y)r_{10}(x,y)r_{01}(x,y)}{1-r^2(x,y)}.
\end{align*}
Therefore, 
\begin{align*}
\frac{ \mathbb E[G'(x)G'(y)\mid \mathcal C]}{\sqrt{\Var[G'(x)\mid \mathcal C] \Var[G'(y)\mid \mathcal C]}}=\rho(x,y)\quad \text{and}\quad 
\sqrt{\Var[G'(x)\mid \mathcal C] \Var[G'(y)\mid \mathcal C]}=\sigma(x,y).
\end{align*}
Applying \cite{LW}*{Corollary 3.1} leads to   
\[
  \mathbb E[|G'(x)G'(y)| \mid \mathcal C]=\frac{2}{\pi}\left(\sqrt{1-\rho^2(x,y)}+\rho(x,y) \arcsin \rho(x,y)\right)\sigma(x,y).
\]
Combining this with \eqref{e2.1}, we derive
\begin{equation}
    \label{e2.6}
  \mathbb E[N(I)(N(I)-1)]=\int_{I^2} \rho_2(x,y)dydx.
\end{equation}
Substituting \eqref{e2.4}, \eqref{e2.5}, and \eqref{e2.6} into \eqref{e2.3}, we  obtain \eqref{e2.2} as required. 
\end{proof}
\begin{remark}
It is worth suggesting that, by leveraging Lemma \ref{L2.1} and conducting a thorough analysis, we can deduce the leading terms in the asymptotics of $\Var[N_n(a,b)]$ for random polynomials described in Section \ref{S1.1}, with $\omega$ being Gaussian.
\end{remark}
\subsection{Proof of Theorem \ref{T1.1}} \label{S2.2} 
 Let us now apply Lemma \ref{L2.1} to the Gaussian elliptic polynomial $P_n(x)$. Using the binomial theorem we see that the covariance function of $P_n(x)/\sqrt{\Var[P_n(x)]}$ is given by 
\[
r(x,y)=\frac{\mathbb E[P_n(x)P_n(y)]}{\sqrt{\Var[P_n(x)]\Var[P_n(y)]}}=\frac{(1+xy)^n}{\sqrt{(1+x^2)^n(1+y^2)^n}}.
\]
A straightforward calculation reveals that 
\begin{align*}
    r_{10}(x,y)&=nr(x,y)\frac{(y-x)}{(1+xy)(1+x^2)},\\
    r_{01}(x,y)&=nr(x,y)\frac{(x-y)}{(1+xy)(1+y^2)},\\
    r_{11}(x,y)&=nr(x,y)\left(\frac{1}{(1+xy)^2}-\frac{n(x-y)^2}{(1+xy)^2(1+x^2)(1+y^2)}\right).
\end{align*}
Using  $\alpha(x,y)=(y-x)/(1+xy)$ and  $(1+x^2)(1+y^2)=(1+xy)^2+(x-y)^2$, expression \eqref{e2.1} can be reformulated as
\[
\rho_2(x,y)=\frac{1}{\pi^2}\frac{n}{(1+x^2)(1+y^2)}(F_n(\alpha(x,y))+1).
\]
Together with 
\[\rho_1(x)\rho_1(y)=\frac{1}{\pi^2}\frac{n}{(1+x^2)(1+y^2)},\] 
we deduce from \eqref{e2.2} that
\begin{equation}
    \label{e2.7}
    \Var[N_n(a,b)]=I_{n,2}(a,b)+\mathbb E[N_n(a,b)],
\end{equation}
where 
\begin{equation}
    \label{e2.8}
I_{n,2}(a,b):=\frac{1}{\pi^2}\int_a^b\int_a^b \frac{n}{(1+x^2)(1+y^2)}F_n(\alpha(x,y))dydx.
\end{equation}
The proof of Theorem \ref{T1.1} now falls naturally into three following lemmas.
\begin{lemma} \label{L2.3}
If $\alpha(a,b)>0$, then 
\begin{equation}
\label{e2.9}
I_{n,2}(a,b)=K_n(a,b)\mathbb E[N_n(a,b)]-L_n(a,b).
\end{equation}
This gives \eqref{e1.4} when substituted in \eqref{e2.7}.
\end{lemma}
\begin{proof}
If $\alpha(a,b)>0$, then $ab>-1$ and hence, $1+xy\ne 0$ for all $a<x,y<b$. Fix $x\in (a,b)$ and make the change of variables $s=\sqrt n\alpha(x,y)$ for the integral $\int_a^b \frac{\sqrt n  F_n(\alpha(x,y))}{1+y^2}dy$, we see that
\[
    I_{n,2}(a,b)=\frac{1}{\pi^2}\int_a^b \frac{\sqrt n dx}{1+x^2} \int_{\sqrt n\alpha(x,a)}^{\sqrt n\alpha(x,b)} \frac{F_n(s/\sqrt n)}{1+s^2/n}ds.
\]
Using Fubini's theorem, the identity
\[
\arctan \alpha(s/\sqrt n,a)=\arctan a-\arctan (s/\sqrt n),
\]
and \eqref{e1.1}, we find that 
\begin{align*}
  \int_a^b \frac{\sqrt ndx}{1+x^2} \int_{\sqrt n\alpha(x,a)}^{0} \frac{F_n(s/\sqrt n)}{1+s^2/n}ds&=\int_{\sqrt n\alpha(b,a)}^0\frac{F_n(s/\sqrt n)}{1+s^2/n}ds\int_{\alpha(s/\sqrt n,a)}^b \frac{\sqrt ndx}{1+x^2} \\
    &=\frac{\pi^2}{2} K_n(a,b)\mathbb E[N_n(a,b)]-\frac{\pi^2}{2} L_n(a,b).
\end{align*}
Similarly, 
\begin{align*}
    \int_a^b \frac{\sqrt n dx}{1+x^2} \int_0^{\sqrt n\alpha(x,b)} \frac{F_n(s/\sqrt n)}{1+s^2/n}ds=\frac{\pi^2}{2} K_n(a,b)\mathbb E[N_n(a,b)]-\frac{\pi^2}{2} L_n(a,b).
\end{align*}
Combining these we obtain \eqref{e2.9} as required.
\end{proof}
\begin{lemma}
Equation \eqref{e1.5} follows from the fact that
\begin{equation}
    \label{e2.10}
    I_{n,2}(\mathbb R)= K_n \sqrt n.
\end{equation}
\end{lemma}
\begin{proof}
Starting from \eqref{e2.8}, we have
\begin{equation}
    \label{e2.11}
    I_{n,2}(\mathbb R)=\frac{1}{\pi^2}\int_{-\infty}^\infty dx\int_{-\infty}^\infty \frac{ nF_n(\alpha(x,y))}{(1+x^2)(1+y^2)}dy.
\end{equation}
Fix $x\in (-\infty,0)$ and substitute $s=\sqrt n\alpha(x,y)$, we see that
\begin{align*}
\frac{1}{\pi^2}\int_{-\infty}^0dx \int_{-\infty}^\infty \frac{nF_n(\alpha(x,y))}{(1+x^2)(1+y^2)}dy=\frac{1}{\pi^2}\int_{-\infty}^0\frac{\sqrt n}{1+x^2}dx\int_{-\infty}^{\infty}\frac{F_n(s/\sqrt n)}{1+s^2/n}ds=\frac{\sqrt n}{2}K_n.
\end{align*}
Similarly, 
\[
\frac{1}{\pi^2}\int_0^{\infty}dx\int_{-\infty}^\infty \frac{nF_n(\alpha(x,y))}{(1+x^2)(1+y^2)}dy= \frac{\sqrt n}{2}K_n.
\]
Substituting these results into \eqref{e2.11} yields \eqref{e2.10} as claimed.
\end{proof}
\begin{lemma}
If $\alpha(a,b)<0$, then 
\begin{equation}
    \label{e2.12}
   I_{n,2}(a,b)=K_n\mathbb E[N_n(a,b)] +\left(K_n-K_n(a,b)\right)\left(\mathbb E[N_n(a,b)]-\sqrt n\right)-L_n(a,b),
\end{equation}
which implies \eqref{e1.6} when combined with \eqref{e2.7}. 
\end{lemma}
\begin{proof} 
If $\alpha(a,b)<0$, then $ab<-1$ and either $a$ or $b$ is finite.  Thus, $a<-1/b<-1/a<b$ and the equation $1+xy=0$ has a solution $(x,-1/x)$ only if $x\in (a,-1/b)\cup (-1/a,b)$. Write
\begin{align*}
    I_{n,2}(a,b)&=\frac{1}{\pi^2}\int_a^{-1/b} \frac{\sqrt n dx}{1+x^2}\int_{a}^{-1/x} \frac{\sqrt nF_n(\alpha(x,y))}{1+y^2}dy\\
    &+\frac{1}{\pi^2}\int_a^{-1/b} \frac{\sqrt n dx}{1+x^2} \int_{-1/x}^{b}\frac{\sqrt nF_n(\alpha(x,y))}{1+y^2}dy\\
    &+\frac{1}{\pi^2}\int_{-1/b}^{-1/a} \frac{\sqrt n dx}{1+x^2}\int_a^b \frac{\sqrt nF_n(\alpha(x,y))}{1+y^2}dy\\
    &+\frac{1}{\pi^2}\int_{-1/a}^{b} \frac{\sqrt n dx}{1+x^2} \int_{a}^{-1/x} \frac{\sqrt nF_n(\alpha(x,y))}{1+y^2}dy\\
    &+\frac{1}{\pi^2}\int_{-1/a}^{b} \frac{\sqrt n dx}{1+x^2}\int_{-1/x}^{b}\frac{\sqrt nF_n(\alpha(x,y))}{1+y^2}dy.
\end{align*}
We now proceed analogously to the proof of Lemma \ref{L2.3}. 
Using the substitution $s=\sqrt n\alpha(x,y)$, Fubini's theorem, and the facts that
\[
\arctan \alpha(x,y)=\begin{cases}
\arctan y-\arctan x&\text{if}\quad 1+xy>0,\\
\arctan y-\arctan x-\pi&\text{if}\quad 1+xy<0 \text{ and } y>0,\\
\arctan y-\arctan x+\pi&\text{if}\quad 1+xy<0 \text{ and } y<0,
\end{cases}
\]
and 
\[
\arctan x+\arctan(1/x)=\begin{cases}
+\pi/2&\text{if}\quad x>0,\\
-\pi/2&\text{if}\quad x<0,
\end{cases}
\]
we conclude that 
\begin{align*}
    I_{n,2}(a,b)&=\frac 12(K_n+K_n(-1/a,b))(\E[N_n(a,b)]-\sqrt n/2)-\frac 12 L_n(-1/a,b)\\
    &+\frac 12 (K_n-K_n(a,b))(\E[N_n(a,b)]-\sqrt n)+\frac 12(L_n-L_n(a,b))\\
    &+\frac{\sqrt n}{2} K_n -K_n(-1/a,b)(\E[ N_n(a,b)]-\sqrt n/2)-L_n+L_n(-1/a,b)\\
    &+\frac 12 (K_n-K_n(a,b))(\E[N_n(a,b)]-\sqrt n)+\frac 12(L_n-L_n(a,b))\\
    &+\frac 12(K_n+K_n(-1/a,b))(\E[N_n(a,b)]-\sqrt n/2)-\frac 12 L_n(-1/a,b)\\
    &=K_n\E[N_n(a,b)]+(K_n-K_n(a,b))(\E[N_n(a,b)]-\sqrt n)-L_n(a,b),
 \end{align*}
which gives \eqref{e2.12}.
\end{proof}
\subsection{Proof of Theorem \ref{T1.2}} \label{S2.3} 
Before proving Theorem \ref{T1.2}, let us introduce some lemmas that will be crucial to the proof. Notably, Theorem \ref{T1.1} enables us to derive the large $n$ expansion of $\Var[N_n(a,b)]$ using the expansions of $K_n(a,b)$ and $L_n(a,b)$. In order to expand $K_n(a,b)$ and $L_n(a,b)$, we initially demonstrate that the functions $\frac{F_n(s/\sqrt n)}{1+s^2/n}$ and $\frac{F_n(s/\sqrt n)}{1+s^2/n}\sqrt n\arctan(s/\sqrt n)$ can be transformed into series consisting of terms that are powers of $1/n$.
\begin{lemma}\label{L2.6} Given  $0<c_n<\sqrt n$, one has 
\begin{equation}
\label{e2.13}
    \frac{F_n(s/\sqrt n)}{1+s^2/n}=\sum_{k=0}^\infty \frac{f_k(s)}{n^k} \quad \text{uniformly for}\; s\in [0,c_n],
\end{equation}
where $f_k(s)$ have continuous extensions to $[0,\infty)$ such that, as $s\to 0$,
\begin{equation}
    \label{e2.14}
f_k(s)= \begin{cases}
-1+\frac{\pi}{4} s+O(s^3)&\text{if}\quad k=0,\\
-\frac{\pi}{4} s+s^2+O(s^3)&\text{if}\quad k=1,\\
O(s^3)&\text{if}\quad k\ge 2,
\end{cases}
\end{equation}
and, as $s\to \infty$,
\begin{equation}
    \label{e2.15}
f_k(s)=\frac{1}{2^{k+1}k!}s^{4k+4}e^{-s^2}+ O(s^{4k+2}e^{-s^2}),\quad k\ge 0.
\end{equation}
\end{lemma}
\begin{proof}
The proof will be divided into four steps. 
\begin{step} 
Expand $\Delta_n(s/\sqrt n)$.
\end{step}
Observe that 
\[
 \Delta_n(s/\sqrt n)=\left(1+\frac{s^2}{n}\right)^{-n/2}\frac{(1+s^2/n)[1-(1+s^2/n)^{-n}]-s^2}{1-(1+s^2)(1+s^2/n)^{-n}}.
\]
If $s\in [0,c_n]$, then $s^2/n\le c_n^2/n<1$. Hence, for any $c>0$, 
\[
-cn\log(1+s^2/n)=-cs^2+c\sum_{k=1}^\infty \frac{q_k(s)}{k!}\frac{1}{n^k} \quad \text{uniformly for $s\in [0,c_n]$},
\]
in which $q_k(s)=(-s^2)^{k+1}k!/(k+1)$. But then
\begin{equation}
\label{e2.16}
    \left(1+\frac{s^2}{n}\right)^{-cn}=e^{-cs^2}\left(1+\sum_{k=1}^\infty \frac{e_{c,k}(s)}{n^k}\right),
\end{equation}
where 
\begin{equation}
    \label{e2.17}
    e_{c,k}(s)=\frac{1}{k!}\sum_{j=1}^k c^jB_{k,j}(q_1(s),...,q_{k-j+1}(s))
\end{equation}
with $B_{k,j}$ denoting the exponential partial Bell polynomials (see \cite{C}*{\S3.3}). Explicit formulas for these polynomials are as follows
\begin{equation}
    \label{e2.18}
B_{k,j}(q_1(s),...,q_{k-j+1}(s))=\sum \frac{k!}{m_1!\cdots m_{k-j+1}!}\prod_{r=1}^{k-j+1}\left(\frac{(-s^2)^{r+1}}{r+1}\right)^{m_r},
\end{equation}
where the sum is over all solutions in non-negative integers of the equations
\begin{align*}
    m_1+2m_2+\cdots+(k-j+1)m_{k-j+1}&=k,\\
    m_1+m_2+\cdots+m_{k-j+1}&=j.
\end{align*} 
Combining \eqref{e2.17} with \eqref{e2.18} yields 
\begin{equation}
    \label{e2.19}
e_{c,k}(s)= \begin{cases}
\frac{c(-1)^{k+1}}{k+1}s^{2k+2}+O(s^{2k+4})&\text{as}\quad s\to 0,\\
\frac{c^k}{k!2^k}s^{4k}+O(s^{4k-2})&\text{as}\quad s\to \infty.
\end{cases}
\end{equation}
With \eqref{e2.16} and a bit of work, we can write
\[
\left(1+\frac{s^2}{n}\right)^{-n/2}\left[(1+s^2/n)[1-(1+s^2/n)^{-n}]-s^2\right]=\sum_{k=0}^\infty \frac{u_k(s)}{n^k},
\]
in which 
\begin{align*}
u_0(s)=e^{-s^2/2}\left(1-s^2-e^{-s^2}\right),\quad 
u_1(s)=e^{-s^2/2}\left[s^2+\frac{s^4}{4}-\frac{s^6}{4}-e^{-s^2}\left(s^2+\frac{3s^4}{4}\right)\right],
\end{align*}
and, for $k\ge 2$, 
\begin{equation*}
u_k(s)=e^{-s^2/2}\left[s^2e_{1/2,k-1}(s)+(1-s^2)e_{1/2,k}(s)\right]-e^{-3s^2/2}\left[e_{3/2,k}(s)+s^2e_{3/2,k-1}(s)\right].
\end{equation*}
Notice that
\begin{align}
\label{e2.20}    u_0(s)&= \begin{cases}
-\frac{1}{2}s^4+O(s^6)&\text{as}\quad s\to 0,\\
-s^2e^{-s^2/2}+O(e^{-s^2/2})&\text{as}\quad s\to \infty,
\end{cases} \\
\label{e2.21} u_1(s)&= \begin{cases}
\frac{1}{2}s^4+O(s^6)&\text{as}\quad s\to 0,\\
-\frac{1}{4}s^6e^{-s^2/2}+O(s^4e^{-s^2/2})&\text{as}\quad s\to \infty,
\end{cases}
\end{align}
and, by \eqref{e2.19}, for $k\ge 2$,
\begin{equation}
    \label{e2.22}
u_k(s)=\begin{cases}
O(s^{2k+2})&\text{as}\quad s\to 0,\\
-\frac{1}{4^kk!}s^{4k+2}e^{-s^2/2}+O(s^{4k}e^{-s^2/2})&\text{as}\quad s\to \infty.
\end{cases}
\end{equation}
For $s\in (0,c_n]$, one has $0<(1+s^2)(1+s^2/n)^{-n}<1$, and so
\begin{align*}
    \frac{1}{1-(1+s^2)(1+s^2/n)^{-n}}&=1+\sum_{m=1}^\infty (1+s^2)^m\left(1+\frac{s^2}{n}\right)^{-mn}\\
    &=1+\sum_{m=1}^\infty (1+s^2)^me^{-ms^2}\left(1+\sum_{k=1}^\infty \frac{e_{m,k}(s)}{n^k}\right)\\
    &=: v_0(s)+\sum_{k=1}^\infty \frac{v_k(s)}{n^k}.
\end{align*}
Clearly, 
\[
v_0(s)=1+\sum_{m=1}^\infty(1+s^2)^me^{-ms^2}=\frac{1}{1-(1+s^2)e^{-s^2}},
\]
which gives 
\begin{equation}
    \label{e2.23}
    v_0(s)= \begin{cases}
2s^{-4}+O(s^{-2})&\text{as}\quad s\to 0,\\
1+O(s^2e^{-s^2})&\text{as}\quad s\to \infty.
\end{cases}
\end{equation}
For $k\ge 1$, $v_k(s)$ can be expressed in terms of the polylogarithm functions (see \cite{L}) defined by 
\[
\Li_{j}(z):=\sum_{m=1}^\infty \frac{z^m}{m^j}.
\]
Indeed, by definition of $v_k(s)$ and \eqref{e2.17}, 
\begin{align*}
 v_k(s)&=\sum_{m=1}^\infty(1+s^2)^me^{-ms^2}e_{m,k}(s)\\
 &= \sum_{m=1}^\infty(1+s^2)^me^{-ms^2}\frac{1}{k!}\sum_{j=1}^k m^jB_{k,j}(q_1(s),...,q_{k-j+1}(s))\\
 &=\frac{1}{k!}\sum_{j=1}^kB_{k,j}(q_1(s),...,q_{k-j+1}(s))\Li_{-j}((1+s^2)e^{-s^2}).
\end{align*}
In particular, 
\[
v_1(s)=B_{1,1}(q_1(s))\Li_{-1}((1+s^2)e^{-s^2})=\frac{s^4}{2}\frac{(1+s^2)e^{-s^2}}{(1-(1+s^2)e^{-s^2})^2}.
\]
Since, for $1\le j\le k$, 
\[
\Li_{-j}((1+s^2)e^{-s^2})\sim \begin{cases}
j!2^{j+1}s^{-4(j+1)}&\text{as}\quad s\to 0,\\
(1+s^2)e^{-s^2}&\text{as}\quad s\to \infty,
\end{cases}
\]
it follows that 
\begin{equation}
    \label{e2.24}
v_k(s)= \begin{cases}
2s^{-4}+O(s^{-2})&\text{as}\quad s\to 0,\\
\frac{1}{2^kk!}s^{4k+2}e^{-s^2}+O(s^{4k}e^{-s^2})&\text{as}\quad s\to \infty.
\end{cases}
\end{equation}
Next, by the Cauchy product, for $s\in (0,c_n]$,
\[
\left(\sum_{k=0}^\infty \frac{u_k(s)}{n^k}\right)\left(\sum_{k=0}^\infty \frac{v_k(s)}{n^k}\right)=\sum_{k=0}^\infty \frac{\delta_k(s)}{n^k},
\]
where 
\[
\delta_k(s):=\sum_{j=0}^ku_j(s)v_{k-j}(s),\quad k\ge 0.
\]
In particular, 
\begin{align*}
    \delta_0(s)&=\frac{e^{-s^2/2}(1-s^2-e^{-s^2})}{1-(1+s^2)e^{-s^2}},\\
    \delta_1(s)&=\frac{s^4e^{-s^2/2}}{2}\frac{(1-s^2-e^{-s^2})(1+s^2)e^{-s^2}}{(1-(1+s^2)e^{-s^2})^2}\\
    &+\frac{e^{-s^2/2}}{1-(1+s^2)e^{-s^2}}
    \left[s^2+\frac{s^4}{4}-\frac{s^6}{4}-e^{-s^2}\left(s^2+\frac{3s^4}{4}\right)\right].
\end{align*}   
We check at once that 
\begin{equation}
\label{e2.25}
\delta_0(s)= \begin{cases}
-1+O(s^2)&\text{as}\quad s\to 0,\\
-s^2e^{-s^2/2}+O(e^{-s^2/2})&\text{as}\quad s\to \infty.
\end{cases}
\end{equation}
Using \eqref{e2.22}, \eqref{e2.23}, and \eqref{e2.24}, we see that, for $k\ge 1$, 
\[
\delta_k(s)=\begin{cases} u_0(s)v_k(s)+u_1(s)v_{k-1}(s)+O(s^2)&\text{as}\quad s\to 0,\\
u_k(s)v_0(s)+O(s^{4k+4}e^{-3s^2/2})&\text{as}\quad s\to \infty.
\end{cases}
\]
Together with \eqref{e2.20} and \eqref{e2.21}, we arrive at
\begin{equation}
    \label{e2.26}
\delta_k(s)= \begin{cases}
O(s^2)&\text{as}\quad s\to 0,\\
-\frac{1}{4^kk!}s^{4k+2}e^{-s^2/2}+O(s^{4k+4}e^{-3s^2/2})&\text{as}\quad s\to \infty.
\end{cases}
\end{equation}
This implies that the functions $\delta_k(s)$ extend by continuity at $s=0$. Hence, 
\begin{equation}
    \label{e2.27}
    \Delta_n(s/\sqrt n)=\sum_{k=0}^\infty \frac{\delta_k(s)}{n^k} \quad \text{uniformly for}\; s\in [0,c_n].
\end{equation}
\begin{step}
Expand $h(\Delta_n(s/\sqrt n))$, where $h(x):=\sqrt{1-x^2}+x\arcsin x$.
\end{step}
For $s>0$, we see that $-1<\Delta_n(s/\sqrt n)<1$. Thus, by \eqref{e2.27} and Fa\`{a} di Bruno’s formula (see \cite{C}*{\S3.4}),
\begin{align*}
    h(\Delta_n(s/\sqrt n))&=h(0)+\sum_{m=1}^\infty \frac{h^{(m)}(0)}{m!}(\Delta_n(s/\sqrt n))^m\\
    &=1+\sum_{m=1}^\infty \frac{h^{(m)}(0)}{m!}\left(\sum_{k=0}^\infty \frac{\delta_k(s)}{n^k}\right)^m=:z_0(s)+\sum_{k=1}^\infty \frac{z_k(s)}{n^k},
\end{align*}
where 
\[
z_0(s)=1+\sum_{m=1}^\infty \frac{h^{(m)}(0)}{m!}\delta^m_{0}(s)=h(\delta_0(s)),
\]
and, for $k\ge 1$, 
\begin{align*}
    z_k(s)=\frac{1}{k!}\sum_{j=1}^kh^{(j)}(\delta_0(s))B_{k,j}(1! \delta_1(s),...,(k-j+1)!\delta_{k-j+1}(s)).
\end{align*}
In particular, 
\[
z_1(s)=B_{1,1}(\delta_1(s))h'(\delta_0(s))=\delta_1(s) \arcsin(\delta_0(s)).
\]
By \eqref{e2.25} and the asymptotic behaviors of $h(x)$ as $x\to 0$ and as $x\to -1^+$, 
\begin{equation}
    \label{e2.28}
    z_0(s)= \begin{cases}
\pi/2+O(s^2)&\text{as}\quad s\to 0,\\
1+\frac 12 s^4e^{-s^2}+O(s^2e^{-s^2})&\text{as}\quad s\to \infty.
\end{cases}
\end{equation}
Note that $h'(x)=\arcsin x$, so
\[
   h'(x)=\begin{cases}
O(1)&\text{as}\quad x\to -1^+,\\
x+O(x^3)&\text{as}\quad x\to 0,
\end{cases}
\]
and, for $j\ge 2$, 
\[
 h^{(j)}(x)=\begin{cases}
O((1-x^2)^{(3-2j)/2})&\text{as}\quad x\to -1^+,\\
\frac{1+(-1)^{j}}{2}+O(x^2)&\text{as}\quad x\to 0.
\end{cases}
\]
Together with \eqref{e2.25}, we see  that
\[
h'(\delta_0(s))=\begin{cases}
O(1)&\text{as}\quad s\to 0,\\
-s^2e^{-s^2/2}+O(s^6e^{-3s^2/2})&\text{as}\quad s\to \infty,
\end{cases}
\]
and, for $j\ge 2$, 
\[
h^{(j)}(\delta_0(s))=\begin{cases}
O(s^{3-2j})&\text{as}\quad s\to 0,\\
\frac{1+(-1)^{j}}{2}+O(s^4e^{-s^2})&\text{as}\quad s\to \infty.
\end{cases}
\]
Thus, using \eqref{e2.26} and the fact that $1+B_{k,2}(1,...,1)=2^{k-1},$ we get
\begin{equation}
    \label{e2.29}
z_k(s)= \begin{cases}
O(s^2)&\text{as}\quad s\to 0,\\
\frac{1}{2^{k+1}k!}s^{4k+4}e^{-s^2}+O(s^{4k+2}e^{-s^2})&\text{as}\quad s\to \infty.
\end{cases}
\end{equation}
Summarizing, we have 
\begin{equation}
    \label{e2.30}
    h(\Delta_n(s/\sqrt n))=\sum_{k=0}^\infty \frac{z_k(s)}{n^k}\quad \text{uniformly for}\; s\in [0,c_n].
\end{equation}
\begin{step}
Expand $\Gamma_n(s/\sqrt n)$.
\end{step}
For this purpose, let us consider the function $x\mapsto g_s(x)$ given by
\[
g_s(x)=\frac{1-(1+s^2)x}{(1-x)^{3/2}},\quad x\in (-1,1).
\]
For $s>0$, we have $0<(1+s^2/n)^{-n}<1$ and 
\[
\Gamma_n(s/\sqrt n)=\frac{1-(1+s^2)(1+s^2/n)^{-n}}{[1-(1+s^2/n)^{-n}]^{3/2}}=g_s\left((1+s^2/n)^{-n}\right).
\]
Therefore, 
\begin{align*}
    \Gamma_n(s/\sqrt n)&=g_s(0)+\sum_{m=1}^\infty \frac{g_s^{(m)}(0)}{m!}\left(1+\frac{s^2}{n}\right)^{-mn}\\
    &=1+\sum_{m=1}^\infty \frac{g_s^{(m)}(0)}{m!}e^{-ms^2}\sum_{k=0}^\infty \frac{e_{m,k}(s)}{n^k}\\
    &=1+\sum_{m=1}^\infty  \frac{g_s^{(m)}(0)}{m!}e^{-ms^2}+\sum_{k=1}^\infty \left(\sum_{m=1}^\infty  \frac{g_s^{(m)}(0)}{m!}e^{-ms^2}e_{m,k}(s)\right)\frac{1}{n^k}\\
    &=:\gamma_0(s)+\sum_{k=1}^\infty \frac{\gamma_k(s)}{n^k}.
\end{align*}
In particular, 
\begin{align*}
\gamma_0(s)&=1+\sum_{m=1}^\infty  \frac{g_s^{(m)}(0)}{m!}e^{-ms^2}=g_s(e^{-s^2})=\frac{1-(1+s^2)e^{-s^2}}{(1-e^{-s^2})^{3/2}}.
\end{align*}
To determine $\gamma_k(s)$, for $k\ge 1$, we utilize the following identity (see \cite{C}*{\S5.1}), for $m\ge 1$ and $1\le j\le k$,
\[
m^j=\sum_{r=1}^jS(j,r)(m)_r,
\]
where $S(j,r)$ are the Stirling numbers of the second kind, and $(m)_r$ are the falling factorials defined by $(m)_r=m(m-1)\cdots (m-r+1)$. This implies
\begin{align*}
 \gamma_k(s)&=\sum_{m=1}^\infty  \frac{g_s^{(m)}(0)}{m!}e^{-ms^2}e_{m,k}(s)\\
  &=\sum_{m=1}^\infty  \frac{g_s^{(m)}(0)}{m!}e^{-ms^2}\frac{1}{k!}\sum_{j=1}^k m^j B_{k,j}(q_1(s),...,q_{k-j+1}(s))\\
   &=\frac{1}{k!}\sum_{j=1}^k B_{k,j}(q_1(s),...,q_{k-j+1}(s))\sum_{r=1}^jS(j,r)e^{-rs^2}g_s^{(r)}(e^{-s^2}).
\end{align*}
In particular, 
\begin{align*}
    \gamma_1(s)&=B_{1,1}(q_1(s))S(1,1)e^{-s^2}g_s'(e^{-s^2})=\frac{s^4e^{-s^2}(1-2s^2-(1+s^2)e^{-s^2})}{4(1-e^{-s^2})^{5/2}}.
\end{align*}
A trivial verification shows that 
\begin{align}
 \label{e2.31}    \gamma_0(s)&=\begin{cases} \frac 12 s+O(s^3)&\text{as}\quad s\to 0,\\
    1-s^2e^{-s^2}+O(e^{-s^2})&\text{as}\quad s\to \infty,
    \end{cases}\\
  \label{e2.32}   \gamma_1(s)&=\begin{cases} -\frac 12 s+O(s^3)&\text{as}\quad s\to 0,\\
    -\frac 12 s^6e^{-s^2}+O(s^4e^{-s^2})&\text{as}\quad s\to \infty.
    \end{cases}
\end{align}
Notice that 
\begin{align*}
    g_s(e^{-s^2})&=\frac{1-(1+s^2)e^{-s^2}}{(1-e^{-s^2})^{3/2}}= \begin{cases} \frac 12 s+O(s^3)&\text{as}\quad s\to 0,\\
    1+O(s^2e^{-s^2})&\text{as}\quad s\to \infty,
    \end{cases}\\
    g_s^{(r)}(e^{-s^2})&=\frac{2^r}{(2r+1)!!}\frac{1-(1+s^2)e^{-s^2}}{(1-e^{-s^2})^{(2r+3)/2}}-\frac{r2^{r-1}}{(2r-1)!!}\frac{1+s^2}{(1-e^{-s^2})^{(2r+1)/2}}\\
    &= \begin{cases} -\frac{r2^{r-1}}{(2r-1)!!}s^{-(2r+1)}+O(s^{-(2r-1)})&\text{as}\quad s\to 0,\\
    -\frac{r2^{r-1}}{(2r-1)!!}s^2+O(1)&\text{as}\quad s\to \infty.
    \end{cases}
\end{align*}
Combining with \eqref{e2.18} yields, for $k\ge 2$, 
\begin{equation}
     \label{e2.33}  
 \gamma_k(s)=\begin{cases} O(s^{2k-1})&\text{as}\quad s\to 0,\\
    -\frac{1}{2^kk!}s^{4k+2}e^{-s^2}+O(s^{4k}e^{-s^2})&\text{as}\quad s\to \infty.
    \end{cases}    
\end{equation}
Therefore, 
\begin{equation}
    \label{e2.34}
    \Gamma_n(s/\sqrt n)=\sum_{k=0}^\infty \frac{\gamma_k(s)}{n^k}\quad \text{uniformly for}\; s\in [0,c_n].
\end{equation}
\begin{step}
Expand $\frac{F_n(s/\sqrt n)}{1+s^2/n}$.
\end{step}
Combining \eqref{e2.30} with \eqref{e2.34}, we have
\begin{equation}
    \label{e2.35}
F_n(s/\sqrt n)=\left(\sum_{k=0}^\infty \frac{z_k(s)}{n^k}\right)\left(\sum_{k=0}^\infty \frac{\gamma_k(s)}{n^k}\right)-1=\sum_{k=0}^\infty \frac{a_k(s)}{n^k}
\end{equation}
uniformly for $s\in [0,c_n]$, 
in which 
\begin{align*}
    a_0(s)=z_0(s)\gamma_0(s)-1 \quad \text{and}\quad 
    a_k(s)=\sum_{j=0}^kz_j(s)\gamma_{k-j}(s),\quad k\ge 1.
\end{align*}
On account of \eqref{e2.28}, \eqref{e2.29}, \eqref{e2.31}, \eqref{e2.32}, and \eqref{e2.33}, we have 
\begin{align*}
 a_0(s)&=\begin{cases} -1+\frac{\pi}{4} s+O(s^3)&\text{as}\quad s\to 0,\\
    \frac 12 s^4e^{-s^2}+O(s^2e^{-s^2})&\text{as}\quad s\to \infty,
    \end{cases}\\
 a_1(s)&=\begin{cases} -\frac{\pi}{4} s+O(s^3)&\text{as}\quad s\to 0,\\
    \frac 14 s^8e^{-s^2}+O(s^6e^{-s^2})&\text{as}\quad s\to \infty,
    \end{cases}
 \end{align*}
 and, for $k\ge 2$,
 \[
 a_k(s)=\begin{cases} O(s^3)&\text{as}\quad s\to 0,\\
    \frac{1}{2^{k+1}k!}s^{4k+4}e^{-s^2}+O(s^{4k+2}e^{-s^2})&\text{as}\quad s\to \infty.
    \end{cases}
\]
Since
\[
\frac{1}{1+s^2/n}=\sum_{k=0}^{\infty}\frac{(-s^2)^k}{n^k},
\]
it follows from \eqref{e2.35} that
\[
\frac{F_n(s/\sqrt n)}{1+s^2/n}=\left(\sum_{k=0}^\infty \frac{a_k(s)}{n^k}\right)\left(\sum_{k=0}^\infty \frac{(-s^2)^k}{n^k}\right)=\sum_{k=0}^\infty \frac{f_k(s)}{n^k}\] 
uniformly for $s\in [0,c_n]$,
where 
\[
f_k(s):=\sum_{j=0}^k(-1)^js^{2j}a_{k-j}(s),\quad k\ge 0.
\]
As shown above, $f_k(s)$ have continuous extensions to $[0,\infty)$ such that 
\begin{align*}
    f_0(s)&=\begin{cases} -1+\frac{\pi}{4} s+O(s^3)&\text{as}\quad s\to 0,\\
     \frac 12 s^4e^{-s^2}+O(s^2e^{-s^2})&\text{as}\quad s\to \infty,
    \end{cases}\\
    f_1(s)&=\begin{cases} -\frac{\pi}{4}s+s^2+O(s^3)&\text{as}\quad s\to 0,\\
     \frac 14 s^8e^{-s^2}+O(s^6e^{-s^2})&\text{as}\quad s\to \infty,
    \end{cases}
\end{align*}
and, for $k\ge 2$, 
\[
f_k(s)=\begin{cases} O(s^3)&\text{as}\quad s\to 0,\\
  \frac{1}{2^{k+1}k!}s^{4k+4}e^{-s^2}+O(s^{4k+2}e^{-s^2})&\text{as}\quad s\to \infty.
    \end{cases}
\]
Thus, Lemma \ref{L2.6} is verified. 
\end{proof}
Recall that 
\[
f_0(s)=h(\delta_0(s))\gamma_0(s)-1
\]
and 
\[
f_1(s)=h(\delta_0(s))\gamma_1(s)+\delta_1(s)\arcsin(\delta_0(s))\gamma_0(s)-s^2f_0(s),
\]
where explicit formulas for $h(x)$, $\delta_0(s)$, $\delta_1(s)$, $\gamma_0(s)$, and $\gamma_1(s)$ are provided. This means that one can also obtain explicit formulas for both $f_0(s)$ and $f_1(s)$. In Figure \ref{F1}, we show plots of $f_0(s)$ and $f_1(s)$ for $s\in [0,5]$.
\begin{figure}[ht]
\centering
\includegraphics[scale = 1]{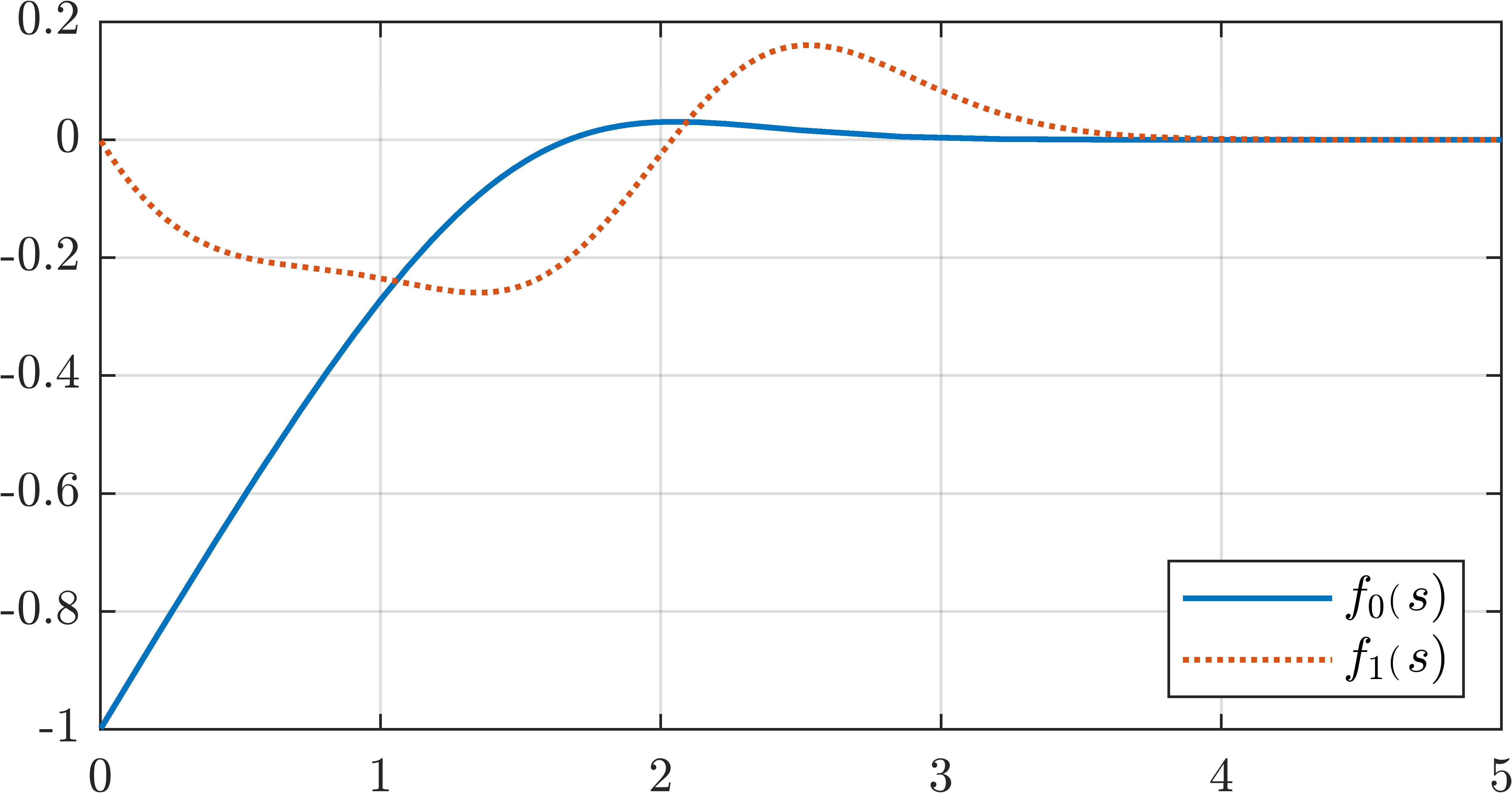}
\caption{Plots of $f_0(s)$ and $f_1(s)$.}
\label{F1}
\end{figure}

Note that, for $-1\le s/\sqrt n \le 1$,
\[
\sqrt n \arctan(s/\sqrt n)=\sum_{k=0}^\infty \frac{(-1)^ks^{2k+1}}{2k+1}\frac{1}{n^k}.
\]
Together with Lemma \ref{L2.6}, we obtain the following lemma. 
\begin{lemma}\label{L2.7}
Given $0<c_n<\sqrt n$, one has
\begin{equation}
\label{e2.36}
   \frac{F_n(s/\sqrt n)}{1+s^2/n} \sqrt n \arctan(s/\sqrt n)=\sum_{k=0}^\infty \frac{g_k(s)}{n^k} \quad \text{uniformly for}\; s\in [0,c_n],
\end{equation}
where
\[
g_k(s)=\sum_{j=0}^k\frac{(-1)^js^{2j+1}}{2j+1}f_{k-j}(s),\quad k\ge 0.
\]
Furthermore, $g_k(s)$ have continuous extensions to $[0,\infty)$ such that, as $s\to 0$, 
\begin{align}
    \label{e2.37}
g_k(s)&=\begin{cases}
-s+\frac{\pi}{4} s^2+O(s^4)&\text{if}\quad k=0,\\
-\frac{\pi}{4} s^2+\frac 43 s^3+O(s^4)&\text{if}\quad k=1,\\
O(s^4)&\text{if}\quad k\ge 2,
\end{cases}
\end{align}
and, as $s\to \infty$,
\begin{equation}
    \label{e2.38}
g_k(s)=\frac{1}{2^{k+1}k!}s^{4k+5}e^{-s^2}+ O(s^{4k+3}e^{-s^2}),\quad k\ge 0.
\end{equation}
\end{lemma}
Note that explicit formulas for $g_0(s)$ and $g_1(s)$ can be obtained from 
$
g_0(s)=sf_0(s)
$
and 
$
g_1(s)=sf_1(s)-\frac{s^3}{3}f_0(s).
$
Plots of $g_0(s)$ and $g_1(s)$ for $s\in [0,5]$ are included in Figure \ref{F2}. 

    \begin{figure}[ht]
\centering
\includegraphics[scale = 1]{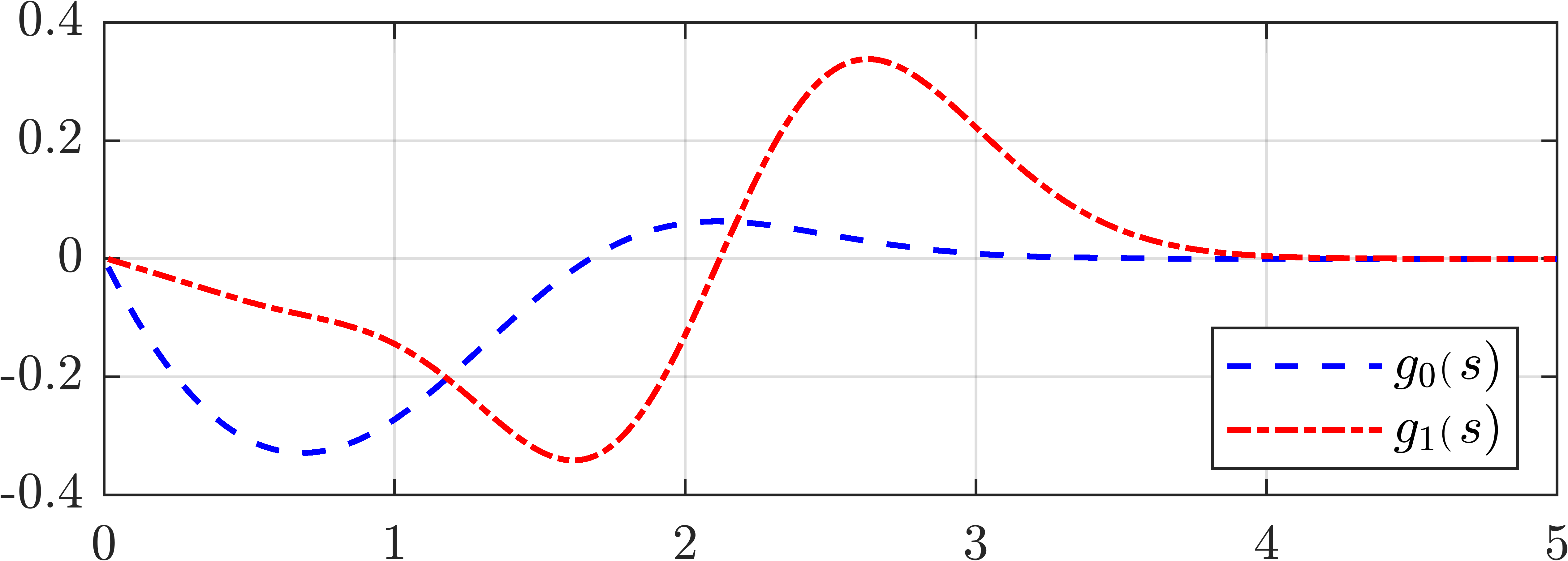}
\caption{Plots of $g_0(s)$ and $g_1(s)$.}
\label{F2}
\end{figure}

We can now define, for $k\ge 0$ and $c>0$, 
\begin{align*}
    &\kappa_k:=\frac{2}{\pi}\int_0^\infty f_k(s)ds, &\kappa_{c,k}:=\frac{2}{\pi}\int_0^c f_k(s)ds,\\
    &\ell_k:=\frac{2}{\pi^2}\int_0^\infty g_k(s)ds, &\ell_{c,k}:=\frac{2}{\pi^2}\int_0^c g_k(s)ds.
\end{align*}
The continuity of $f_k(s)$ and $g_k(s)$, along with the asymptotic behaviors given in \eqref{e2.14}, \eqref{e2.15}, \eqref{e2.37}, and \eqref{e2.38}, justifies the definitions of $\kappa_k, \ell_k, \kappa_{c,k}$, and $\ell_{c,k}$. Explicit formulas for $f_k(s)$ and $g_k(s)$ allow for the numerical computation of $\kappa_k, \ell_k, \kappa_{c,k}$, and $\ell_{c,k}$, for $k=0,1$, and $c>0$. Table \ref{tab1} lists some numerical values, where the integrals were numerically evaluated using MATLAB.

\begin{table}[ht]
    \caption{Numerical values of $\kappa_k, \ell_k$, $\kappa_{1,k}$, and $\ell_{1,k}$ for $k=0,1$.}
    \label{tab1}
    \centering
    \begin{tabular}{@{}lrrrr@{}}
     \hline 
     $k$& \multicolumn{1}{c}{$\kappa_k$}& \multicolumn{1}{c}{$\ell_k$}& \multicolumn{1}{c}{$\kappa_{1,k}$}& \multicolumn{1}{c}{$\ell_{1,k}$}\\
     \hline 
      $0$&$-0.4282689510$&$-0.0580365252$&$-0.3955313789$&$-0.0505415303$\\
      $1$&$-0.1522064957$&$-0.0082122652$&$-0.1093878905$&$-0.0138350833$\\
         \hline
    \end{tabular}
\end{table}

We emphasize that the expansions provided in \eqref{e2.13} and \eqref{e2.36} enable us to express $K_n(a,b)$ and $L_n(a,b)$ as series of terms that are powers of $1/n$, under the condition $|\alpha_n|<\sqrt n$. Our subsequent objective is to estimate $|K_n-K_n(a,b)|$ and $|L_n-L_n(a,b)|$ when $|\alpha_n|$ is arbitrarily large.
    \begin{lemma}\label{L2.8} Suppose $|\alpha_n|\ge 2$. There exists a positive constant $n_0$ such that for all $n\ge n_0$,
\begin{align}
    \label{e2.39}
    |K_n-K_n(a,b)|&\le \frac{8}{\pi}|\alpha_n|^3\left(1+\alpha_n^2/n\right)^{-n},\\
    \label{e2.40}
     |L_n-L_n(a,b)|&\le \frac{4}{\pi}\sqrt n|\alpha_n|^3\left(1+\alpha_n^2/n\right)^{-n}.
\end{align}
\end{lemma}
\begin{proof}
Observe that 
\begin{align*}
   |K_n-K_n(a,b)|&\le \frac{2}{\pi}\int_{|\alpha_n|}^\infty \frac{|F_n(s/\sqrt n)|}{1+s^2/n}ds,\\
   |L_n-L_n(a,b)|&\le  \frac{2}{\pi^2}\int_{|\alpha_n|}^\infty \frac{|F_n(s/\sqrt n)|}{1+s^2/n}\sqrt n\arctan(s/\sqrt n)ds.
\end{align*}
Since $0\le \arctan(s/\sqrt n)\le \frac{\pi}{2}$ for $s\in [|\alpha_n|,\infty)$, it follows that \eqref{e2.40} is a consequence of \eqref{e2.39}.

We now prove \eqref{e2.39}. Choose $n_0>0$ such that for $n\ge n_0$ and $s\ge 2$, we have 
\[
s^2(1+s^2/n)^{-n/2} \le \frac 14.
\]
Then, for $n\ge n_0$ and $s\ge 2$,
\[
|\Delta_n(s/\sqrt n)|=(1+s^2/n)^{-n/2}\frac{|(1+s^2/n)[1-(1+s^2/n)^{-n}]-s^2|}{1-(1+s^2/n)^{-n}-s^2(1+s^2/n)^{-n}} \le 2s^2(1+s^2/n)^{-n/2}.
\]
Let $s_n:=2s^2(1+s^2/n)^{-n/2}$. Then, $0\le s_n \le 1/2$ and 
\[
1\le h(x)=\sqrt{1-x^2}+x\arcsin x \le 1+\frac{x^2}{2\sqrt{1-s_n^2}},\quad x\in [-s_n,s_n].
\]
This gives
\[
1\le h(\Delta_n(s/\sqrt n)) \le 1+\frac{s_n^2}{2\sqrt{1-s_n^2}}\le 1+\frac{s_n^2}{\sqrt 3}.
\]
Together with 
\begin{align*}
    1-2s^2(1+s^2/n)^{-n}\le \Gamma_n(s/\sqrt n)&=\frac{1-(1+s^2/n)^{-n}-s^2(1+s^2/n)^{-n}}{[1-(1+s^2/n)^{-n}]^{3/2}}\\
    &\le \frac{1}{\sqrt{1-(1+s^2/n)^{-n}}} \le 1+(1+s^2/n)^{-n},
\end{align*}
we deduce that 
\begin{align*}
    |F_n(s/\sqrt n)|&=|h(\Delta_n(s/\sqrt n))\Gamma_n(s/\sqrt n)-1|\le s_n^2 = 4s^4(1+s^2/n)^{-n},
\end{align*}
hence, for $n\ge n_0$ and $s\ge 2$, 
\begin{equation}
    \label{e2.41}
    \frac{|F_n(s/\sqrt n)|}{1+s^2/n}\le 4s^4(1+s^2/n)^{-n-1}.
\end{equation}
 For $|\alpha_n|\ge 2$ and $n\ge n_0$, the function $s\mapsto s^3\left(1+s^2/n\right)^{-n/2}$ achieves its maximum value on $[|\alpha_n|,\infty)$ at $s=|\alpha_n|$, and therefore
\begin{align*}
\int_{|\alpha_n|}^\infty 4s^4\left(1+\frac{s^2}{n}\right)^{-n-1}ds &\le 4|\alpha_n|^3\left(1+\frac{\alpha_n^2}{n}\right)^{-n/2}\int_{|\alpha_n|}^\infty s\left(1+\frac{s^2}{n}\right)^{-n/2-1}ds\\
&=4|\alpha_n|^3\left(1+\frac{\alpha_n^2}{n}\right)^{-n}.
\end{align*}
Combining with \eqref{e2.41} yields \eqref{e2.39} as required.
\end{proof}
We can now formulate the asymptotic expansions of $K_n(a,b)$ and  $L_n(a,b)$.  
\begin{lemma}\label{L2.9} As $n\to \infty$, if $ \alpha_n^2/\log n \to \infty$, then 
\begin{align}
    \label{e2.42}
    K_n(a,b)&\sim\sum_{k=0}^{\infty}\frac{\kappa_k}{n^k},\\
   L_n(a,b)&\sim \sum_{k=0}^{\infty}\frac{\ell_k}{n^k}.
    \label{e2.43}
\end{align}
Consequently, as $n\to \infty$,
\begin{align}
    \label{e2.44}
    K_n&\sim\sum_{k=0}^{\infty}\frac{\kappa_k}{n^k},\\
   L_n&\sim \sum_{k=0}^{\infty}\frac{\ell_k}{n^k}.
    \label{e2.45}
\end{align}
\end{lemma}
\begin{proof} We first prove \eqref{e2.42}. 
If $\alpha_n^2\ge n$, one has 
\[
|\alpha_n|^3\left(1+\alpha_n^2/n\right)^{-n}\le n^{3/2}2^{-n}.
\]
But then \eqref{e2.39} implies that the integral $\frac{2}{\pi}\int_{\sqrt n}^{|\alpha_n|} \frac{F_n(s/\sqrt n)}{1+s^2/n}ds$ is negligible. Thus, it suffices to assume that $\alpha_n^2<n$. 
 By \eqref{e2.13}, 
 \begin{equation}
     \label{e2.46}
 K_n(a,b)=\sum_{k=0}^\infty \left(\frac{2}{\pi}\int_0^{|\alpha_n|} f_k(s)ds\right)\frac{1}{n^k}.
 \end{equation}
We now show that
\begin{equation}
    \label{e2.47}
    \sum_{k=0}^{\infty} \left(\frac{2}{\pi}\int_{|\alpha_n|}^\infty f_k(s)ds\right)\frac{1}{n^k}\ll |\alpha_n|^3e^{-\alpha^2_n+\alpha_n^4/2n}.
\end{equation}
In fact, since $|\alpha_n|\to \infty$ as $n\to \infty$, we see that, for any fixed $k\ge 0$ and all $n$ sufficiently large, the function $s\mapsto s^{4k+3}e^{-s^2/2}$ achieves its maximum value on $[|\alpha_n|,\infty)$ at $s=|\alpha_n|$. This implies
\begin{align*}
0\le \int_{|\alpha_n|}^{\infty}\frac{s^{4k+4}e^{-s^2}}{2^{k+1}k!}ds &\le \frac{|\alpha_n|^{4k+3}e^{-\alpha_n^2/2}}{ 2^{k+1}k!}\int_{|\alpha_n|}^{\infty}se^{-s^2/2}ds=\frac{|\alpha_n|^{4k+3}e^{-\alpha_n^2}}{2^{k+1}k!}.
\end{align*}
Therefore, 
\begin{align*}
\sum_{k=0}^{\infty}\left(\frac{2}{\pi}\int_{|\alpha_n|}^{\infty}\frac{s^{4k+4}e^{-s^2}}{2^{k+1}k!}ds\right)\frac{1}{n^k}&\le \frac{|\alpha_n|^3e^{-\alpha_n^2+\alpha_n^4/2n}}{\pi},
\end{align*}
which gives \eqref{e2.47} when combined with \eqref{e2.15}. Next, in view of  \eqref{e2.46} and \eqref{e2.47}, the series on the right-hand side of \eqref{e2.42} converges and 
\[
K_n(a,b)=\sum_{k=0}^{\infty}\frac{\kappa_k}{n^k}+O\left(|\alpha_n|^3e^{-\alpha^2_n+\alpha_n^4/2n}\right).
\]
Since $O(|\alpha_n|^3e^{-\alpha^2_n+\alpha_n^4/2n})$ is negligible when $\alpha_n^2/\log n \to \infty$, we get \eqref{e2.42}.

Similarly, \eqref{e2.43}  follows from applying Lemma \ref{L2.7} and \eqref{e2.40}.

Finally, by Lemma \ref{L2.8}, \eqref{e2.44} and \eqref{e2.45} follow from \eqref{e2.42} and \eqref{e2.43}, respectively. 
\end{proof}

Note that if $\alpha^2_n$ does not grow faster than $\log n$, then  $|\alpha_n|^3e^{-\alpha^2_n+\alpha_n^4/2n}\ge n^{-c}$ for some constant $c>0$. In this case, we are thus looking for finite expansions of $K_n(a,b)$ and $L_n(a,b)$.
\begin{lemma}\label{L2.10}
As $n\to \infty$, if $|\alpha_n|\to \infty$ and $\alpha^2_n=O(\log n)$, then for $d_n$ given by \eqref{e1.7}, we have
\begin{align*}
    K_n(a,b)=\sum_{k=0}^{d_n}\frac{\kappa_k}{n^k}+O(|\alpha_n|^3e^{-\alpha_n^2}) \quad \text{and}\quad 
    L_n(a,b)=\sum_{k=0}^{d_n}\frac{\ell_k}{n^k}+O(\alpha_n^4e^{-\alpha_n^2}).
\end{align*}
\end{lemma}
\begin{proof} Since $\alpha^2_n=O(\log n)$, $d_n$ is bounded. A slight change in the proof of Lemma \ref{L2.6} actually shows that, for $s\in [0,|\alpha_n|]$,
\begin{align*}
    \frac{F_n(s/\sqrt n)}{1+s^2/n}=\sum_{k=0}^{d_n}\frac{f_k(s)}{n^k}+O\left(\frac{\alpha_n^{4d_n}}{n^{d_n+1}}\right).
\end{align*}
It follows that 
\begin{align*}
    K_n(a,b)&=\sum_{k=0}^{d_n}\left(\frac{2}{\pi}\int_0^{|\alpha_n|}f_k(s)ds\right)\frac{1}{n^k} +O\left(\frac{|\alpha_n|^{4d_n+1}}{n^{d_n+1}}\right)\\
    &=\sum_{k=0}^{d_n}\frac{\kappa_k}{n^k}-\sum_{k=0}^{d_n}\left(\frac{2}{\pi}\int_{|\alpha_n|}^\infty f_k(s)ds\right)\frac{1}{n^k}+O\left(\frac{|\alpha_n|^{4d_n+1}}{n^{d_n+1}}\right).
\end{align*}
Analysis similar to that in the proof of \eqref{e2.47} shows 
\[
\sum_{k=0}^{d_n}\left(\frac{2}{\pi}\int_{|\alpha_n|}^\infty f_k(s)ds\right)\frac{1}{n^k}\le \frac{1+d_n}{\pi}|\alpha_n|^3e^{-\alpha_n^2}\ll |\alpha_n|^3e^{-\alpha_n^2}.
\]
This clearly forces 
\[
 K_n(a,b)=\sum_{k=0}^{d_n}\frac{\kappa_k}{n^k}+O(|\alpha_n|^3e^{-\alpha_n^2}).
\]
The term $L_n(a,b)$ can be handled in much the same way.
\end{proof}
Next, we establish the asymptotics of $K_n(a,b)$ and $L_n(a,b)$ when $\alpha_n=o(1)$.
\begin{lemma}\label{L2.11} As $n\to \infty$, if $\alpha_n=o(1)$, then
\begin{align}
    \label{e2.48}
    K_n(a,b)&=-\frac{2}{\pi}|\alpha_n|+\frac{1}{4}\alpha_n^2-\frac 14 \frac{\alpha_n^2}{n}+\frac{2}{3\pi}\frac{|\alpha_n|^3}{n}+O(\alpha_n^4),\\
    \label{e2.49}
    L_n(a,b)&=-\frac{1}{\pi^2}\alpha_n^2+\frac{1}{6\pi}|\alpha_n|^3-\frac{1}{6\pi}\frac{|\alpha_n|^3}{n}+\frac{2}{3\pi^2}\frac{\alpha_n^4}{n}+O(|\alpha_n|^5).
\end{align}
\end{lemma}
\begin{proof}
By \eqref{e2.13} and \eqref{e2.14}, 
\begin{align*}
    K_n(a,b)&=\frac{2}{\pi}\int_0^{|\alpha_n|}\left[-1 +\frac{\pi }{4}s+\frac{1}{n}\left(-\frac{\pi}{4}s+s^2\right)\right]ds+O(\alpha_n^4)\\
    &=-\frac{2}{\pi}|\alpha_n|+\frac{1}{4}\alpha_n^2-\frac 14 \frac{\alpha_n^2}{n}+\frac{2}{3\pi}\frac{|\alpha_n|^3}{n}+O(\alpha_n^4),
\end{align*}
and \eqref{e2.48} is proved. Similarly, \eqref{e2.49} follows from applying Lemma \ref{L2.7}.
\end{proof}
We are now ready to prove Theorem \ref{T1.2}.
\begin{proof}[Proof of Theorem \ref{T1.2}] The proof relies on the asymptotic behavior of $\alpha_n$ as $n$ becomes large.
\begin{enumerate}
    \item  Assume first that $|\alpha_n|\to \infty$ as $n\to \infty$. Then \eqref{e1.9} is a consequence of Theorem \ref{T1.1} and Lemma \ref{L2.9}, while \eqref{e1.8} follows from Theorem \ref{T1.1}, the relation \eqref{e2.44}, and Lemma \ref{L2.10}.
    \item If $\alpha_n=c>0$, then  
    \[\mathbb E[N_n(a,b)]=\frac{\sqrt{ n}}{\pi}\arctan\frac{c}{\sqrt n}.\] 
    Thus, using \eqref{e1.4}, Lemmas \ref{L2.6} and \ref{L2.7}, we deduce \eqref{e1.10}. For $\alpha_n=-c$, \eqref{e1.11} follows from \eqref{e1.6}, \eqref{e2.44}, Lemmas \ref{L2.6} and \ref{L2.7}, along with the fact that
    \[\mathbb E[N_n(a,b)]=\sqrt n+ \frac{\sqrt n}{\pi}\arctan\frac{c}{\sqrt n}.\]
    \item Suppose that $\alpha_n=o(1)$ as $n\to \infty$. If $\alpha_n>0$, then 
    \[
    \mathbb E[N_n(a,b)]=\frac{\sqrt n}{\pi}\arctan \alpha(a,b)= 
    \frac{1}{\pi}\left(\alpha_n -\frac{\alpha_n^3}{3n}\right)+O(\alpha_n^5).\] 
    Hence, \eqref{e1.12} follows from \eqref{e1.4} and Lemma \ref{L2.11}.  
    Next, for $\alpha_n<0$, we see that      
    \[\mathbb E[N_n(a,b)]=\sqrt n+ \frac{1}{\pi}\left(\alpha_n -\frac{\alpha_n^3}{3n}\right)+O(|\alpha_n|^5).\]  
    Combining this with \eqref{e1.6}, Lemma \ref{L2.11}, and the facts that $\sqrt n/n^{u_n}=O(|\alpha_n|^5)$ and $ \alpha_n/n^{v_n}=O(|\alpha_n|^5)$, we establish \eqref{e1.13}. If additionally, $v_n\to \infty$, implying $u_n\to \infty$ and $\alpha_n=o(n^{-v_n/4})$, which is negligible, then \eqref{e1.14} follows as a consequence of \eqref{e1.13}. Finally, substituting \eqref{e2.44} into \eqref{e1.5} yields \eqref{e1.15}.
\end{enumerate}
\end{proof}
\section{Cumulants and their asymptotics}
In this section, we thoroughly explore the computation and asymptotic analysis of cumulants related to the number of real zeros of GAFs. Additionally, we provide proofs for Theorems \ref{T1.14}, \ref{T1.20}, \ref{T1.24}, and \ref{T1.3}.
\subsection{Correlation functions} \label{S3.1}  To compute the cumulants of the number of real zeros of GAFs, we employ the concept of correlation functions and truncated correlation functions, as outlined below (see Do and Vu \cite{DV}  and Nazarov and Sodin \cite{NS} for additional details).

Let $Z$ be a random point process on $\mathbb R$. 
For $k\ge 1$, the function $\rho_{k}:\mathbb R^k \to \mathbb R$ is called the $k$-point correlation function of $Z$ if, for any compactly supported $C^\infty$ function $f:\mathbb R^k \to \mathbb R$, it holds that 
\[
\mathbb E\left[\sum f(x_1,...,x_k)\right]=\int_{\mathbb R^k}f(\xi_1,...,\xi_k)\rho_{k}(\xi_1,...,\xi_k)d\xi_1\cdots d\xi_k,
\]
where the sum is taken over all possible ordered $k$-tuples $(x_1,...,x_k)$ of distinct points in $Z$. 
The $k$-point correlation function is symmetric and locally integrable on $\mathbb R^k$.
Furthermore, if there is $\epsilon>0$ such that $\rho_k$ is locally $L^{1+\epsilon}$ integrable, then for any interval $U\subset \mathbb R$, it holds that
\[
\mathbb E[N_Z(U)(N_Z(U)-1)\cdots(N_Z(U)-k+1)]=\int_{U^k}\rho_k(\xi_1,...,\xi_k)d\xi_1\cdots d\xi_k,
\] 
where $N_Z(U)=|Z\cap U|$ (see, for example, \cite{DV}*{\S7}).

It is important to note that the $k$-point correlation function does not always exist. Within the scope of this paper, the presence of correlation functions directly derives from the Kac-Rice formula (see, for example, \cite{AW}*{Chapter 3}). For instance, if $Z$ denotes the multiset of real zeros of a smooth, non-degenerate, centered Gaussian process $G$, the $k$-point correlation function $\rho_k(\pmb \xi)$ is well-defined for $\pmb \xi=(\xi_1,...,\xi_k)$ of distinct points in $\mathbb R$, and it is given by
\begin{align*}
    \rho_k(\pmb \xi)=\int_{\mathbb R^k}|y_1\cdots y_k|D_k(\pmb y;\pmb \xi)dy_1\cdots dy_k,
\end{align*}
where $\pmb y=(0,y_1,...,0, y_k)\in \mathbb R^{2k}$ and  $D_k(\pmb y;\pmb \xi)$ is the probability density of the Gaussian vector $(G(\xi_1),G'(\xi_1),...,G(\xi_k),G'(\xi_k))$.   This density can be more explicitly expressed as
\[
D_k(\pmb y;\pmb \xi)=\frac{e^{-\frac 12 \langle \pmb y(\Gamma(\pmb \xi))^{-1},\pmb y \rangle}}{(2\pi)^{k}\sqrt{\det \Gamma(\pmb \xi)}},
\]
where $\Gamma(\pmb \xi)$ is the covariance matrix of $(G(\xi_1),G'(\xi_1),...,G(\xi_k),G'(\xi_k))$.

We define $\Pi(k)$ as the set of all unordered partitions of the set $\{1,\ldots, k\}$ into disjoint nonempty blocks, and $\Pi(k,j)$ as the set of all unordered partitions of the set $\{1,\ldots, k\}$ into exactly $j$ disjoint nonempty blocks. For a partition $\gamma$ in $\Pi(k,j)$, we denote the blocks as $\{\gamma_1,\ldots, \gamma_j\}$ with an arbitrarily chosen enumeration and the lengths of the blocks as $l_i = |\gamma_i|$ for $1 \le i \le j$. For $\pmb\xi = (\xi_1, \ldots, \xi_k)$ and $\gamma_j \subset \{1, \ldots, k\}$, let $\pmb\xi_{\gamma_j}$ stand for  $(\xi_i)_{i\in \gamma_j}$.

The function $\rho^T_{k}$, defined as
\begin{equation} \label{e3.1}
    \rho_{k}^T(\pmb \xi)=\sum_{j=1}^k(-1)^{j-1}(j-1)!\sum_{\gamma\in \Pi(k,j)}\rho_{l_1}(\pmb \xi_{\gamma_1})\cdots \rho_{l_j}(\pmb \xi_{\gamma_j}),
\end{equation}
is called the truncated $k$-point correlation function of $Z$. 

We can verify that $\rho^T_{1}(\xi_1)=\rho_{1}(\xi_1)$, $\rho^T_{2}(\xi_1,\xi_2)=\rho_{2}(\xi_1,\xi_2)-\rho_{1}(\xi_1)\rho_{1}(\xi_2)$, 
\begin{align*}
    \rho_{3}^T(\xi_1,\xi_2,\xi_3)&=\rho_{3}(\xi_1,\xi_2,\xi_3)-\rho_{2}(\xi_1,\xi_2)\rho_{1}(\xi_3)-\rho_{2}(\xi_2,\xi_3)\rho_{1}(\xi_1)-\rho_{2}(\xi_1,\xi_3)\rho_{1}(\xi_2)\\
    &\quad +2\rho_{1}(\xi_1)\rho_{1}(\xi_2)\rho_{1}(\xi_3),
\end{align*}
and so on. Moreover, the inversion of \eqref{e3.1} takes the form 
\begin{equation}
    \label{e3.2}
    \rho_k(\pmb \xi)=\sum_{\gamma \in \Pi(k)}\rho^T_{l_1}(\pmb \xi_{\gamma_1})\cdots \rho^T_{l_j}(\pmb \xi_{\gamma_j}).
\end{equation}

 The computation of the $k$th cumulant $s_k[N_Z(U)]$ necessitates knowledge of $m$-point truncated correlation functions $\rho_m^T$ for all $m\in \{1,...,k\}$. 
\begin{lemma} \label{L3.1} Let $k\ge 1$. If $\rho_m^T \in L^{1}(U^m)$ for all $m\in \{1,...,k\}$, then
    \[
        s_k[N_Z(U)]=\sum_{\gamma\in \Pi(k)}\int_{U^{|\gamma|}}\rho^T_{|\gamma|}(\pmb \xi_\gamma)d\pmb\xi_\gamma,
    \]
where $|\gamma|$ is the number of blocks in the partition $\gamma$ and $d\pmb\xi_\gamma$ is the Lebesgue measure on $U^{|\gamma|}$.
\end{lemma}
For detailed proofs, we refer the reader to, for example, \cite{DV}*{Appendix B} or \cite{NS}*{Appendix}. 
\subsection{Estimates for correlation functions of real zeros of real GAFs} \label{S3.2}
In what follows, let $k\ge 1$ be an integer and $Q$ be a real GAF. Let $\rho_k$ and $\rho_k^T$ denote the $k$-point and truncated $k$-point correlation functions for the real zeros of $Q$, respectively. The proofs in this subsection adhere closely to the approach in \cite{DV}, building upon arguments from \cite{NS}.

The following lemma asserts that if $Q$ is $2k$-nondegenerate, then $\rho_k$ is locally bounded (see \cite{DV}*{\S 10.1} for a proof).   
\begin{lemma}[\cite{DV}] \label{L3.2}
    If $Q$ is $2k$-nondegenerate, then for any $M>0$, there exists a finite positive constant $C_{Q, M, k}$ such that, for all $\xi_1, \ldots, \xi_k \in [-M, M]$, one has
    \[
    \frac{1}{C_{Q,M,k}}\prod_{1\le i<j\le k}|\xi_i-\xi_j| \le \rho_k(\xi_1,...,\xi_k) \le C_{Q,M,k}\prod_{1\le i<j\le k}|\xi_i-\xi_j|.
    \]
\end{lemma}

The next result illustrates a clustering property for $\rho_k$ when $Q$ fulfills hypothesis \ref{H3}.
\begin{lemma} \label{L3.3}
    Suppose that $Q$ satisfies hypothesis \ref{H3}. There exist positive constants $C_k$ and $D_k$ such that the following holds: For any $\pmb \xi=(\xi_1,...,\xi_k)$ of distinct points in $\mathbb R$ and any partition $\pmb \xi= \pmb \xi_I \cup \pmb \xi_J$ with $d=d(\pmb \xi_I, \pmb \xi_J)\ge D_k \tau_k$, we have 
    \[
    \left|\frac{\rho_k(\pmb \xi)}{\rho_{|I|}(\pmb \xi_I)\rho_{|J|}(\pmb \xi_J)}-1 \right|\le C_k\psi(d-\tau_k).
    \]
\end{lemma}
\begin{proof} 
    For $c_k, \tau_k$, and $\psi$ as in \ref{H3}, 
    let $\epsilon= 2c_k\psi(d-\tau_k)>0$. Since $\lim_{x\to \infty}\psi(x)=0$, we can choose $D_k \ge 2$ such that $\epsilon<1$. Let $C_k>0$ be such that 
    \[
    0<1-C_k\psi(d-\tau_k) \le \left(\frac{1-\epsilon}{1+\epsilon}\right)^k \quad\text{and}\quad \left(\frac{1+\epsilon}{1-\epsilon}\right)^k \le 1+C_k\psi(d-\tau_k).
    \]
    The proof is completed by showing that 
    \begin{equation}\label{e3.3}
        \left(\frac{1-\epsilon}{1+\epsilon}\right)^k \rho_{|I|}(\pmb \xi_I)\rho_{|J|}(\pmb \xi_J) \le \rho_k(\pmb \xi) \le \left(\frac{1+\epsilon}{1-\epsilon}\right)^k\rho_{|I|}(\pmb \xi_I)\rho_{|J|}(\pmb \xi_J).
    \end{equation}
    By \ref{H3}, 
    \[
    \left|\mathbb  E\left[L_I^{\pmb z}Q(\pmb \xi_I)L_J^{\pmb z}Q(\pmb \xi_J)\right]\right| \le \frac{1}{2}\epsilon \left(\mathbb  E[\left|L_I^{\pmb z}Q(\pmb \xi_I)\right|^2]+\mathbb  E[\left|L_J^{\pmb z}Q(\pmb \xi_J)\right|^2]\right),
    \]
    which implies 
    \begin{align*}
      (1-\epsilon) \left(\mathbb  E[\left|L_I^{\pmb z}Q(\pmb \xi_I)\right|^2]+\mathbb  E[\left|L_J^{\pmb z}Q(\pmb \xi_J)\right|^2]\right) &\le E[\left|L^{\pmb z}Q(\pmb \xi)\right|^2]\\
      & \le (1+\epsilon) \left(\mathbb  E[\left|L_I^{\pmb z}Q(\pmb \xi_I)\right|^2]+\mathbb  E[\left|L_J^{\pmb z}Q(\pmb \xi_J)\right|^2]\right).  
    \end{align*}
    Let $\Lambda(\pmb \xi)$ be the covariance matrix of $(Q(\xi_1),Q'(\xi_1),...,Q(\xi_k), Q'(\xi_k))$. Similarly, we can define the covariance matrices $\Lambda_I(\pmb \xi_I)$ and $\Lambda_J(\pmb \xi_J)$. We set 
    \[\Lambda_{I,J}(\pmb \xi)=\begin{pmatrix}
        \Lambda_I(\pmb \xi_I)&0\\
        0&\Lambda_J(\pmb \xi_J)
    \end{pmatrix}.\]
    By \eqref{e1.24}, we deduce that 
    \[(1-\epsilon)\Lambda_{I,J}(\pmb \xi) \le \Lambda(\pmb \xi) \le (1+\epsilon)\Lambda_{I,J}(\pmb \xi).\]
    This yields
    \[
    \det \Lambda(\pmb \xi)\ge (1-\epsilon)^{2k} \det \Lambda_{I,J}(\pmb \xi) = (1-\epsilon)^{2k} \det \Lambda_I(\pmb \xi_I) \det \Lambda_J(\pmb \xi_J),
    \]
    and hence, applying the Kac-Rice formula, we derive
    \begin{align*}
    \rho_k(\pmb \xi)&=\frac{1}{(2\pi)^{k}\sqrt{\det \Lambda(\pmb \xi)}}\int_{\mathbb R^k}|y_1\cdots y_k|e^{-\frac 12\langle \pmb y (\Lambda(\pmb \xi))^{-1},\pmb y\rangle}dy_1\cdots dy_k\\
    &\le \frac{(1-\epsilon)^{-k}}{(2\pi)^{k}\sqrt{\det \Lambda_{I,J}(\pmb \xi)}}\int_{\mathbb R^k}|y_1\cdots y_k|e^{-\frac{1}{2(1+\epsilon)}\langle \pmb y(\Lambda_{I,J}(\pmb \xi))^{-1},\pmb y\rangle}dy_1\cdots dy_k \\
        &=\left(\frac{1+\epsilon}{1-\epsilon}\right)^k\rho_{|I|}(\pmb \xi_I)\rho_{|J|}(\pmb \xi_J),
    \end{align*}
    where $\pmb y=(0,y_1,...,0,y_k)$ and the last equality is obtained by changing variables and using the fact that 
    \[
    (\Lambda_{I,J}(\pmb \xi))^{-1} = \begin{pmatrix}
        (\Lambda_I(\pmb \xi_I))^{-1}&0\\
        0&(\Lambda_J(\pmb \xi_J))^{-1}
    \end{pmatrix}.
    \]
    Similarly, we have 
    \[
    \rho_k(\pmb \xi) \ge \left(\frac{1-\epsilon}{1+\epsilon}\right)^k \rho_{|I|}(\pmb \xi_I)\rho_{|J|}(\pmb \xi_J),
    \]
    and \eqref{e3.3} is proved.
\end{proof}

Lemma \ref{L3.3} demonstrates that when variables in $\mathbb{R}^k$ are partitioned into two well-separated clusters, the correlation function $\rho_k$ closely approximates the product of the corresponding factors. Hypothesis \ref{H2} ensures translation invariance of $\rho_m$ ($1\le m \le k$). Complementing the local estimates derived from Lemma \ref{L3.2}, we establish the following uniform estimates for $\rho_k$.
\begin{lemma} \label{L3.4}
    If $Q\in \mathcal A_k$, there exists a positive constant $C_k$ such that
    \[
    \frac{1}{C_k}\prod_{1\le i<j\le k}\min(1,|\xi_i-\xi_j|) \le \rho_k(\xi_1,...,\xi_k) \le C_k\prod_{1\le i<j\le k}\min(1,|\xi_i-\xi_j|).
    \]
\end{lemma}
\begin{proof}
    We proceed with induction on $k$. The base case $k=1$ is evident, with $\rho_1$ being a constant. For any $k\ge 2$, assuming the lemma holds for $m$-point correlation functions with $m\le k-1$, we aim to prove it for $k$-point functions. We consider two scenarios.

    If $\pmb \xi=(\xi_1,...,\xi_{k})\in \mathbb R^k$ can be split into two groups $\pmb \xi_I$ and $\pmb \xi_J$ with $d(\pmb \xi_I, \pmb \xi_J)\ge D_k\tau_k$, then employing Lemma \ref{L3.3}, we obtain
    \[
 (1-C_k\psi(d-\tau_k))\rho_{|I|}(\pmb \xi_I)\rho_{|J|}(\pmb \xi_J)\le \rho_k(\pmb \xi)\le (1+C_k\psi(d-\tau_k))\rho_{|I|}(\pmb \xi_I)\rho_{|J|}(\pmb \xi_J),
    \]
    and the result follows from the induction hypothesis.
    
    If no such splitting exists, then $\diam(\pmb \xi)=\max_{1\le i<j\le k}|\xi_i-\xi_j|$ is bounded. In this case, the result follows from the local bounds in Lemma \ref{L3.2} and translation invariance due to \ref{H2}.
\end{proof}
By combining Lemma \ref{L3.3} and Lemma \ref{L3.4}, we readily obtain the following additive version of the clustering property.
\begin{lemma} \label{L3.5}
    If $Q\in \mathcal A_k$, there exist positive constants $C_k$ and $\tau_k$ such that the following holds: For $\pmb \xi=(\xi_1,...,\xi_k)\in \mathbb R^k$ and any partition $\pmb \xi= \pmb \xi_I \cup \pmb \xi_J$ with $d=d(\pmb \xi_I, \pmb \xi_J)\ge 2\tau_k$, we have 
    \[
    \left|\rho_k(\pmb \xi)-\rho_{|I|}(\pmb \xi_I)\rho_{|J|}(\pmb \xi_J) \right|\le C_k\psi(d-\tau_k).
    \]
\end{lemma}

The asymptotic factorization of $k$-point correlation functions, as stated in Lemma \ref{L3.5}, results in the asymptotic decay of truncated $k$-point functions as the variable's diameter increases.
\begin{lemma}\label{L3.6}
    Assuming $Q\in \mathcal A_k$, there exist finite positive constants $c_k$ and $C_k$ such that for any $\pmb \xi=(\xi_1,...,\xi_k) \in \mathbb R^k$, the following inequality holds
    \begin{equation}\label{e3.4}
        |\rho_k^T(\pmb \xi)| \le C_k \hat{\psi}(c_k\diam(\pmb \xi)),
    \end{equation}
    where $\hat{\psi} =\min(1, \psi)$ and $\diam(\pmb \xi)=\max_{1\le i<j\le k}|\xi_i-\xi_j|$ is the diameter of $\pmb \xi$.
\end{lemma}
\begin{proof}
     For $k=1$, \eqref{e3.4} is obvious since $\rho_1^T=\rho_1$, which is a constant. For $k\ge 2$, we aim to show that there exist finite positive constants $c_k$ and $C_k$ such that, for any $\pmb \xi=(\xi_1,...,\xi_k) \in \mathbb R^k$ and any partition $\{1,\dots,k\}=I\cup J$, 
      \begin{equation} \label{e3.5}
    |\rho_k^T(\pmb \xi)| \le C_k \hat{\psi}(c_kd(\pmb\xi_I,\pmb \xi_J)).
    \end{equation}
     Since for any $\pmb \xi\in \mathbb R^k$, there exists a partition $\{1,...,k\}=I\cup J$ such that $d(\pmb \xi_I, \pmb \xi_J)\ge \frac{1}{2k}\diam(\pmb \xi)$, it follows that \eqref{e3.5} implies \eqref{e3.4}.
     
     We prove \eqref{e3.5} by induction on $k$. For $k=2$, it follows from the clustering of $\rho_2$ and the uniform boundedness of $\rho_1$ and $\rho_2$ that
    \[
    |\rho_2^T(\xi_1,\xi_2)|=|\rho_2(\xi_1,\xi_2)-\rho_1(\xi_1)\rho_1(\xi_2)| \le C \min(1, \psi(c|\xi_1-\xi_2|)).
    \]
    Let $k\ge 3$. Fix a partition $\{1,...,k\}=I\cup J$ and define $\Pi(I,J)$ as the set of non-trivial partitions of $\{1,...,k\}$ that mix $I\cup J$; that is, it contains partitions with at least two blocks such that there exists a block intersecting both $I$ and $J$. By the inversion formula \eqref{e3.2}, we have 
    \[
    \rho_k(\pmb \xi)-\rho_{|I|}(\pmb \xi_I)\rho_{|J|}(\pmb \xi_J)=\rho_k^T(\pmb \xi)+\sum_{\gamma\in \Pi(I,J)}\rho^T_{l_1}(\pmb \xi_{\gamma_1})\cdots \rho^T_{l_j}(\pmb \xi_{\gamma_j}),
    \]
    which implies
    \[
    |\rho_k^T(\pmb\xi)|\le |\rho_k(\pmb \xi)-\rho_{|I|}(\pmb \xi_I)\rho_{|J|}(\pmb \xi_J)|+\sum_{\gamma\in \Pi(I,J)}|\rho^T_{l_1}(\pmb \xi_{\gamma_1})\cdots \rho^T_{l_j}(\pmb \xi_{\gamma_j})|.
    \]
    Therefore, by employing Lemma \ref{L3.5} and the induction assumption, we can assert that there exist positive constant $c_k$ and $C_k$ such that \eqref{e3.5} holds.
\end{proof}

It is noteworthy that in \cite{BSZ}*{\S5}, Bleher, Shiffman, and Zelditch derived analogous estimates for correlation functions and truncated correlation functions utilizing the Wick formula, subject to the condition $\min_{i\ne j}|\xi_i-\xi_j|\ge c>0$.

\subsection{Proof of Theorem \ref{T1.14}} \label{S3.3} Applying Lemma \ref{L3.1},  we express $s_m[N_Q(R)]$ as
\[
s_m[N_Q(R)] =\sum_{\gamma\in \Pi(m)}\int_{[0,R]^{|\gamma|}}\rho^T_{|\gamma|}(\pmb \xi_\gamma)d\pmb\xi_\gamma,
\]
where $|\gamma|$ is the number of blocks in the partition $\gamma$ and $d\pmb\xi_\gamma$ is the Lebesgue measure on $[0,R]^{|\gamma|}$. Our goal is to establish, for each $j\in \{1,...,m\}$, the existence of bounded functions $\hat \theta_j$ and $\hat\lambda_j: [0,\infty)\to \mathbb R$ such that
\begin{equation} \label{e3.6}
    \int_{[0,R]^j}\rho^T_j(\pmb \xi)d\pmb\xi = R\hat \theta_j(R)+\hat\lambda_j(R).
\end{equation}
Assuming \eqref{e3.6}, we obtain \eqref{e1.22}  with 
\begin{equation} \label{e3.7}
\theta_m^Q(R)=\sum_{\gamma\in \Pi(m)}\hat \theta_{|\gamma|}(R) \quad \text{and}\quad \lambda_m^Q(R)=\sum_{\gamma\in \Pi(m)}\hat \lambda_{|\gamma|}(R).
\end{equation}
Clearly, \eqref{e3.6} holds for $j=1$, where $\hat \theta_1(R)=\rho_1(0)$ and $\hat\lambda_1(R)=0$. Now, consider the case where $j\ge 2$. Due to translation invariance,
\begin{align*}
    \int_{[0,R]^j}\rho^T_j(\pmb \xi)d\pmb\xi &=\int_0^Rd\xi_j\int_{[0,R]^{j-1}}\rho^T_j(\xi_1-\xi_j,...,\xi_{j-1}-\xi_j,0)d\xi_1\cdots d\xi_{j-1}.
\end{align*}
Through a change of variables and application of Fubini's theorem, we see that
\begin{align*}
    \int_{[0,R]^j}\rho^T_j(\pmb \xi)d\pmb\xi & =\int_0^R d\xi_j \int_{[-\xi_j,R-\xi_j]^{j-1}}\rho_j^T(t_1,...,t_{j-1},0)dt_1\cdots dt_{j-1}\\
    &= \int_0^R d\xi_j \int_{[-R,R]^{j-1}}\rho_j^T(t_1,...,t_{j-1},0)\prod_{i=1}^{j-1}\pmb 1_{[-\xi_j, R-\xi_j]}(t_i)dt_1\cdots dt_{j-1}\\
    &=  \int_{[-R,R]^{j-1}}dt_1\cdots dt_{j-1}\int_0^R \rho_j^T(t_1,...,t_{j-1},0) \prod_{i=1}^{j-1}\pmb 1_{[-t_i, R-t_i]}(\xi_j)d\xi_j\\
    &=  \int_{[-R,R]^{j-1}}dt_1\cdots dt_{j-1} \rho_j^T(t_1,...,t_{j-1},0)\\
    &\qquad \qquad \qquad \times [R-\max(t_1,...,t_{j-1},0)+\min(t_1,...,t_{j-1},0)]\\
    &= R\int_{[-R,R]^{j-1}}\rho_j^T(\pmb t,0) d\pmb t -  \int_{[-R,R]^{j-1}}\rho_j^T(\pmb t,0)\diam(\pmb t,0) d\pmb t,
\end{align*}
where $\pmb t=(t_1,...,t_{j-1})$, $d\pmb t=dt_1\cdots dt_{j-1}$, and $\diam(\pmb t,0)$ is the diameter of the configuration $(t_1,...,t_{j-1},0)$. This implies \eqref{e3.6} with 
\[
\hat \theta_j(R)=\int_{[-R,R]^{j-1}}\rho_j^T(\pmb t,0) d\pmb t \quad\text{and}\quad \hat \lambda_j(R)=-\int_{[-R,R]^{j-1}}\rho_j^T(\pmb t,0)\diam(\pmb t,0) d\pmb t.
\]
We deduce from Lemma \ref{L3.6} that 
\[
|\hat \theta_j(R)|\le \int_{[-R,R]^{j-1}}|\rho_j^T(\pmb t,0)| d\pmb t\ll \int_{[-R,R]^{j-1}}\psi(c_k\diam(\pmb t,0))d\pmb t.
\]
Similarly, 
\begin{align*}
    |\hat \lambda_j(R)| \ll \int_{[-R,R]^{j-1}}\psi(c_k\diam(\pmb t,0))\diam(\pmb t,0) d\pmb t.
\end{align*}
Hence, the boundedness of $\hat\theta_j$ and $\hat\lambda_j$ follows from the subsequent lemma and the fact that
\[
\int_0^\infty \psi(x)x^{k-1}dx<\infty.
\]
\begin{lemma} \label{L3.7}
    For any $\psi\in \Psi_k$ and $R>0$, it holds that
    \begin{equation} \label{e3.8}
       k\int_0^R \psi(x)x^{k-1}dx \le \int_{(-R,R)^{k}}\psi(\diam(\pmb \xi,0)) d\pmb \xi \le 2^kk\int_0^{2R} \psi(x)x^{k-1}dx. 
    \end{equation}
\end{lemma}
Assuming Lemma \ref{L3.7}, we can define
\[
\hat \theta_{j,\infty}=\int_{\mathbb R^{j-1}}\rho_j^T(\pmb t,0) d\pmb t \quad\text{and}\quad \hat \lambda_{j,\infty}=-\int_{\mathbb R^{j-1}}\rho_j^T(\pmb t,0)\diam(\pmb t,0) d\pmb t.
\]
Moreover, there exists a function $\psi\in \Psi_k$ such that, as $R\to \infty$, we have
\[
|\hat\theta_j(R)-\hat\theta_{j,\infty}|\ll  \int_{R}^\infty \psi(x)x^{j-2}dx
\quad\text{and}\quad 
|\hat\lambda_j(R)-\hat\lambda_{j,\infty}| \ll  \int_{R}^\infty \psi(x)x^{j-1}dx.
\]
Substituting these estimates into \eqref{e3.7}, we deduce
\[
\theta_m^Q(R) =\theta_{m,\infty}^Q+O\left(\int_{R}^\infty \psi(x)x^{m-2}dx\right)\quad\text{and}\quad \lambda_m^Q(R) =\lambda_{m,\infty}^Q+O\left( \int_{R}^\infty \psi(x)x^{m-1}dx\right),
\]
where 
\[
\theta_{m,\infty}^Q=\sum_{\gamma\in \Pi(m)}\hat \theta_{|\gamma|,\infty} \quad \text{and}\quad \lambda_{m,\infty}^Q=\sum_{\gamma\in \Pi(m)}\hat \lambda_{|\gamma|,\infty}.
\]
Thus, \eqref{e1.23} is verified. It remains to prove Lemma \ref{L3.7}. 
\begin{proof}[Proof of Lemma \ref{L3.7}]
   We begin by noting that 
   \[
    \int_{(0,R)^{k}}\psi(\diam(\pmb \xi,0)) d\pmb \xi \le \int_{(-R,R)^{k}}\psi(\diam(\pmb \xi,0)) d\pmb \xi \le 2^k\int_{(0,2R)^{k}}\psi(\diam(\pmb \xi,0)) d\pmb \xi.
    \]
    For $\pmb \xi \in (0,R)^k$, we have $\diam(\pmb \xi,0)=\max(\xi_1,...,\xi_k):=\max \pmb \xi$. Hence, \eqref{e3.8} will be proved once we show that
    \begin{equation} \label{e3.9}
    \int_{(0,R)^{k}}\psi(\max \pmb \xi) d\pmb \xi =k\int_0^R \psi(x)x^{k-1}dx.
    \end{equation}
    For $\pmb\xi=(\xi_1,...,\xi_{k-1},\xi_k)$, let $\pmb\xi_{k-1}=(\xi_1,...,\xi_{k-1})$. Applying Fubini's theorem, we have
    \begin{align*}
       &\int_{(0,R)^{k}}\psi(\max\pmb\xi) d\pmb \xi  =\int_{(0,R)^{k-1}}d\pmb \xi_{k-1} \int_0^R\psi(\max\pmb \xi) d\xi_k \\
       &=\int_{(0,R)^{k-1}}d\pmb \xi_{k-1}\left( \int_0^{\max\pmb \xi_{k-1}}\psi(\max\pmb \xi_{k-1}) d\xi_k+\int_{\max\pmb \xi_{k-1}}^R\psi(\xi_k) d\xi_k\right) \\
        &=\int_{(0,R)^{k-1}}d\pmb \xi_{k-1}\left( \psi(\max\pmb \xi_{k-1}) \max\pmb \xi_{k-1}+\int_{0}^R\psi(\xi_k)\pmb 1_{(\max\pmb \xi_{k-1},R)}(\xi_k) d\xi_k\right) \\
        &=\int_{(0,R)^{k-1}}\psi(\max\pmb \xi_{k-1}) \max\pmb \xi_{k-1}d\pmb \xi_{k-1}+\int_{0}^R\psi(\xi_k) \xi_k^{k-1}d\xi_k.
    \end{align*}
    By repeatedly applying this argument, we obtain
    \[
     \int_{(0,R)^{k}}\psi(\max\pmb\xi) d\pmb \xi=\sum_{j=1}^k \int_0^R \psi(\xi_j)\xi_j^{k-1}d\xi_j,
    \]
    which gives \eqref{e3.9}.
\end{proof}
\subsection{Proof of Theorem \ref{T1.20}} \label{S3.4} For each $j\in \{1,...,m\}$, let $\rho_{n,j}$ and $\rho_{n,j}^T$ denote the $j$-point and truncated $j$-point correlation functions for the real zeros of $Q_n$, respectively. According to the Kac-Rice formula,
\[
    \rho_{n,j}(\pmb \xi) =\frac{1}{(2\pi)^j\sqrt{\det \Lambda_n(\pmb \xi)}}\int_{\mathbb R^j}|y_1\cdots y_j|e^{-\frac 12\langle \pmb y (\Lambda_n(\pmb \xi))^{-1},\pmb y\rangle}dy_1\cdots dy_j,
    \]
where $\pmb \xi=(\xi_1,...,\xi_j)$, $\Lambda_n(\pmb \xi)$ is the covariance matrix of $(Q_n(\xi_1),Q_n'(\xi_1),...,Q_n(\xi_j), Q_n'(\xi_j))$, and $\pmb y=(0,y_1,...,0,y_j)$. Note that hypothesis \ref{H4} implies 
\[
(1-\varepsilon_n) \Lambda(\pmb \xi) \le \Lambda_n(\pmb \xi)\le (1+\varepsilon_n) \Lambda(\pmb \xi),
\]
where $\Lambda(\pmb \xi)$ is the covariance matrix of $(Q(\xi_1),Q'(\xi_1),...,Q(\xi_j), Q'(\xi_j))$. Analysis similar to that in the proof of Lemma \ref{L3.3} shows that
\[
\left(\frac{1-\varepsilon_n}{1+\varepsilon_n}\right)^j\rho_j(\pmb \xi) \le \rho_{n,j}(\pmb \xi) \le \left(\frac{1+\varepsilon_n}{1-\varepsilon_n}\right)^j \rho_j(\pmb \xi).
\]
Combining this with \eqref{e3.1} and the uniform boundedness of $\rho_j$,  we deduce  that
\[
|\rho_{n,j}^T(\pmb \xi)-\rho_j^T(\pmb \xi)| \ll \varepsilon_n,\quad j=1,...,m.
\]
Consequently, by employing Lemma \ref{L3.1} and the fact that $\rho_j^T$ is translation-invariant, we derive
\[
|s_m[N_{Q_n}(I_n)]-s_m[N_Q(|I_n|)]| \ll \varepsilon_n|I_n|^m. 
\]
According to Theorem \ref{T1.14},
\[
s_m[N_Q(|I_n|)]=|I_n|\theta_m^Q(|I_n|)+\lambda_m^Q(|I_n|),
\]
and hence \eqref{e1.27} follows. 

As $n\to \infty$, if $|I_n|\to \infty$ and $\varepsilon_n|I_n|^m\to 0$, then \eqref{e1.28} is a consequence of \eqref{e1.27} and \eqref{e1.23}.
\subsection{Proof of Theorem \ref{T1.24}} \label{S3.5}  For each $j\in \{1,...,m\}$, let $\widetilde\rho_{j}$ and $\widetilde \rho_{j}^T$ denote the $j$-point and truncated $j$-point correlation functions for the real zeros of $P(x)=Q(\varrho(x))$, respectively, where $\varrho(x)=\int_0^x\rho(t)dt$. Given $(x_1,...,x_m)\in \mathbb R^m$, $(\Delta x_1,...,\Delta x_m)\in \mathbb R^m$, and $j\in\{1,...,m\}$, let
\[
\xi_j=\varrho(x_j),\quad \Delta \xi_j=\varrho(x_j+\Delta x_j)-\varrho(x_j)=\int_{x_j}^{x_j+\Delta x_j}\rho(t)dt,
\]
and denote by $N_Q(\xi_j,\xi_j+\Delta\xi_j)$ the number of real zeros of $Q$ between $\xi_j$ and $\xi_j+\Delta \xi_j$. Note that $P$ has a real zero between $x_j$ and $x_j+\Delta x_j$ if and only if $Q$ has a real zero between $\xi_j$ and $\xi_j+\Delta \xi_j$. This implies  
\begin{align*}
    \widetilde\rho_{j}(x_1,...,x_j)&=\lim_{\Delta x_1,...,\Delta x_j\to 0} \frac{\mathbb E[N_P(x_1,x_1+\Delta x_1)\cdots N_P(x_j,x_j+\Delta x_j)]}{|\Delta x_1 \cdots \Delta x_j|}\\
    &=\lim_{\Delta x_1,...,\Delta x_j\to 0} \frac{\mathbb E[N_Q(\xi_1,\xi_1+\Delta \xi_1)\cdots N_Q(\xi_j,\xi_j+\Delta \xi_j)]}{|\Delta x_1 \cdots \Delta x_j|}\\
    &=\lim_{\Delta x_1,...,\Delta x_j\to 0} \frac{|\Delta \xi_1 \cdots \Delta \xi_j|}{|\Delta x_1 \cdots \Delta x_j|}\frac{\mathbb E[N_Q(\xi_1,\xi_1+\Delta \xi_1)\cdots N_Q(\xi_j,\xi_j+\Delta \xi_j)]}{|\Delta \xi_1 \cdots \Delta \xi_j|}\\
    &= \rho(x_1)\cdots \rho(x_j) \rho_j(\xi_1,...,\xi_j),
\end{align*}
where $\rho_j$ is the $j$-point correlation function of the real zeros of $Q$. Together with \eqref{e3.1}, we deduce that 
\[
\widetilde \rho_j^T(x_1,...,x_j)=\rho(x_1)\cdots \rho(x_j) \rho_j^T(\varrho(x_1),...,\varrho(x_j)),\quad j=1,...,m.
\]
Therefore, 
\begin{align*}
\int_{(a,b)^j}\widetilde \rho_j^T(x_1,...,x_j)dx_1\cdots dx_j &=\int_{(a,b)^j}\rho(x_1)\cdots \rho(x_j)\rho_j^T(\varrho(x_1),...,\varrho(x_j))dx_1\cdots dx_j \\
&=\int_{(\varrho(a),\varrho(b))^j}\rho_j^T(\xi_1,...,\xi_j)d\xi_1\cdots d\xi_j \\
&=\int_{(\varrho(a),\varrho(b))^j}\rho_j^T(\xi_1-\varrho(a),...,\xi_j-\varrho(a))d\xi_1\cdots d\xi_j \\
&=\int_{(0,R)^j}\rho_j^T(\xi_1,...,\xi_j)d\xi_1\cdots d\xi_j \\
&= R\hat \theta_j(R)+\hat\lambda_j(R),
\end{align*}
where $R:=\varrho(b)-\varrho(a)=\int_a^b\rho(t)dt$ and we used \eqref{e3.6} and the fact that $\rho^T_j$ is translation-invariant. According to Lemma \ref{L3.1}, we have 
\[
s_m[N_P(a,b)]=\sum_{\gamma \in \Pi(m)}[R\hat \theta_{|\gamma|}(R)+\hat\lambda_{|\gamma|}(R)]= R\theta_m^Q(R)+\lambda_m^Q(R),
\]
which establishes \eqref{e1.29}.
\subsection{Correlation functions of real zeros of elliptic polynomials} \label{S3.6}
We denote by $\rho_{n,k}$ and $\rho^T_{n,k}$ the $k$-point and truncated $k$-point correlation functions for the real zeros of the  Gaussian elliptic polynomial $P_n(x)$, respectively. As shown in \cite{BD}, for any $\pmb\xi=(\xi_1,...,\xi_k)$ of distinct points in $\mathbb R$, it holds that
\begin{equation}
\label{e3.10}
    \rho_{n,k}(\pmb\xi)=\prod_{j=1}^k\left(\frac{\sqrt n}{1+\xi_j^2}\right)\int_{\mathbb R^k}|y_1\cdots y_k|D_{n,k}(\pmb y;\pmb \xi)dy_1\cdots dy_k,
\end{equation}
where $\pmb y=(0,y_1,...,0,y_k)\in \mathbb R^{2k}$ and $D_{n,k}(\cdot;\pmb\xi)$ is a Gaussian density with the covariance matrix
\[\Sigma_n(\pmb \xi)=\left(\Sigma_{ij}^{(n)}\right)_{i,j=1}^k,\]
in which 
\begin{equation}\label{e3.11}
    \Sigma_{ij}^{(n)}=(1+\alpha^2(\xi_i,\xi_j))^{-n/2}\begin{pmatrix}
    1&-\sqrt n \alpha(\xi_i,\xi_j)\\
    \sqrt n \alpha(\xi_i,\xi_j)&1+(1-n)\alpha^2(\xi_i,\xi_j)
    \end{pmatrix}.
\end{equation}
In particular, 
\[\rho_{n,1}^T(\xi_1)=\rho_{n,1}(\xi_1)=\frac{\sqrt n}{\pi(1+\xi_1^2)}.\]
For $k\ge 2$, to find a scaling limit of $\rho_{n,k}^T$, let us make the change of variables
\begin{equation}
\label{e3.12}
    t_j=\sqrt n \alpha(\xi_1,\xi_j)=\sqrt{n}\frac{\xi_j-\xi_1}{1+\xi_1\xi_j},\quad j=2,...,k,
\end{equation}
so that
\begin{equation}
    \label{e3.13}
    \alpha(\xi_i,\xi_j)=\alpha(t_i/\sqrt n,t_j/\sqrt n),\quad i,j=2,...,k.
\end{equation}
For $\pmb t=(t_1,...,t_k)$, let $\pmb t_0=(0,t_2,...,t_k)$. Then the integral $\int_{\mathbb R^k}|y_1\cdots y_k|D_{n,k}(\pmb y; \pmb \xi)dy_1\cdots dy_k$ appeared in \eqref{e3.10} can be interpreted as a function of $\pmb t_0$. More precisely, by letting $t_1=0$, we deduce from \eqref{e3.11} and \eqref{e3.13} that 
\begin{align*}
\int_{\mathbb R^k}|y_1\cdots y_k|D_{n,k}(\pmb y;\pmb\xi)dy_1\cdots dy_k=\int_{\mathbb R^k}|y_1\cdots y_k|d_{n,k}(\pmb y;\pmb t_0)dy_1\cdots dy_k,
\end{align*}
in which $d_{n,k}(\cdot;\pmb t)$ is a  Gaussian density  with the covariance matrix
\[\Omega_n(\pmb t)=\left(\Omega_{ij}^{(n)}\right)_{i,j=1}^k,\]
where
\begin{equation}
    \label{e3.14}
    \Omega_{ij}^{(n)}=\left(1+\alpha^2\left(\frac{t_i}{\sqrt n},\frac{t_j}{\sqrt n}\right)\right)^{-n/2} \begin{pmatrix}
    1&-\sqrt n \alpha\left(\frac{t_i}{\sqrt n},\frac{t_j}{\sqrt n}\right)\\
    \sqrt n \alpha\left(\frac{t_i}{\sqrt n},\frac{t_j}{\sqrt n}\right)&1+(1-n)\alpha^2\left(\frac{t_i}{\sqrt n},\frac{t_j}{\sqrt n}\right)
    \end{pmatrix}.
\end{equation}
For $\gamma_j\subset \{1,...,k\}$ with $l_j=|\gamma_j|\ge 1$, let us introduce the $l_j$-point functions
\[
    \Theta_{n,l_j}(\pmb t_{\gamma_j})=\int_{\mathbb R^{l_j}}|y_1\cdots y_{l_j}|d_{n,l_j}(\pmb y_{\gamma_j};\pmb t_{\gamma_j})dy_1\cdots dy_{l_j},
\]
where $\pmb y_{\gamma_j}=(0,y_1,...,0,y_{l_j})$ and $\pmb t_{\gamma_j}=(t_i)_{i\in \gamma_j}$.  We also consider the corresponding truncated functions
\[
  \Theta^T_{n,k}(\pmb t)=\sum_{j=1}^k(-1)^{j-1}(j-1)!\sum_{\gamma\in \Pi(k,j)}\Theta_{n,l_1}(\pmb t_{\gamma_1})\cdots \Theta_{n,l_j}(\pmb t_{\gamma_j}).  
\]
Put this way, one has 
\begin{equation}
\label{e3.15}
    \rho_{n,k}^T(\pmb \xi) =\left(\prod_{j=1}^k\frac{\sqrt n}{1+\xi_j^2}\right)\Theta^T_{n,k}(\pmb t_0).
\end{equation}
\begin{lemma} \label{L3.8}
For $k\ge 1$, let $\Theta_k^T$ denote the truncated $k$-point correlation function for the real zeros of the Weyl series $W$. As $n\to \infty$, $\Theta_{n,k}^T(\pmb t)$ converges pointwise to $\Theta_k^T(\pmb t)$ on $\mathbb R^k$. Furthermore, if $0<\alpha_n<\sqrt n$, it holds uniformly that 
\begin{equation} \label{e3.16}
    \Theta^T_{n,k}(\pmb t) =\Theta_k^T(\pmb t)+O(\alpha_n^2/n),\quad \pmb t\in [-\alpha_n,\alpha_n]^k.
\end{equation}
\end{lemma}
\begin{proof}
    For any $\pmb t=(t_1,...,t_k)\in \mathbb R^k$, it holds that 
    \[
    \lim_{n\to \infty}\sqrt n \alpha\left(\frac{t_i}{\sqrt n},\frac{t_j}{\sqrt n}\right) =t_j-t_i,\quad 1\le i,j\le k.
    \]
    Combining this with \eqref{e3.14} leads to
\begin{equation} \label{e3.17}
    \lim_{n\to \infty}\Omega_n(\pmb t) =\Omega(\pmb t),
\end{equation}
where
\begin{equation}
    \label{e3.18}
    \Omega(\pmb t)=(\Omega_{ij})_{i,j=1}^k\quad \text{with}\quad \Omega_{ij}=e^{-(t_j-t_i)^2/2}\begin{pmatrix}
    1&-(t_j-t_i)\\
    t_j-t_i&1-(t_j-t_i)^2
    \end{pmatrix}.
\end{equation}
Let $\Theta_k$ denote the $k$-point correlation function for the real zeros of the Weyl series $W$. Since $\widetilde W(t)=e^{-t^2/2}W(t)$ and $W(t)$ have the same real zeros, it follows that $\Theta_k$ is also the $k$-point correlation function for the real zeros of $\widetilde W$. Applying the Kac-Rice formula and \eqref{e3.18}, we verify that
\[
\Theta_k(\pmb t)=\int_{\mathbb R^k}|y_1\cdots y_k|d_{k}(\pmb y;\pmb t)dy_1\cdots dy_k,
\]
where $d_{k}(\cdot;\pmb t)$ is a Gaussian density with the covariance matrix $\Omega(\pmb t)$. Therefore, \eqref{e3.17} implies
\[
\lim_{n\to \infty}\Theta_{n,k}(\pmb t)= \Theta_k(\pmb t),\quad \pmb t\in \mathbb R^k.
\]
Together with \eqref{e3.1}, we deduce that 
\[
\lim_{n\to \infty}\Theta_{n,k}^T(\pmb t)= \Theta_k^T(\pmb t),\quad \pmb t\in \mathbb R^k,
\]
as required.

For $\pmb t=(t_1,...,t_k)\in [-\alpha_n,\alpha_n]^k$, since
\[
\sqrt n \alpha\left(\frac{t_i}{\sqrt n},\frac{t_j}{\sqrt n}\right) =(t_j-t_i)(1+O(\alpha_n^2/n)),\quad 1\le i,j\le k,
\]
it follows that 
\[
    \Omega_n(\pmb t) =\Omega(\pmb t) (1+O(\alpha_n^2/n)),
\]
and consequently,
\[
\Theta_{n,k}(\pmb t)= \Theta_k(\pmb t)(1+O(\alpha_n^2/n)), \quad \pmb t \in [-\alpha_n,\alpha_n]^k.
\]
By \eqref{e3.1} and the uniform boundedness of $\Theta_j$ on $\mathbb R^j$, for $j=1,...,k$,  we establish \eqref{e3.16}.
\end{proof}
Henceforth, for $\pmb t=(t_1,t_2,...,t_k)\in \mathbb R^{k}$, define
\begin{equation}
\label{e3.19}
    \widetilde \Theta_{n,k}^T(\pmb t)=\frac{\Theta_{n,k}^T(\pmb t)}{\prod_{j=1}^k(1+t_j^2/n)}.
\end{equation}
\begin{lemma} \label{L3.9} 
For $\alpha_n>0$, let $U_n=(-\alpha_n,\alpha_n)^{k-1}$ and $A_n=\min(\alpha_n,\sqrt{n})$. As $n\to \infty$, if $\alpha_n\to \infty$, then 
    \begin{equation}
        \label{e3.20} \int_{U_n}\widetilde \Theta_{n,k}^T(\pmb t_0)d\pmb t_0 =\int_{\mathbb R^{k-1}}\Theta_{k}^T(\pmb t_0)d\pmb t_0+o(1/A_n),
    \end{equation}
    where $\pmb t_0=(0,t_2,...,t_k)$ and $d\pmb t_0=dt_2\cdots dt_k$.
\end{lemma}
\begin{proof} 
Let $U_n^c=\mathbb R^{k-1}\backslash U_n$.  By Lemmas \ref{L3.6}, \ref{L3.7}, and \ref{L3.8}, there exist positive constants $c_k$ and $C_k$ such that 
 \begin{align*}
     \int_{U_n^c}|\widetilde \Theta_{n,k}^T(\pmb t_0)|d\pmb t_0 &\le C_k\int_{U_n^c}e^{-c_k(\diam(\pmb t_0))^2}d\pmb t_0\\
     &\ll \int_{\alpha_n}^\infty e^{-c_kx^2}x^{k-2}dx\\
     &=O(\alpha_n^{k-3}e^{-c_k\alpha_n^2}).
 \end{align*}
 
 Assume first that $\alpha_n \le \log n$, so $\alpha_n^{k+1}/n=o(1/A_n)$. By \eqref{e3.16},  
\begin{align*}
   \int_{U_n}\widetilde \Theta_{n,k}^T(\pmb t_0)d\pmb t_0 &=\int_{U_n}\Theta_{k}^T(\pmb t_0)d\pmb t_0+O(\alpha_n^{k+1}/n)\\
   &= \int_{\mathbb R^{k-1}}\Theta_{k}^T(\pmb t_0)d\pmb t_0 +O(\alpha_n^{k-3}e^{-c_k\alpha_n^2})+O(\alpha_n^{k+1}/n),
\end{align*}
which implies \eqref{e3.20}. 

Now, let us suppose $\alpha_n>\log n$. Define $V_n=(-\log n, \log n)^{k-1}$ and $V_n'=U_n\backslash V_n$. Utilizing the aforementioned estimates, we obtain
\begin{align*}
   \int_{U_n}\widetilde \Theta_{n,k}^T(\pmb t_0)d\pmb t_0 &=\int_{V_n}\widetilde \Theta_{n,k}^T(\pmb t_0)d\pmb t_0 +\int_{V_n'}\widetilde \Theta_{n,k}^T(\pmb t_0)d\pmb t_0\\
   &= \int_{\mathbb R^{k-1}}\Theta_{k}^T(\pmb t_0)d\pmb t_0+O((\log n)^{k+1}/n) + O((\log n)^{k-3}e^{-c_k\log^2n}),
\end{align*}
establishing \eqref{e3.20}.
\end{proof}
\begin{remark}\label{R3.10}
   Similarly, for any polynomial $p(\mathbf{t}_0)$, we have
    \[
    \int_{\mathbb R^{k-1}}|\widetilde \Theta_{n,k}^T(\pmb t_0)p(\pmb t_0)|d\pmb t_0=O(1).
    \]
\end{remark}
\subsection{Proof of Theorem \ref{T1.3}} \label{S3.7}
Our proof commences by recalling the relationship between $s_k[N_n(a,b)]$ and truncated correlation functions $\rho_{n,j}^T$, where $j\in \{1,...,k\}$. According to Lemma \ref{L3.1}, we have  
\begin{equation}
\label{e3.21}
    s_k[N_n(a,b)]=\sum_{\gamma\in \Pi(k)}\int_{(a,b)^{|\gamma|}}\rho^T_{n,|\gamma|}(\pmb \xi_\gamma)d\pmb\xi_\gamma,
\end{equation}
where $|\gamma|$ is the number of blocks in the partition $\gamma$ and $d\pmb\xi_\gamma$ is the Lebesgue measure on $(a,b)^{|\gamma|}$. The subsequent step involves the estimation of the integrals present on the right-hand side of \eqref{e3.21}.
\begin{lemma} \label{L3.11}
For $k\ge 1$, we have, as $n\to \infty$,
\begin{equation}\label{e3.22}
    \int_{(a,b)^k}\rho_{n,k}^T(\pmb \xi)d\pmb \xi =\theta_k \mathbb E[N_n(a,b)]+O(1),
\end{equation}
where $\theta_1=1$ and for $k\ge 2$, 
\[
    \theta_k=\pi \int_{\mathbb R^{k-1}}\Theta^T_k(\pmb t_0)d\pmb t_0.
\]
\end{lemma}
Note that Theorem \ref{T1.3} immediately follows from applying Lemma \ref{L3.11} and \eqref{e3.21}. More precisely, by substituting \eqref{e3.22} into \eqref{e3.21}, we obtain \eqref{e1.16}, where
\[
\beta_k:=\sum_{\gamma\in \Pi(k)}\theta_{|\gamma|}=\pi \theta_{k,\infty}^W.
\]

For $k=1$, Lemma \ref{L3.11} is trivial. Assume now that $k\ge 2$. Making the change of variables \eqref{e3.12},  we see that 
\[
    \frac{\sqrt n}{1+\xi_j^2}d\xi_j=\frac{1}{1+t_j^2/n}dt_j,\quad j=2,...,k.
\]
Combining this with \eqref{e3.15}, we derive
\begin{equation}
 \label{e3.23}
    \rho_{n,k}^T(\pmb \xi)d\pmb \xi =\frac{\sqrt n}{1+\xi_1^2}  \widetilde\Theta^T_{n,k}(\pmb t_0)  d\xi_1d\pmb t_0,
\end{equation}
where $\widetilde\Theta^T_{n,k}$ is given by \eqref{e3.19}, $\pmb t_0=(0,t_2,...,t_k)$, and $d\pmb t_0=dt_2\cdots dt_k$.

To prove Lemma \ref{L3.11}, we consider three cases of $\alpha_n=\sqrt n\alpha(a,b)$. For brevity, let $U_n=(-|\alpha_n|,|\alpha_n|)^{k-1}$ and $I_{n,k}(a,b)=\int_{(a,b)^k}\rho_{n,k}^T(\pmb \xi)d\pmb \xi$.
\begin{claim}\label{cl1}
If $\alpha_n>0$, then \eqref{e3.22} holds. 
\end{claim}
\begin{proof}
Using \eqref{e3.23}, we obtain
\begin{align*}
I_{n,k}(a,b)&=\int_a^b\frac{\sqrt n }{1+\xi_1^2}d\xi_1 \int_{U_n(\xi_1)}\widetilde\Theta^T_{n,k}(\pmb t_0) d\pmb t_0,
\end{align*}
where 
\[
U_n(\xi_1):=\left\{(t_2,...,t_k)\in (a,b)^{k-1}: \sqrt n \alpha(\xi_1,a)<t_2,...,t_k<\sqrt n\alpha(\xi_1,b) \right\}.
\]
By Fubini's theorem, 
\begin{equation*}
I_{n,k}(a,b)=\int_{U_n}\widetilde\Theta^T_{n,k}(\pmb t_0) d\pmb t_0 \int_a^b \sqrt n\frac{G(\xi_1,t_2,...,t_k)}{1+\xi_1^2}d\xi_1,
\end{equation*}
where
\begin{equation*}
G(\xi_1,t_2,...,t_k):=\prod_{j=2}^k\bigg(\pmb 1_{(-\alpha_n,0)}(t_j)\pmb 1_{(\alpha(t_j/\sqrt n,a),b)}(\xi_1)+\pmb 1_{(0,\alpha_n)}(t_j)\pmb 1_{(a,\alpha(t_j/\sqrt n,b))}(\xi_1)\bigg).
\end{equation*}

 For $k\ge 2$, let $\Pi_2(k)$ denote the set of all ordered pair $(\gamma_1,\gamma_2)$ of disjoint subsets of $ \{2,...,k\}$ such that  $\gamma_1\cup\gamma_2=\{2,...,k\}$. For each  $\gamma=(\gamma_1,\gamma_2)\in \Pi_2(k)$, we introduce the function
\begin{equation*}
G_{\gamma}(\xi_1,t_2,...,t_k):=\prod_{j\in \gamma_1}\pmb 1_{(-\alpha_n,0)}(t_j)\pmb 1_{(\alpha(t_j/\sqrt n,a),b)}(\xi_1) \prod_{i\in \gamma_2}\pmb 1_{(0,\alpha_n)}(t_i)\pmb 1_{(a,\alpha(t_i/\sqrt n,b))}(\xi_1)
\end{equation*}
so that
\begin{equation*}
I_{n,k}(a,b)=\int_{U_n}\widetilde\Theta^T_{n,k}(\pmb t_0) d\pmb t_0 \sum_{\gamma\in \Pi_2(k)}\int_a^b\sqrt n \frac{G_\gamma(\xi_1,t_2,...,t_k)}{1+\xi_1^2}d\xi_1.
\end{equation*}
For each $\gamma=(\gamma_1,\gamma_2)\in \Pi_2(k)$ and $(t_2,...,t_k)\in \mathbb R^{k-1}$, let
\[
t_{\gamma_1}^{\min}=\begin{cases}
0&\text{if}\quad \gamma_1=\varnothing,\\
\min_{j\in \gamma_1}t_j&\text{if}\quad \gamma_1\ne \varnothing,
\end{cases} \quad \text{and}\quad t_{\gamma_2}^{\max}=\begin{cases}
0&\text{if}\quad \gamma_2=\varnothing,\\
\max_{i\in \gamma_2}t_i&\text{if}\quad \gamma_2\ne \varnothing.
\end{cases}
\]
By a direct computation, we obtain
\begin{align*}
&\int_a^b \sqrt n\frac{G_{\gamma}(\xi_1,t_2,...,t_k)}{1+\xi_1^2}d\xi_1\\
&\quad = \pi \prod_{j\in \gamma_1}\pmb 1_{(-\alpha_n,0)}(t_j)\prod_{i\in \gamma_2}\pmb 1_{(0,\alpha_n)}(t_i)  \mathbb E[N_n(a,b)]\\
&\quad +\prod_{j\in \gamma_1}\pmb 1_{(-\alpha_n,0)}(t_j)\prod_{i\in \gamma_2}\pmb 1_{(0,\alpha_n)}(t_i)\sqrt n\left(\arctan\frac{t_{\gamma_1}^{\min}}{\sqrt n}-\arctan \frac{t_{\gamma_2}^{\max}}{\sqrt n}\right).
\end{align*}
For any fixed $(t_2,...,t_k)\in \mathbb R^{k-1}$, we have
\[
\lim_{n\to \infty} \sqrt n\left(\arctan\frac{t_{\gamma_1}^{\min}}{\sqrt n}-\arctan \frac{t_{\gamma_2}^{\max}}{\sqrt n}\right) =t_{\gamma_1}^{\min}-t_{\gamma_2}^{\max}.
\]
Therefore, using Remark \ref{R3.10}, we deduce that, as $n\to \infty$,
\[
I_{n,k}(a,b)=\left(\pi \int_{U_n}\widetilde\Theta^T_{n,k}(\pmb t_0) d\pmb t_0\right) \mathbb E[N_n(a,b)]+O(1).
\]
Since $\alpha_n>0$, we have 
\[
\mathbb E[N_n(a,b)]=\frac{\sqrt n}{\pi}\arctan\frac{\alpha_n}{\sqrt n} \le \min\left(\frac{\alpha_n}{\pi},\frac{\sqrt n}{2}\right) \ll A_n.
\]
Combining this with Lemma \ref{L3.9}, we deduce that
\[
I_{n,k}(a,b)=\left(\pi \int_{\mathbb R^{k-1}}\Theta_k^T(\pmb t_0)dt_2\cdots dt_k+o(1/A_n)\right) \mathbb E[N_n(a,b)]+O(1),
\]
and \eqref{e3.22} follows.
\end{proof}
\begin{claim} The asymptotic formula \eqref{e3.22} holds for  $\alpha_n=0$. 
\end{claim}
\begin{proof}
If $\alpha_n=0$, then $(a,b)=\mathbb R$. Using \eqref{e3.23} and Lemma \ref{L3.9}, we have 
\begin{align*}
     I_{n,k}(\mathbb R)&= \int_{\mathbb R}\frac{\sqrt n }{1+\xi_1^2}d\xi_1 \int_{\mathbb R^{k-1}}\widetilde\Theta^T_{n,k}(\pmb t_0) d\pmb t_0\\
     &= \left(\pi \int_{\mathbb R^{k-1}}\widetilde\Theta^T_{n,k}(\pmb t_0) d\pmb t_0\right) \mathbb E[N_n(\mathbb R)]\\
     &=\left(\pi \int_{\mathbb R^{k-1}}\Theta_k^T(\pmb t_0)d\pmb t_0+o(1/\sqrt n)\right) \mathbb E[N_n(\mathbb R)]\\
     &=\left(\pi \int_{\mathbb R^{k-1}}\Theta_k^T(\pmb t_0)d\pmb t_0\right) \mathbb E[N_n(\mathbb R)]+o(1),
     \end{align*}
which yields \eqref{e3.22}.
\end{proof}
\begin{claim} If $\alpha_n<0$, then \eqref{e3.22} occurs.
\end{claim} 
\begin{proof}
We first write 
\begin{align*}
   I_{n,k}(a,b)&=\int_a^{-1/b}d\xi_1\int_{(a,b)^{k-1}}\rho_{n,k}^T(\pmb\xi)d\xi_2 \cdots d\xi_k\\
   &+\int_{-1/b}^{-1/a}d\xi_1\int_{(a,b)^{k-1}}\rho_{n,k}^T(\pmb \xi)d\xi_2 \cdots d\xi_k+\int_{-1/a}^bd\xi_1\int_{(a,b)^{k-1}}\rho_{n,k}^T(\pmb \xi)d\xi_2 \cdots d\xi_k.
\end{align*}
By \eqref{e3.23} and Fubini's theorem,  
\[
I_{n,k}(a,b)=\int_{\mathbb R^{k-1}}\widetilde\Theta^T_{n,k}(\pmb t_0) d\pmb t_0 \int_a^b\sqrt n \frac{H(\xi_1,t_2,...,t_k)}{1+\xi_1^2}d\xi_1,
\]
where
\begin{align*}
    H(\xi_1,t_2,...,t_k)&=\pmb 1_{(a,-1/b)}(\xi_1) \prod_{j=2}^{k}\bigg(\pmb 1_{(-\infty,\alpha_n)}(t_j)\pmb 1_{(a,\alpha(t_j/\sqrt n,b))}(\xi_1)\\
    &\qquad\qquad +\pmb 1_{(n/\alpha_n,0)}(t_j)\pmb 1_{(\alpha(t_j/\sqrt n,a),-1/b)}(\xi_1)+\pmb 1_{(0,\infty)}(t_j)\bigg)\\
    &+\pmb 1_{(-1/b,-1/a)}(\xi_1) \prod_{j=2}^{k}\bigg(\pmb 1_{(-\infty,n/\alpha_n)}(t_j)\pmb 1_{(\alpha(t_j/\sqrt n,a),-1/a)}(\xi_1)\\
    &\qquad\qquad+\pmb 1_{(n/\alpha_n,-n/\alpha_n)}(t_j)+\pmb 1_{(-n/\alpha_n,\infty)}(t_j)\pmb 1_{(-1/b,\alpha(t_j/\sqrt n,b))}(\xi_1)\bigg)\\
    &+\pmb 1_{(-1/a,b)}(\xi_1) \prod_{j=2}^{k}\bigg(\pmb 1_{(-\infty,0)}(t_j)+\pmb 1_{(0,-n/\alpha_n)}(t_j)\pmb 1_{(-1/a,\alpha(t_j/\sqrt n,b))}(\xi_1)\\
    &\qquad\qquad+\pmb 1_{(-\alpha_n,\infty)}(t_j)\pmb 1_{(\alpha(t_j/\sqrt n,a),b)}(\xi_1)\bigg). 
\end{align*}
We may now estimate $I_{n,k}(a,b)$ using the same method as in the proof of Claim \ref{cl1} to conclude that
\[
I_{n,k}(a,b)=\left(\pi \int_{\mathbb R^{k-1}}\widetilde\Theta^T_{n,k}(\pmb t_0) d\pmb t_0\right) \mathbb E[N_n(a,b)] +O(1),
\]
which establishes the asymptotic formula \eqref{e3.22} when combined with Lemma \ref{L3.9}.
\end{proof}
\section{Asymptotic normality} \label{S4} In this section, we briefly discuss the CLTs for the number of real zeros of Gaussian processes. Drawing on the method of moments (see, for example, \cite{Bi}*{\S30}), we deduce that if $\{X_n\}$ is a sequence of random variables such that, as $n\to \infty$,
\[
s_1[X_n]\to 0,\quad s_2[X_n]\to 1,\quad \text{and}\quad s_k[X_n]\to 0 \quad (k\ge 3),
\]
then $X_n\xrightarrow{d} \mathcal N(0,1)$. Utilizing this insight for the normalized number of real zeros allows us to deduce the CLTs outlined in Sections \ref{S1.2} and \ref{S1.3}.  Since the proofs are analogous, we only present the proof of Theorem \ref{T1.5}.
\begin{proof}[Proof of Theorem \ref{T1.5}]
    Let 
\[
X_n:=\frac{N_n(a,b)-\mathbb E[N_n(a,b)]}{\sqrt{\Var[N_n(a,b)]}},\quad n\ge 1.
\]
We show that as $n\to \infty$, $X_n\xrightarrow{d} \mathcal N(0,1)$, provided that either $\alpha_n\le 0$ or $\alpha_n\to \infty$ as $n\to \infty$. 

Notably,  $s_1[X_n]=\mathbb E[X_n]=0$ and $s_2[X_n]=\Var[X_n]=1$. To complete the proof, we show that for $k\ge 3$, $s_k[X_n]\to 0$ as $n\to \infty$. Utilizing the semi-invariance property, one has
\[
s_k[X_n]=\frac{s_k[N_n(a,b)]}{(\Var[N_n(a,b)])^{k/2}}.
\]
If either $\alpha_n\le 0$ or $\alpha_n\to \infty$ as $n\to \infty$, then $\mathbb E[N_n(a,b)]\to \infty$. By Theorem \ref{T1.3}, 
\[
s_k[X_n]=\frac{\beta_k\mathbb E[N_n(a,b)]+O(1)}{(\beta_2\mathbb E[N_n(a,b)]+O(1))^{k/2}} \to 0 \quad \text{as}\quad n\to \infty,
\]
and Theorem \ref{T1.5} is proved. 
\end{proof}
\section{Asymptotics of moments and strong law of large numbers}
In this section, we use the asymptotics of cumulants to deduce corresponding results for central moments and moments, ultimately establishing the strong law of large numbers.
\subsection{Asymptotics of moments} \label{S5.1} It is well-known that we can express the $k$th central moment in terms of the first $k$ cumulants. Here, for the reader's convenience, we briefly explain how to obtain this explicit expression using Fa\`{a} di Bruno’s formula (see \cite{C}*{\S3.4}) and the exponential partial Bell polynomials (see \cite{C}*{\S3.3}), both of which are tools utilized in Section \ref{S2.3}. More precisely, consider the cumulant- and central moment-generating functions of a bounded random variable $X$ given by
\[
K(t)=\log \mathbb E[e^{tX}] \quad \text{and}\quad C(t)=\mathbb E\big[e^{t(X-\mathbb E[X])}\big],
\]
respectively. Then 
\[
s_k[X]=\frac{d^k}{dt^k}K(t)\bigg|_{t=0} \quad \text{and}\quad \mu_k[X]=\frac{d^k}{dt^k}C(t)\bigg|_{t=0}.
\]
Since $C(t)=e^{K(t)-t\mathbb E[X]}$, it follows from Fa\`{a} di Bruno’s formula that
\begin{align*}
    \mu_k[X]&=\frac{d^k}{dt^k}e^{K(t)-t\E[X]}\bigg|_{t=0}=\sum_{j=1}^kB_{k,j}(0,s_2[X],...,s_{k-j+1}[X]).
\end{align*}
Recall that, for $1\le j\le k$,
\begin{align*}
    &B_{k,j}(x_1,...,x_{k-j+1})=\sum\frac{k!}{m_1!\cdots m_{k-j+1}!}\prod_{r=1}^{k-j+1}\left(\frac{x_r}{r!}\right)^{m_r},
\end{align*}
where the sum is over all solutions in non-negative integers of the equations
\begin{align*}
    m_1+2m_2+\cdots+(k-j+1)m_{k-j+1}&=k,\\
    m_1+m_2+\cdots+m_{k-j+1}&=j.
\end{align*}
Note that $m_1\ge 1$ whenever $j>k/2$, so $B_{k,j}(0,x_2,...,x_{k-j+1})=0$ for all $j>k/2$. Therefore, for $k\ge 2$,
\begin{equation}
    \label{e5.1}
    \mu_k[X]=\sum_{j=1}^{\lfloor k/2 \rfloor}B_{k,j}(0,s_2[X],...,s_{k-j+1}[X]).
\end{equation}
When $X$ represents the number of real zeros of Gaussian processes considered in this paper, it is safe to employ \eqref{e5.1} since all mentioned quantities are finite.
\begin{proof}[Proof of Corollary \ref{C1.7}]
By \eqref{e5.1}, it holds for $k\ge 2$ that
\[
\mu_k[N_n(a,b)]=\sum_{j=1}^{\lfloor k/2 \rfloor}B_{k,j}(0,s_2[N_n(a,b)],...,s_{k-j+1}[N_n(a,b)]).
\]
Together with \eqref{e1.16}, we obtain
\begin{align*}
\mu_{2k}[N_n(a,b)]&=B_{2k,k}(0,s_2[N_n(a,b)],...,s_{k+1}[N_n(a,b)]) +O((\E[N_n(a,b)])^{k-1})\\
&=\frac{(2k)!}{k!}\left(\frac{s_2[N_n(a,b)]}{2!}\right)^k+O((\E[N_n(a,b)])^{k-1}),
\end{align*}
which yields \eqref{e1.17}. Similarly, 
\begin{align*}
\mu_{2k+1}[N_n(a,b)]&=B_{2k+1,k}(0,s_2[N_n(a,b)],...,s_{k+2}[N_n(a,b)]) +O((\E[N_n(a,b)])^{k-1})\\
&=\frac{(2k+1)!}{(k-1)!}\left(\frac{s_2[N_n(a,b)]}{2!}\right)^{k-1}\frac{s_3[N_n(a,b)]}{3!} +O((\E[N_n(a,b)])^{k-1}),
\end{align*}
which implies \eqref{e1.18}.
\end{proof}
\begin{remark}\label{r.asym.central.moments}
    Utilizing \eqref{e5.1} and Theorem \ref{T1.14}, we can derive a precise expression for the central moment $\mu_k[N_Q(R)]$ whenever $Q\in \mathcal A_k$. Consequently, if $Q\in \mathcal A_\infty$, then from Remark \ref{r.asym.cumulants}, it follows that for each positive integer $k\ge 2$, as $R\to \infty$, $\mu_k[N_Q(R)]$ admits a full asymptotic expansion of the form
    \[
    \mu_k[N_Q(R)]\sim \sum_{j=1}^{\lfloor k/2 \rfloor}B_{k,j}(0,R\theta_{2,\infty}^Q+\lambda_{2,\infty}^Q,...,R\theta_{k-j+1,\infty}^Q+\lambda_{k-j+1,\infty}^Q).
    \]
    When $k$ is even, the right-hand side is a polynomial of $R$ of degree $k/2$. When $k$ is odd, it is a polynomial of $R$ of degree at most $(k-1)/2$. 
    
    Similar conclusions apply to $N_W(R)$, $N_{Q_n}(I_n)$, and $N_{W_n}(I_n)$.
    
    We emphasize that in \cite{AL1}*{Theorem 1.6}, Ancona and Letendre investigated a broader context; nonetheless, our estimates prove more precise in specific instances.
\end{remark}
\begin{remark}\label{r.asym.moments}
    We also have an explicit expression for the $k$th moment $\mathbb E[(N_Q(R))^k]$ in terms of the first $k$ cumulants $s_j[N_Q(R)]$ for $1\le j\le k$, as follows  (see, for example, \cite{NS}*{Equation 26}):
    \[
    \mathbb E[(N_Q(R))^k]=\sum_{j=1}^k\sum_{\gamma \in \Pi(k,j)}s_{l_1}[N_Q(R)]\cdots s_{l_j}[N_Q(R)].
    \]
    Combining this with Theorem \ref{T1.14}, we obtain an exact formula for $\mathbb E[(N_Q(R))^k]$ whenever $Q\in \mathcal A_k$. Furthermore, Remark \ref{r.asym.cumulants} implies that if $Q\in A_\infty$, then, as $R\to \infty$, 
    \[
      \mathbb E[(N_Q(R))^k] \sim \sum_{j=1}^k\sum_{\gamma \in \Pi(k,j)}(R\theta_{l_1,\infty}^Q+\lambda_{l_1,\infty}^Q)\cdots (R\theta_{l_j,\infty}^Q+\lambda_{l_j,\infty}^Q),
    \]
     where the right-hand side is a polynomial of $R$ of degree $k$. Note that this full asymptotic expansion specifically applies to the Gaussian Weyl series $W$.
    
    In \cite{DV}*{Theorem 6}, Do and Vu demonstrated that
     \[\mathbb E[(N_W^\phi(R))^k]\le C_{\phi,k} R^k.\]
    Choosing $\phi=\pmb 1_{[0,R]}$ implies 
    \[\mathbb E[(N_W(R))^k]\le C_k R^k.\]
    In this particular case, our estimate is more precise.
\end{remark}
\subsection{Strong law of large numbers} \label{S5.2} In this subsection, we aim to establish a strong law of large numbers for the number of real zeros using asymptotics of central moments, reinforced by a Borel-Cantelli type argument. Only the proof for Theorem \ref{T1.9} is provided here, as the proofs for other results follow a similar approach.
\begin{proof}[Proof of Theorem \ref{T1.9}]
The given assumption implies
\[
\sum_{n=1}^\infty \frac{1}{(\mathbb E[N_n(a,b)])^k}<\infty.
\]
By Corollary \ref{C1.7}, this leads to
\begin{align*}
    \mathbb E\left[\sum_{n=1}^\infty\left(\frac{N_n(a,b)}{\mathbb E[N_n(a,b)]}-1\right)^{2k}\right]&=\sum_{n=1}^\infty \frac{\mu_{2k}[N_n(a,b)]}{(\mathbb E[N_n(a,b)])^{2k}}=O\left(\sum_{n=1}^\infty \frac{1}{(\mathbb E[N_n(a,b)])^k}\right)<\infty.
\end{align*}
Therefore, almost surely,
\[
\sum_{n=1}^\infty\left(\frac{N_n(a,b)}{\mathbb E[N_n(a,b)]}-1\right)^{2k}<\infty,
\]
yielding 
\[
\frac{N_n(a,b)}{\E[N_n(a,b)]}-1 \xrightarrow{a.s.} 0 \quad \text{as } n\to \infty,
\]
and the theorem is proved.
\end{proof}
\acks The author extends deep gratitude to his esteemed PhD advisor, Yen Do, for invaluable guidance and suggestions throughout the preparation of this paper. Appreciation is also expressed to the anonymous reviewers whose feedback significantly enhanced the paper's quality.

\end{document}